\documentclass[a4paper,12pt]{article}
\usepackage[utf8]{inputenc}
\usepackage{amsmath}
\usepackage{amssymb}
\usepackage{amsthm}
\usepackage{xcolor}

\usepackage{enumerate}
\usepackage{amssymb, amsmath,graphicx}
\usepackage{todonotes}
\usepackage{verbatim}
\usepackage{float}

\theoremstyle{plain}
\newtheorem{thm}{Theorem}[section]
\newtheorem{Prop}[thm]{Proposition}
\newtheorem{Def}[thm]{Definition}
\newtheorem{lma}[thm]{Lemma}
\newtheorem{cor}[thm]{Corollary}

\newtheorem{remark}[thm]{Remark}
\newtheorem{ex}[thm]{Example}

\newtheorem*{example*}{Example}
\newtheorem*{remark*}{Remark}

\newcommand{\Ent}{\mathrm{Ent}}

\newcommand{\Dom}{\mathcal{D}}
\def\D{{\mathcal D}}

\newcommand{\DP}{{P}^*}

\newcommand{\Lip}{\mathrm{Lip}}
\newcommand{\F}{\mathcal{F}}
\newcommand{\E}{\mathcal{E}}

\newcommand{\eps}{\varepsilon}

\renewcommand{\d}{{\sf d}}
\newcommand{\B}{{\sf B}}

\newcommand{\X}{{X}}

\newcommand{\mm}{\mathfrak m}
\newcommand{\sloc}{{\rm sloc}}
\newcommand{\loc}{{\rm loc}}

\newcommand{\cont}{{\rm cont}}

\newcommand{\BE}{{\sf BE}}
\newcommand{\RCD}{{\sf RCD}}
\newcommand{\CD}{{\sf CD}}
\newcommand{\EVI}{{\sf EVI}}
\newcommand{\TE}{{\sf TE}}
\newcommand{\DTE}{{\sf DTE}}
\newcommand{\GE}{{\sf GE}}

\def\N{{\mathbb N}}

\def\R{{\mathbb R}}

\def\PP{{\mathbb P}}
\def\EE{{\mathbb E}}
\newcommand{\Prob}{\mathcal{P}}	
\newcommand{\Geo}{\mathrm{Geo}}

\title{Distribution-Valued Ricci Bounds for  Metric Measure Spaces, Singular Time Changes, and Gradient Estimates for Neumann Heat Flows
}
\author{Karl-Theodor Sturm}
 \bigskip
 
\date{\it\normalsize Dedicated to the memory of Professor Kazumasa Kuwada}

\begin{document}

\maketitle

\begin{abstract} We will study metric measure spaces $(X,\d,\mm)$  beyond the scope of spaces with synthetic lower Ricci bounds. In particular, we introduce distribution-valued lower Ricci bounds $\BE_1(\kappa,\infty)$
\begin{itemize}
\item for which  we prove the  equivalence with sharp gradient estimates,
\item the class of which will be preserved under time changes with arbitrary $\psi\in\Lip_b(X)$, and
\item which are satisfied for the Neumann Laplacian on arbitrary semi-convex subsets $Y\subset X$.
\end{itemize}
In the latter case, the distribution-valued Ricci bound will be given by the signed measure
$\kappa= k\,\mm_Y + \ell\,\sigma_{\partial Y}$
where $k$ denotes a variable synthetic lower bound for the Ricci curvature of $X$ and $\ell$ denotes a lower bound for the ``curvature of the boundary'' of $Y$, defined in purely metric terms.

We also present a new localization argument which allows us to pass on the RCD property to arbitrary open subsets of RCD spaces. And we introduce new synthetic notions for boundary curvature, second fundamental form, and boundary measure for subsets of RCD spaces.
\end{abstract}

\tableofcontents

\bigskip

\section{Introduction}

\paragraph{Background.}
Synthetic lower bounds for the Ricci curvature as introduced in the foundational papers \cite{lott2009, sturm2006a, sturm2006b} by Lott \& Villani and the author, opened the door for the development of a far reaching, vast theory of metric measure spaces $(X,\d,\mm)$ with lower bounded Ricci curvature. 
The theory is particularly rich if one assumes in addition that the spaces are infinitesimally Hilbertian.
For such spaces, Ambrosio, Gigli \& Savare in a series of seminal papers \cite{ambrosio2014a, ambrosio2014b, ambrosio2015, gigli2018} developed  a powerful first order calculus, based on (minimal weak upper)  gradients of functions and  on gradient flows for semiconvex functionals, in particular, energy on $L^2(X,\mm)$ and entropy on ${\mathcal P}_2(X,\d)$.
This was complemented by a huge number of contributions by many others, leading e.g.~to 
sharp estimates for volume growth and diameter, gradient estimates, transport estimates, Harnack inequalities, logarithmic Sobolev inequalities, isoperimetric inequalities, 
splitting theorems, maximal diameter theorems, and further rigidity results, see e.g.~\cite{erbar2015, kopfer2019, cavalletti2016, gigli2013, ketterer2015a, ketterer2015b,   erbar2017} and references therein.
Moreover, deep insights into the local structure of such spaces have been obtained \cite{mondino2019}, \cite{brue2018} and  also an impressive second order calculus could be developed \cite{gigli2018}.

\paragraph{Objective.}
The purpose of the current paper is to enlarge the scope of metric measure spaces with synthetic lower Ricci bounds far beyond  uniform  bounds. We will study in detail mm-spaces $(X,\d,\mm)$ with variable Ricci bounds $k: X\to\R$. More precisely, we will 
 present the Eulerian and the Lagrangian  characterizations of ``Ricci curvature at $x$ bounded from below by $k(x)$ and dimension bounded from above by $N$'' and prove their equivalence.

\medskip

Most importantly, we will also study  mm-spaces with distribution-valued Ricci bounds.
The crucial point will be to present a formulation of the Bakry-\'Emery inequality $\BE_1(\kappa,\infty)$ for  $\kappa\in W^{-1,\infty}(X)$
\begin{itemize}
\item which  allows us to prove its  equivalence with sharp gradient estimates,
\item the class of which will be preserved under time changes with arbitrary $\psi\in\Lip_b(X)$, 
\item and which is satisfied for the Neumann Laplacian on arbitrary semi-convex subsets $Y\subset X$.
\end{itemize}
In the latter case, the distribution-valued Ricci bound will be given by the signed measure
\begin{equation}\label{k+l}\kappa= k\,\mm_Y + \ell\,\sigma_{\partial Y}\end{equation}
where $k$ denotes a variable synthetic lower bound for the Ricci curvature of $X$ and $\ell$ denotes a variable lower bound for the ``curvature of the boundary'' of $Y$, defined in purely metric terms.
We introduce new synthetic notions for boundary curvature, second fundamental form, and boundary measure for subsets of RCD spaces.

In our approach, the technique of time change will play a key role.
In operator language, ``time change'' with weight $e^\psi$  means that $\E$, the Cheeger energy for the mm-space $(X,\d,\mm)$, is now considered as a quadratic form on $L^2(X,e^{2\psi}\mm)$. This changes the underlying geometry and -- with appropriate choices of $\psi$ -- it allows  non-convex sets to be made convex (``convexification'').
 
The distribution-valued Ricci bound $\BE_1(\kappa,\infty)$ 
 with $\kappa$ as in \eqref{k+l} will imply a gradient estimate 
for the Neumann heat flow $(\nabla P_{t}^Y)_{t\ge0}$ on $Y$  of the 
 type
\begin{eqnarray}\label{initial-grad}
\big| \nabla P_{t/2}^Yf\big|(x)&\le
&\mathbb E^Y_{x}\Big[  e^{-\frac12\int_0^{t}  k(B^Y_{s})ds-\frac12\int_0^{t} \ell(B^Y_s)dL^{\partial Y}_s}
\cdot\big|\nabla f(B^Y_{t})\big|
\Big].
\end{eqnarray}
Here $(\PP_x^Y, B_t^Y)_{x\in Y,t\ge0}$ denotes reflected Brownian motion  on $Y$ and 
$(L^{\partial Y}_t)_{t\ge0}$, the local time of $\partial Y$, is defined  via Revuz correspondence 
as the positive continuous additive functional associated with the surface measue $\sigma_{\partial Y}$.

Note that 
\begin{itemize}
\item for non-convex $Y$, no estimate of  type \eqref{initial-grad} can hold true without taking into account the curvature of the boundary;
\item even for convex $Y$,  estimate \eqref{initial-grad} will improve upon all previous estimates which ignore the curvature of the boundary. 
\end{itemize}
For instance, for the Neumann heat flow on the unit ball of $\R^n$, the right hand side of \eqref{initial-grad} will decay as $C_0e^{-C_1t}$ for large $t$ whereas ignoring $\ell$ will lead to bounds of order $C_0$.

\medskip

We also present a new powerful  localization argument which allows us to pass on the RCD property to arbitrary open subsets of RCD spaces. 

\medskip

\paragraph{Outline.}
Besides this Introduction, the paper has five sections, each of them of independent interest. Let us briefly summarize them.

In {\bf Section 2}, we define and analyze metric measure spaces with Ricci curvature bounded from below by distributions. 
Our $\BE_1(\kappa,\infty)$ condition for $\kappa\in W^{-1,\infty}(X)$ is the first formulation of a synthetic Ricci bound with distribution-valued $\kappa$ which leads to a sharp gradient estimate.

{\bf Section 3} is devoted to the study of mm-spaces with variable Ricci bounds. The main result will be the proof of the equivalence of  the Eulerian curvature-dimension condition (or ``Bakry-\'Emery condition'') $\BE_2(k,N)$ and the Lagrangian curvature-dimension condition (or ``Lott-Sturm-Villani condition'') CD$(k,N)$ -- as well as four other related conditions. This provides an extension of the seminal paper \cite{erbar2015} towards variable $k$ (instead of constant $K$) and of the recent paper \cite{braun2019} towards finite $N$ (instead of $N=\infty$).

In {\bf Section 4} we present two extensions of our recent work \cite{han2019} with B.~Han on    transformation of the curvature-dimension condition under time-change, both of fundamental importance.
Firstly, we prove that for  $\phi\in\Lip_\loc(X)\cap\Dom_\loc(\Delta)$,   time change with weight $\frac1\phi$ leads to a mm-space $(X',\d',\mm')$ with $X'=\{\phi>0\}$ which satisfies $\RCD(k',N')$ for suitable $k',N'$. This is of major general interest since it allows for localization within the class of RCD-spaces.
Secondly, we prove that for arbitrary $\psi\in\Lip_b(X)$,  time change with weight $e^\psi$ leads to a mm-space with distribution-valued Ricci bound $\kappa$ given in terms of the
distribution-valued Laplacian $\underline\Delta\psi$. This will be a crucial ingredient in our strategy for the proof of the gradient estimate in the final Section 6.

In {\bf Section 5} we
 extend the existence result and the contraction estimate for gradient flows for semiconvex functions from \cite{sturm2018a} to the setting of locally semiconvex functions. The contraction estimate for the flow will be in terms of the variable lower bound for the local semiconvexity of the potential.
And we will prove the fundamental Convexification Theorem which allows  us to transform the metric  of a  mm-space $(X,\d,\mm)$  in such a way that a given semiconvex subset $Y\subset X$ will become locally geodesically convex  w.r.t.~the new  metric $\d'$. 
Moreover, in a purely metric manner, we introduce the notion of  variable lower bound for the curvature of the boundary. 
In the Riemannian setting, such a bound will be equivalent to a lower bound for the second fundamental form of the boundary.   

The paper reaches its climax in {\bf Section 6} with the proof of the gradient estimates for the Neumann heat flow on not necessarily convex subsets $Y\subset X$.
The proof of these gradient estimates is quite involved. It builds on results from all other sections of the paper.
\begin{itemize}
\item 
 Given a semi-convex subset $Y$ of an $\RCD(k,N)$-space  $(X,\d,\mm)$, to get started, we perform a time-change with weight $e^\psi$ in order to make $Y$ 
  locally geodesically convex 
   in $(X,\d'):=(X,e^\psi\odot\d)$. The choice $\psi=(\epsilon-\ell)\, V$ with $V=\pm\d(.,\partial Y)$,  any $\epsilon>0$, and $\ell$ being a lower bound for the curvature of $\partial Y$ will do the job, see \emph{Section 5}.
\item Under the assumption that $\psi\in
\Dom_\loc(\Delta)$, the transformation formula for time changes provides a
$\RCD(k',N')$-condition for the time-changed space $(X,\d',\mm')$,     \emph{Section 4}.
\item Together with the local geodesical convexity of $Y$ this implies that also the restricted space $(Y,\d_Y',\mm_Y')$ satisfies the $\RCD(k',N')$-condition. Making 
use of the equivalence of Eulerian and Lagrangian characterizations of curvature-dimension conditions, we  conclude the $\BE_2(k',N')$-condition for $(Y,\d_Y'\,\mm_Y')$, \emph{Section 3}.  

\item To end up with $(Y,\d_Y\,\mm_Y)$ requires a ``time re-change'', i.e. another time change, now with weight $e^{-\psi}$.
In general, however, $\psi$ will not be in the domain of the Neumann Laplacian $\Delta^Y$. Ricci bounds under time re-change thus have to be formulated 
as $\BE_1(\kappa,\infty)$-condition for some $\kappa\in W^{-1,\infty}(X)$
in terms of the distributional Laplacian $\underline\Delta^Y\psi$,  \emph{Section 4}.

\item The $\BE_1(\kappa,\infty)$-condition  will imply the gradient estimate
$\big| \nabla P_{t}^Yf\big|\le P^{\kappa}_t \big| \nabla f\big|$
for the Neumann heat flow $(\nabla P_{t}^Y)_{t\ge0}$ on $Y$ in terms of a suitable semigroup 
$(P_{t}^\kappa)_{t\ge0}$,
\emph{Section 2}.  
\end{itemize}
This ``taming semigroup'' $(P_{t}^\kappa)_{t\ge0}$ will be represented in terms of the Brownian motion on $Y$ by means of the Feynman-Kac formula involving the integral $\int_0^tk(B_s)ds$ (taking into account the effects of the Ricci curvature in $Y$), and the integral $\int_0^t\ell(B_s)dL_s$ (taking into account the effects of the  curvature of $\partial Y$).


\paragraph{Basic concepts and notations.}
Throughout this paper, $(\X,\d,\mm)$ will be an arbitrary metric measure space, that is, $\d$ is a complete separable metric on $\X$ inducing the topology of $\X$ and $\mm$ is a Borel measure which is  
finite on  sets of an open covering.
Moreover, we assume that   $(\X,\d,\mm)$ is infinitesimally Hilbertian and that $\mm$ is finite on bounded sets.

To simplify notation, we often will write $L^p(X)$ or $L^p(\mm)$ or just $L^p$ instead of $L^p(X,\mm)$ and, similarly, $\Lip(X)$ instead of $\Lip(X,\d)$.
The space of Lipschitz  functions with \underline{b}ounded  \underline{s}upport 
on $X$ will be denoted by
$\Lip_{bs}(X)$ 
  whereas as usual $\Lip_{b}(X)$ denotes the space of bounded Lipschitz functions.
  The number $\Lip f$ will denote the Lipschitz constant of $f$.

 
Let us briefly recall that the energy functional (``Cheeger energy'') $\E:  L^2(X)\to [0,\infty]$ 
is defined as 
$$\E(f)=\int_X \big| Df\big|^2\, d\mm$$
in terms of  the minimal weak upper gradient $|Df|$ (which in the sequel often will also be denoted by $|\nabla f|$).
The set $\Lip_{bs}(X)$ is dense in $W^{1,2}(X):=\Dom(\E):=\{f\in L^2(X):\ \E(f)<\infty\}$.
The minimal weak upper gradient $|Df|$ gives rise to a map $W^{1,2}(X)\to L^1(X), \ f\mapsto \Gamma(f):=|Df|^2=|\nabla f|^2$ such that $\E(f)=\int\Gamma(f)\, d\mm$.
By $W^{1,2}_\loc(\X)$ we denote the set of all ($\mm$-equivalence classes of) measurable functions $f$ on $\X$ such that each point in $\X$ has a neighborhood $U$ such that $f=f_U$ $\mm$-a.e.~on $U$ for some $f_U\in W^{1,2}(\X)$. By the Lindel\"of property of complete separable metric spaces (and by using  truncation by means of standard cut--off functions on metric balls) it follows that $f\in W^{1,2}_\loc(\X)$ if and only if there exist an exhausting sequence of open sets $U_n\subset\X$ and a sequence of $f_n\in W^{1,2}(\X)$ such that  $f=f_n$ $\mm$-a.e.~on $U_n$ for each $n$.

Our assumption that 
$(\X,\d,\mm)$ is infinitesimally Hilbertian simply means  that the energy $\E$ is a quadratic form or, in other words, that its domain $W^{1,2}(X)$ is a Hilbert space. In this case, by polarization, $\E$ and $\Gamma$ extend to bilinear maps 
$\Gamma: \ W^{1,2}(X)\times W^{1,2}(X)\to L^1(X)$ and 
 $\E: \ W^{1,2}(X)\times W^{1,2}(X)\to\R$ with $(\phi,\psi)\mapsto \int_X \Gamma(\phi,\psi)\,d\mm$. 
 
 Indeed, 
 the bilinear form $\E$ is a quasi-regular Dirichlet form on $L^2(X,\mm)$, \cite{savare2014}. Its generator $\Delta$ is the  ``Laplacian'' on the mm-space $(\X,\d,\mm)$. The associated semigroup (``heat semigroup'') $(e^{\Delta\,t})_{t\ge0}$ on $L^2(X,\mm)$ will extend to a positivity preserving, $\mm$-symmetric, bounded semigroup $(P_t)_{t\ge0}$ on each $L^p(X,\mm)$  with
 $$\big\| P_t\big\|_{L^p(X,\mm)\to L^p(X,\mm)}\le 1\qquad\text{ for each $p\in [1,\infty]$,}$$
  strongly continuous on $L^p(X,\mm)$ if $p<\infty$.
  Quasi-regularity of $\E$ implies that each $f\in W^{1,2}(X)$ admits a quasi continuous version $\tilde f$ (and two such versions coincide q.e.~on $X$). Thus in particular, for each $f\in\bigcup_{p\in[1,\infty]} L^p(X,\mm)$ and $t>0$, there exists a quasi-continuous version $\tilde P_tf$ of $P_tf$ (uniquely determined q.e.).
 The $\mm$-reversible, continuous Markov process $\big({\mathbb P}_x, B_t\big)_{x\in X, t\ge0}$ (with life time $\zeta$)  associated with $\E$ is called ``Brownian motion'' on $X$.
 It is uniquely characterized by the fact that
 \begin{equation}\label{BM}
P_{t/2}f(x)=
{\mathbb E}_x\big[f(B_{t})\,1_{\{t<\zeta\}}\big],\qquad P_tf(x)=
{\mathbb E}_x\big[f(B_{2t})\,1_{\{2t<\zeta\}}\big].
\end{equation}
 (The factor 2 arises from the fact that by standard convention, the generator of the Brownian motion is $\frac12\Delta$ whereas  the generator of the heat semigroup in our setting is $\Delta$.)
 
 \medskip
 
 \thanks{\it The author would like to thank Mathias Braun, Zhen-Qing Chen, Matthias Erbar, Nicola Gigli, and Tapio Rajala for fruitful discussions and valuable contributions.

Financial support by the European Union through the ERC-AdG ``RicciBounds''
and by the DFG through the Excellence Cluster ``Hausdorff Center for Mathematics'' and through the Collaborative Research Center 1060
 is gratefully acknowledged.}

\bigskip\bigskip

\section{$W^{-1,\infty}$-valued Ricci bounds}
The goal of this section is to define and analyze metric measure spaces with Ricci curvature bounded from below by distributions. In particular, we will 
give a meaning to this extended notion of synthetic lower Ricci bounds and -- most importantly -- we will prove that these Ricci bounds lead to sharp  estimates for the gradient of the heat flow.
These results are of independent interest. 

In the context of this paper, they are of particular importance since in Section \ref{sec6} we will prove that the Ricci curvature of a semiconvex subset $Y$ of an $\RCD$-space $(X,\d,\mm)$ is bounded from below by the $ W^{-1,\infty}(X)$-distribution
$$\kappa= k\,\mm_Y + \ell\,\sigma_{\partial Y}$$
where $k$ denotes a variable synthetic lower bound for the Ricci curvature of $X$ and $\ell$ denotes a lower bound for the ``curvature of the boundary'' of $Y$ while $\sigma_{\partial Y}$ denotes the ``surface measure'' on $\partial Y$. In particular, the Ricci curvature of $Y$ will be bounded from below by a function if and only if $Y$ is convex.

\subsection{Taming Semigroup}

In the sequel, we also need certain normed spaces, denoted by $W^{1,1+}(X), W^{1,\infty}(X)$ and $W^{-1,\infty}(X)$.
We will define these spaces tailor made for the purpose of this paper. Our concept will be based on  
the 2-minimal weak upper gradient $|Df|$.

\begin{Def}
We put $$W^{1,\infty}(X):=\Big\{f\in W^{1,2}_\loc(X): \ \big\| |f| + |Df|\big\|_{L^\infty}<\infty\Big\}$$
and $W_*^{1,\infty}(X):=\big\{f\in W^{1,2}_\loc(X): \ \big\|  |Df|\big\|_{L^\infty}<\infty\big\}$. Moreover, we put
$$W^{1,1+}(X):=\Big\{f\in L^1(X): f_{[n]}\in W^{1,2}(X) \text{ for }n\in\N  \text{ and } \sup_n \big\| |f_{[n]}| +|Df_{[n]}|\big\|_{L^1}<\infty\Big\}$$
where $f_{[n]}:=(f\wedge n)\vee(-n)$ denotes the truncation of $f$ at levels $\pm n$, and
$$\big\|f\big\|_{W^{1,1+}}:=\sup_n \big\| |f_{[n]}| +|Df_{[n]}|\big\|_{L^1}= \big\| f\big\|_{L^1}+\sup_n \big\| |Df_{[n]}|\big\|_{L^1}.$$
\end{Def}

\begin{remark} 
  {\bf i)} 
The precise definition of these spaces will not be so relevant for us. What we need are the following properties:
\  $W^{1,1+}(X)$ contains all squares of functions from $W^{1,2}(X)$; \
$W^{1,\infty}(X)$ includes $\Lip_b(X)$; \
$\Gamma$ 
extends to a continuous bilinear map $W^{1,1+}(X)\times W^{1,\infty}(X)\to L^1(X,\mm)$.

{\bf ii)} 
$W^{1,\infty}(X)$ is a Banach space. 
If the mm-space $(X,\d,\mm)$ satisfies some \RCD$(K,\infty)$-condition, according to the Sobolev-to-Lipschitz property, the space $W^{1,\infty}(X)$ will coincide with the space $\Lip_b(X)$ and the space $W_*^{1,\infty}(X)$ will coincide with the space $\Lip(X)$.

{\bf iii)} 
$W^{1,1+}(X)$ is a normed space but in general not complete. For instance, the functions $f_j(r)=\sqrt{ r\vee (1/j)}$, $j\in\N$, on $X=[-1,1]$ will constitute a Cauchy sequence in $W^{1,1+}(X)$ but their $L^1$-limit $f_\infty(r)=\sqrt r$ is not contained in $W^{1,1+}(X)$.
For Riemannian $(X,\d,\mm)$, the completion of $W^{1,1+}(X)$ will coincide with $W^{1,1}(X)$.

For general $(X,\d,\mm)$, the definition of $W^{1,1}(X)$ is quite sophisticated and allows for ambiguity, see e.g.  \cite{ambrosio2014}.
For a detailed study of the 
 spaces $W^{1,p}(X)$ for $p\in(1,\infty)$, see \cite{gigli2016}. 
 
 {\bf iv)} For $f\in W^{1,1+}(X)$, there exists a unique $|Df|\in L^1(X)$ with 
 $$|Df|=|Df_{[n]}| \ \mm\text{-a.e. on }\big\{|f|\le n\big\} \quad\text{for each }n\in\N.$$
 Indeed, by locality of the minimal weak upper gradient, the family $|Df_{[n]}|, n\in\N$, is consistent in the  sense that $|Df_{[n]}|=|Df_{[j]}|$  $\mm$-a.e. on 
the set  $\big\{|f|\le \min(n,j)\big\}$ for each $n,j\in\N$. Hence, $|Df_{[n]}|, n\in\N$, is a Cauchy sequence in $L^1(X)$ and therefore, it admits a unique limit in $L^1(X)$, denoted by $|Df|$.
\end{remark}

\begin{lma}\label{prodcuts}
\quad $f,g\in W^{1,2}(X)\quad\Longrightarrow\quad f\, g\in W^{1,1+}(X)$.
\end{lma}

\begin{proof}
It suffices to prove the claim for $f=g$. Given $f\in W^{1,2}(X)$, put $h=f^2$. Then obviously $h\in L^1(X)$. Moreover, $h_{[n^2]}\in W^{1,2}(X)$ for each $n$
 since $\big| h_{[n^2]}\big|\le n\, \big| f\big|$ and
$\big| D h_{[n^2]}\big|\le2n\, \big| Df\big|$. Finally,
$$\sup_n\int  \big| D h_{[n^2]}\big|\,d\mm\le 2\int |f|\, |Df|\,d\mm\le \big\|f\big\|^2_{W^{1,2}}.$$
This proves the claim.
\end{proof}

\begin{lma}\label{lem-extend-e} The map
 $\Gamma$ extends to a continuous bilinear map
 $\Gamma: \ W^{1,1+}(X)\times W_*^{1,\infty}(X)\to L^1(X)$ and 
 $\E$ extends to a continuous bilinear form 
 $$\E: \ W^{1,1+}(X)\times W_*^{1,\infty}(X)\to\R, \  (f,g)\mapsto \int_X \Gamma(f,g)\,d\mm.$$ 
 Here and in the sequel continuity on $ W_*^{1,\infty}(X)$ is meant w.r.t.~the semi-norm $f\mapsto \big\| |Df| \big\|_{L^\infty}$.
\end{lma}

\begin{proof}  $\big( |Df_{[n]}|\big)_n$ is a Cauchy sequence in $L^1(X)$ for $f\in W^{1,1+}(X)$. Hence, 
$\big( \Gamma(f_{[n]},g)\big)_n$ is a Cauchy sequence in $L^1(X)$
for $g\in  W_*^{1,\infty}(X)$. Denoting its limit by $\Gamma(f,g)$, yields
$$\big|\E(f,g)\big|\le\int_X \big|\Gamma(f,g)\big|d\mm=\lim_n \int_{\{ |f|\le n\}} \big|\Gamma(f,g)\big|d\mm\le\big\| f\big\|_{W^{1,1+}}\cdot \big\| g\big\|_{W_*^{1,\infty}}.$$
\end{proof}

\begin{Def} \quad $W^{-1,\infty}(X):= {W^{1,1+}(X)}'$, the topological dual  of $W^{1,1+}(X)$.
\end{Def}

More precisely, this space should be denoted by  $W^{-1,\infty-}(X)$. We prefer the notation 
$W^{-1,\infty}(X)$ for simplicity -- and in view of the fact that ${W^{1,1+}(X)}'={W^{1,1}(X)}'$ in `regular' cases.

\begin{remark}   \  
$L^\infty(X)$ continuously embeds into $W^{-1,\infty}(X)$  via
$$\langle \phi,k\rangle_{W^{1,1+}, W^{-1,\infty}}:=\int \phi\, k\,dm\qquad (\forall \phi\in W^{1,1+}(X))$$
for each $k\in L^\infty(X)$. 
\end{remark}

\begin{cor}
A continuous linear operator $\underline\Delta: \ W_*^{1,\infty}(X)\to W^{-1,\infty}(X)$ can be defined by 
$$\langle \phi,\underline\Delta\psi\rangle_{W^{1,1+}, W^{-1,\infty}} =-\int_X\Gamma(\phi,\psi)\,d\mm\qquad (\forall \phi\in W^{1,1+}(X)).$$

On $W_*^{1,\infty}(X)\cap\Dom(\Delta)
$, this operator  
obviously coincides with the usual Laplacian $\Delta$.
\end{cor}

\begin{ex} 
Let $(X,\d,\mm)$ be the standard 1-dimensional mm-space with $X=\R$ and let $x_i^n$ for $n\in\N$ and $i=1,\ldots, 2^{n-1}$ be the centers of the intervals of length $3^{-n}$ in the mid-third construction of the Cantor set $S\subset [0,1]$.
Choose $\varphi(x)=(\frac12-|x|)_+$ or, more sophisticated, choose a nonnegative function $\varphi\in {\mathcal C}^2(\R)$ with $\{\varphi>0\}=(-1/2,1/2)$ 
such  that the closures of $\{\Delta\varphi>0\}$ and $\{\Delta\varphi<0\}$ are disjoint.
Put
\begin{equation}\label{CantorPhi}\Phi(x):=\lim_{j\to\infty}\Phi_j(x), \quad \Phi_j(x):=\sum_{n=1}^j\sum_{i=1}^{2^{n-1}} 3^{-n}\varphi\big(3^n(x-x_i^n)\big).\end{equation}
Then $\Phi\in W^{1,\infty}(\R)$, more precisely, 
$$\|\Phi\|_\infty\le \frac13\|\varphi\|_\infty, \quad \|\nabla \Phi\|_\infty\le \|\nabla \varphi\|_\infty.$$
But $\underline\Delta\Phi$ is not a signed Radon measure. 

To prove the latter, 
 for each $j\in\N$,  choose  a ${\mathcal C}^1$-function  $f_j\le1$ on $\R$ with $f_j=1$ on $\{\Delta\Phi_j>0\}$ and $f_j=0$ on $\{\Phi_j\not=\Phi\}$.
 Then with $C:=\int_\R \big(\Delta\varphi(x)\big)_+dx>0$,
 \begin{eqnarray*}
 \int_0^1 d\big(\underline\Delta\Phi\big)_+&\ge&
  \int_0^1 f_jd\big(\underline\Delta\Phi\big)_+
  =  \int_0^1 f_jd\big(\underline\Delta\Phi_j\big)_+
  =\int_0^1 f_j\,\Delta\Phi_j\,dx
 =\int_0^1 (\Delta\Phi_j)_+\,dx\\
%
&=&\sum_{n=1}^j \sum_{i=1}^{2^{n-1}} 3^n
\int_0^1 \big( \Delta\varphi
\big(3^n(x-x_i^n)\big)
\big)_+dx
=C\, (2^j-1)\ \to\ \infty
 \end{eqnarray*}
 as $j\to\infty$. Thus $\int_0^1 d\big(\underline\Delta\Phi\big)_+=\infty$.
Furthermore, by scaling 
$$\int_0^{3^{-k}} d\big(\underline\Delta\Phi\big)_+=\infty$$
for all $k\in\N$. This proves that $\big(\underline\Delta\Phi\big)_+$ is not a locally finite measure and thus  $\underline\Delta\Phi$ is no Radon measure.
\end{ex}

\begin{Prop}\label{p-kappa} Given $\kappa\in W^{-1,\infty}(X)$, we
define a closed  bilinear form $\E^\kappa$ on $L^2(X)$ by
$$\E^\kappa(f,g):=\E(f,g)+
\langle f\,g,\kappa\rangle_{W^{1,1+}, W^{-1,\infty}}
$$
for $f,g\in \Dom(\E^\kappa):=W^{1,2}(X)$. The form is bounded from below on  $L^2(X)$ by $-C(C+1)$ where $C:=\|\kappa\|_{W^{-1,\infty}(X)}$.

Associated to it, there is a  strongly continuous, positivity preserving semigroup $(P^\kappa_t)_{t\ge0}$ on $L^2(X)$ with
$$\|P^\kappa_t\|_{L^2\to L^2}\le e^{C(C+1)t}.$$
\end{Prop}

\begin{remark}
(i)
The form $\E^\kappa$ is not only lower bounded, it is a ``form small perturbation of $\E$''. Indeed, for every $\delta>0$ and all $f\in W^{1,2}(X)$
$$\E^\kappa(f,f)\ge (1-\delta)\, \E(f,f)-\Big(C+\frac{C^2}\delta\Big)  \, \|f\|^2_{L^2}.$$
(ii) 
If $\kappa\in L^\infty(X)$ 
then  $(P^\kappa_t)_{t\ge0}$ is given by the Feynman-Kac formula associated with the Schr\"odinger operator $-\Delta +\kappa$ with potential $\kappa$:
\begin{equation*}
P^\kappa_tf(x)=
{\mathbb E}_x\big[e^{-\int_0^t \kappa(B_{2s})ds}f(B_{2t})\,1_{\{2t<\zeta\}}\big]
\end{equation*}
where $({\mathbb P}_x, (B_{t})_{t\ge0})$ denotes Brownian motion starting in $x\in X$.

Note that ${\mathbb P}_x$-a.s.~for $\mm$-a.e.~$x$, the random variables $f(B_{2t})$ and $\int_0^t \kappa(B_{2s})ds$  do not depend on the choice of the Borel versions of $f$ and $\kappa$, resp., since ${\mathbb E}_g\big[f(B_{2t})\,1_{\{2t<\zeta\}}\big]=\int_X g\,P_tf\,d\mm$ and
${\mathbb E}_g\big[\int_0^t \kappa(B_{2s})\,1_{\{2s<\zeta\}}ds\big]=\int_0^t\int_Xg\,P_s\kappa\,\,d\mm\,ds$ for $g\in L^1(X,\mm)$.
\end{remark}

\begin{proof}[Proof of Proposition and Remark (i)] The lower boundedness and more generally the form smallness easily follow from
\begin{eqnarray*}\Big|\langle f^2 ,\kappa\rangle_{W^{1,1+}, W^{-1,\infty}}\Big|&\le& C\cdot \|f^2\|_{W^{1,1+}}\le C\cdot 
 \|f\|_{L^2}^2+2C\cdot \|f\|_{L^2}\cdot \E(f)^{1/2}\\
 &\le&
 (C+C^2/\delta)\cdot 
 \|f\|_{L^2}^2+\delta \cdot  \E(f).
\end{eqnarray*}
In particular, this implies that $\E^\kappa(f)\ge -(C+C^2)\|f\|^2$ for all $f$ and thus by spectral calculus
$$\langle f,P^\kappa_tf\rangle_2\ge e^{-(C+C^2)\,t}\|f\|_2^2.$$

According to the first Beurling-Deny criterion, the semigroup $(P^\kappa_t)_{t\ge0}$ is positivity preserving if and only if 
$$f\in\Dom(\E^\kappa)\quad\Rightarrow\quad |f|\in\Dom(\E^\kappa) \text{ and } \E^\kappa(|f|)\le \E^\kappa(f),$$
see \cite{davies1990}, Theorem 1.3.2. This criterion obviously is fulfilled. Indeed, $\E^\kappa(|f|)= \E^\kappa(f)$ for $f\in W^{1,2}(X)=\Dom(\E^\kappa)$.
\end{proof}

Of particular interest will be to analyze the semigroup $(P^\kappa_t)_{t\ge0}$ in the case where $\kappa=-\underline\Delta\psi$ for some $\psi\in\Lip(X)$. Recall that this semigroup is well understood in the ``regular'' case where $\psi\in\Dom(\Delta)\cap L^\infty(X)$. Indeed, then
\begin{equation}\label{Fey-Kac}
P^\kappa_{t/2}f(x)=
{\mathbb E}_x\big[e^{\frac12\int_0^t \Delta\psi(B_{s})ds}f(B_{t})\,1_{\{t<\zeta\}}\big]
\end{equation}
or, in other words,
$P^\kappa_tf(x)=
{\mathbb E}_x\big[e^{\int_0^t \Delta\psi(B_{2s})ds}f(B_{2t})\,1_{\{2t<\zeta\}}\big]$.
For general $\psi$, however, this Feynman-Kac formula a priori does not make sense. We will have to find an appropriate replacement of it.

\begin{Prop}\label{psi-form}
 {\bf (i)} \ 
Given $\psi\in  \Lip(X)$, put $\kappa=-\underline\Delta\psi$. Then the 
closed, lower bounded bilinear form $\E^{\kappa}$ on $L^2(X)$ with domain $W^{1,2}(X)$ is given by
\begin{equation}\E^{\kappa}(f,g)=\E(f,g)+\E(fg,\psi).\end{equation}
(For the last expression here we used the fact that $\E$ extends to a continuous bilinear form 
 $W^{1,1+}(X)\times \Lip(X)\to\R$, Lemma \ref{lem-extend-e}, and that 
  $fg\in W^{1,1+}(X)$ for  $f,g\in W^{1,2}(X)$, Lemma \ref{prodcuts}.)

The strongly continuous, positivity preserving semigroup on $L^2(X)$ associated to it satisfies
$$\|P^{\kappa}_t\|_{L^2\to L^2}\le e^{(\Lip\,\psi)^2\,t}.$$

{\bf (ii)} \ 
Put $\hat\mm:=e^{-2\psi}\mm$. Then the unitary transformation (= Hilbert space isomorphism)
$$\Phi: \quad L^2(X,\mm)\to L^2(X,\hat\mm), \quad f\mapsto \hat f=e^{\psi}f$$
maps the quadratic form $\E^{\kappa}$, densely defined on $L^2(X,\mm)$, onto the quadratic form
$$\hat\E^\kappa(g):=\E^\kappa(e^{-\psi}g)=\int_X\Big[\Gamma(g)-g^2\,\Gamma(\psi)\Big]\,d\hat\mm,$$
densely defined on $L^2(X,\hat\mm)$ and  bounded from below by $-(\Lip\psi)^2\,\|g\|^2_{L^2}$.\\
(Since $\psi$ is bounded on bounded sets, $\Gamma$ coincides with the Gamma-operator for the metric measure space $(X,\d,\hat\mm)$ and  $\Phi$ maps $W^{1,2}(X,\d,\mm)$ bijectively onto $W^{1,2}(X,\d,\hat\mm)$.
Moreover,
$\hat\E$ is just a perturbation of the canonical energy on  $(X,\d,\hat\mm)$ by a bounded zeroth order term.)

 {\bf (iii)} \ 
The semigroup $(\hat P_t^\kappa)_{t\ge0}$ on $L^2(X,\hat\mm)$ associated with the
the quadratic form
$\hat\E^\kappa$
is related to the semigroup $(P_t^\kappa)_{t\ge0}$ on $L^2(X,\mm)$
via
$$\hat P_t^\kappa f:=e^{\psi}\, P^\kappa_t\big(e^{-\psi}f\big).$$ Furthermore, it can be represented in terms of the heat semigroup
$(\hat P_t)_{t\ge0}$ on $L^2(X,\hat\mm)$
by the Feynman-Kac formula with potential $-\Gamma(\psi)$. Since the latter is a bounded function,
the semigroup is bounded on each $L^p(X,\hat\mm)$ with
$$\big\|\hat P_t^\kappa\big\|_{L^p(X,\hat\mm)\to L^p(X,\hat\mm)}\le e^{(\Lip\,\psi)^2t}$$
for all $p\in[1,\infty]$. 
This allows us to conclude that for each $\psi\in\Lip_b(X)$, the original semigroup satisfies
\begin{equation}\big\|P_t^\kappa\big\|_{L^p(X,\mm)\to L^p(X,\mm)}\le e^{|1- 2/p|\,\operatorname{osc}\psi+(\Lip\,\psi)^2t}.
\end{equation}
\end{Prop}

\begin{proof} 
The norm estimate in (i) follows from the fact that
$$\E^{\kappa}(f,f)\ge\E(f,f)-\Lip\,\psi\cdot\int\Gamma(f^2)^{1/2}d\mm\ge-(\Lip\,\psi)^2\cdot\|f\|_{L^2}^2$$
and the estimate in (iii) from
\begin{eqnarray*}
\big\|P^\kappa_tf\big\|_{L^p(\mm)}&=&\Big(\int e^{-p\psi}\,\big[ \hat P_t^\kappa(e^\psi\,f)\big]^p\, e^{2\psi}d\hat\mm\Big)^{1/p}\le
e^{-\inf [(1-2/p)\psi]}\, e^{(\Lip\psi)^2t}\cdot \big\|e^\psi f\big\|_{L^p(\hat\mm)}\\
&\le&
e^{-\inf [(1-2/p)\psi]}\, e^{(\Lip\,\psi)^2t}\,e^{\sup [(1-2/p)\psi]}\cdot \big\| f\big\|_{L^p(\mm)}.
\end{eqnarray*}
The rest is straightforward.
\end{proof}

\bigskip

A more explicit representation for the semigroup $(P^\kappa_t)_{t\ge0}$ will be possible by extending the Fukushima decomposition which in turn is an extension of the famous Ito decomposition. In the Euclidean case with smooth $\psi$,
the latter states that
$$\psi(B_t)=\psi(B_0)+\int_0^t\nabla\psi(B_s)dB_s+\frac12\int_0^t\Delta\psi(B_s)ds.$$
This  indicates a way how to replace the expression $\frac12\int_0^t\Delta\psi(B_s)ds$ appearing in \eqref{Fey-Kac} 
by expressions which only involve first (and zero) order derivatives of $\psi$.

\begin{lma}[``Fukushima decomposition''] \ 

 {\bf (i)} \ For each $\psi\in\Lip_{bs}(X)$ there exists a unique 
 martingale additive functional  $M^\psi$  and a unique   continuous additive functional which is of zero quadratic variation $N^\psi$ such that 
\begin{equation}\label{Fuku}\psi(B_t)=\psi(B_0)+M^\psi_t+N^\psi_t\qquad\text{ $(\forall t\in[0,\zeta))\quad{\mathbb P}_x$-a.s.~for q.e.~$x\in X$}
\end{equation}

 {\bf (ii)} \ For each $\psi\in\Lip(X)$ there exists a unique local martingale additive functional  $M^\psi$  such that for each  $z\in X$,
 $$M^\psi_t=\lim_{n\to\infty}M^{\psi_n}_t\qquad\text{ $(\forall t\in[0,\zeta))\quad{\mathbb P}_x$-a.s.~for q.e.~$x\in X$}$$
 where $M^{\psi_n}$ denotes the martingale additive functional associated with the function 
 $\psi_n=\chi_n\cdot\psi\in\Lip_{bs}(X)$ according to part (i) and where $\chi_n(.)=[1-\d(\B_n(z),.)]_+$ for $n\in\N$.
 
  {\bf (iii)} \  The quadratic variation of $M^\psi$ is given by
 $$\langle M^\psi\rangle_t=\int_0^t\Gamma(\psi)(B_s)ds\qquad\text{ $(\forall t\in[0,\zeta))\quad {\mathbb P}_x$-a.s.~for q.e.~$x\in X$}$$
 for any choice of a Borel version of the function $\Gamma(\psi)\in L^\infty(X,\mm)$.
\end{lma}
For the defining properties of ``martingale additive functionals'' and of ``continuous additive functionals of zero quadratic variation'' (as well as for the relevant equivalence relations  that underlie the uniqueness statement) 
we refer to the monograph \cite{fukushima2011}.

\begin{proof} Assertion (i) is one of the key results in  \cite{fukushima2011}. Indeed, it is proven there as Theorem 5.2.2 for general quasi continuous $\psi\in \Dom(\E)$ and it is extended in Theorem 5.5.1 by localization to a more general class which contains $\Lip(X)$.
Also assertion (iii) for $\psi\in\Lip_{bs}(X)$ is a standard result, see 
\cite{fukushima2011}, Theorem 5.2.5. Let us briefly discuss its extension to general  $\psi\in\Lip(X)$. 

Given $\psi\in\Lip(X)$ and $z\in X$, we define $\psi_n=\chi_n\cdot\psi$ with cut-off functions $\chi_n$ as above and stopping times
$\tau_\lambda:=\inf\{t\ge0: B_t\not\in \B_\lambda(z)\}$ for $\lambda\in\N$. Then we put
$$M^{\lambda}_t:=M^{\psi_n}_{t\wedge\tau_\lambda}$$
for any $\lambda\le n$. Thus $M^\lambda$ is a martingale with ${\mathbb E}_xM^\lambda_t=0$ and
$$\langle M^\lambda\rangle_t=\int_0^{t\wedge\tau_\lambda}\Gamma(\psi)(B_s)ds\qquad\text{ ${\mathbb P}_x$-a.s.~for q.e.~$x\in X$}.$$
It follows that for q.e.~$x$, the family $(M^\lambda_t)_{\lambda\in\N}$ is an $L^2$-bounded martingale w.r.t.~${\mathbb P}_x$ with
$${\mathbb E}_x\Big[\big( M^\lambda_t\big)^2\Big]
\le {\mathbb E}_x\Big[\int_0^t\Gamma(\psi)(B_s)ds\Big]=\int_0^t P_{s/2}\Gamma(\psi)(x)ds\le (\Lip\, \psi)^2\cdot t.$$
Thus the limit $M_t:=\lim_{\lambda\to\infty}M_t^\lambda$ exists and is a martingale w.r.t.
${\mathbb P}_x$ for q.e.~$x\in X$.
Moreover, $\langle M\rangle_t=\lim_{\lambda\to\infty}\langle M^\lambda\rangle_t=\int_0^t\Gamma(\psi)(B_s)ds$.
\end{proof}

\begin{Prop} Given $\psi\in\Lip_b(X)$, put
\begin{equation}\label{N-ito}N^\psi_t:=\psi(B_t)-\psi(B_0)-M^\psi_t\end{equation}
with $M^\psi_t$ as defined in the previous Lemma, part (ii). Then 
with $(P^\kappa_{t/2})_{t\ge0}$ as defined in Proposition \ref{p-kappa}, for each $f\in \bigcup_{p\in[1,\infty]} L^p(X,\mm)$,
\begin{equation}\label{N-repr}
P^\kappa_{t/2}f(x)={\mathbb E}_x\big[e^{N^\psi_t}f(B_t)\,1_{\{t<\zeta\}}\big]\qquad\text{ for $\mm$-a.e.~$x\in X$}.\end{equation}
\end{Prop}

\begin{proof} In order to derive the representation formula \eqref{N-repr} with the additive functional $N$ given by \eqref{N-ito},
we will replace the (non-existing) Feynman-Kac transformation with potential $\frac12\underline\Delta\psi$ by
\begin{itemize}
\item[\bf (a)] a Girsanov transformation with drift $-\Gamma(\psi,.)$
\item[\bf (b)]  together with a Feynman-Kac transformation with potential $\frac12\Gamma(\psi)$
\item[\bf (c)]  followed by a Doob transformation with function $e^{\psi}$.
\end{itemize}
Each of these transformations provides a multiplicative factor in the representation of the semigroup which together amount to 
$$e^{-M_t^\psi+\frac12\langle M\rangle_t}\cdot e^{-\frac12\int_0^t\Gamma(\psi)(B_s)ds}\cdot e^{\psi(B_t)-\psi(B_0)}=e^{N^\psi_t}.$$
Let us perform these transformations first under the additional assumption that $\psi\in \Lip_{bs}(X)$ in which case all details can be found in the paper \cite{chen2002} since in this case $\psi\in \Dom(\E)$ and
$\mu_{\langle \psi\rangle}=\Gamma(\psi)\,\mm$ is a Kato class measure (indeed, it is a measure with bounded density).

  {\bf (a)} In the first step, we pass from the metric measure space $(X,\d,\mm)$ to the metric measure space $(X,\d,\hat\mm)$ with $\hat\mm=e^{-2\psi}\mm$ or, equivalently, we pass from the Dirichlet form $\E(f)=\int\Gamma(f)d\mm$ on $L^2(X,\mm)$ to the
  Dirichlet form $\hat\E(f)=\int\Gamma(f)d\hat\mm$ on $L^2(X,\hat\mm)$. This 
    amounts to pass from the heat semigroup $(P_t)_{t\ge0}$ to the semigroup $(\hat P_t)_{t\ge0}$ given by Girsanov's formula
$$\hat P_{t/2}f(x)=\hat{\mathbb E}_x\big[f(B_t)\big]={\mathbb E}_x\big[e^{-M^\psi_t-\frac12\langle M^\psi\rangle_t}\,f(B_t)\,1_{\{t<\zeta\}}\big]
\qquad\text{ for $\mm$-a.e.~$x\in X$}$$
with $M^\psi$ being the martingale additive functional as introduced in the previous Lemma.

{\bf (b)} In the second step, we pass from the
  Dirichlet form $\hat\E(f)=\int\Gamma(f)d\hat\mm$ on $L^2(X,\hat\mm)$ to the 
  Dirichlet form $\hat\E^\kappa(f)=\int\big[\Gamma(f)-\Gamma(\psi)\cdot f^2\big]d\hat\mm$ on $L^2(X,\hat\mm)$.
  This amounts to pass from the  semigroup $(\hat P_t)_{t\ge0}$ to the semigroup 
  $(\hat P^\kappa_t)_{t\ge0}$ given by Feynman-Kac's formula
$$\hat P^\kappa_{t/2}f(x)=\hat{\mathbb E}_x\big[e^{\frac12\int_0^t\Gamma(\psi)(B_s)ds}\, f(B_t)\,1_{\{t<\zeta\}}\big]={\mathbb E}_x\big[e^{-M^\psi_t}\,f(B_t)\,1_{\{t<\zeta\}}\big]\qquad\text{ for $\mm$-a.e.~$x\in X$}.$$

  {\bf (c)} In the final step, we pass from the 
  Dirichlet form $\hat\E^\kappa(f)=\int\big[\Gamma(f)-\Gamma(\psi)\cdot f^2\big]d\hat\mm$ on $L^2(X,\hat\mm)$ to the 
  Dirichlet form $\E^\kappa(f)=\int\big[\Gamma(f)+\Gamma(f^2,\psi)\big]d\mm$ on $L^2(X,\mm)$, see previous Proposition.
  This amounts to pass from the  semigroup $(\hat P^\kappa_t)_{t\ge0}$ to the semigroup 
  $(P^\kappa_t)_{t\ge0}$ given by Doob's formula
  $$P^\kappa_{t/2}f(x)=e^{-\psi(x)}\hat P^\kappa_{t/2}(e^{\psi}f)(x)=
  {\mathbb E}_x\big[e^{-\psi(B_0)-M^\psi_t+\psi(B_t)}\,f(B_t)\,1_{\{t<\zeta\}}\big]={\mathbb E}_x\big[e^{N^\psi_t}\,f(B_t)\,1_{\{t<\zeta\}}\big]$$
 for $\mm$-a.e.~$x\in X$.
  This proves the claim in the case $\psi\in \Lip_{bs}(X)$. 
  
  For general $\psi\in \Lip(X)$, we choose cut-off functions $\chi_n, n\in\N$, as in the previous Lemma and put $\psi_n=\chi_n\cdot\psi$ and $\kappa_n=-\underline\Delta\psi_n$. Then by the previous argumentation,
   $$P^{\kappa_n}_{t/2}f(x)={\mathbb E}_x\big[e^{N^{\psi_n}_t}\,f(B_t)\,1_{\{t<\zeta\}}\big]\qquad\text{ for $\mm$-a.e.~$x\in X$}$$
   for each $n\in\N$. It remain to prove 
 \begin{equation}\label{eins}
P^{\kappa_n}_{t/2}f(x)\to P^{\kappa}_{t/2}f(x)\qquad\text{ for $\mm$-a.e.~$x\in X$}
\end{equation}
 as $n\to\infty$ as well as
\begin{equation}\label{zwei}
{\mathbb E}_x\big[e^{N^{\psi_n}_t}\,f(B_t)\,1_{\{t<\zeta\}}\big] \to {\mathbb E}_x\big[e^{N^{\psi}_t}\,f(B_t)\,1_{\{t<\zeta\}}\big]\qquad\text{ for $\mm$-a.e.~$x\in X$}.
\end{equation}

To prove the latter, let us first restrict to $f\in L^p(X,\mm)$ for some $p\in(1,\infty]$. Then 
\begin{eqnarray*}
\Big| {\mathbb E}_x\big[e^{N^{\psi_n}_t}\,f(B_t)\,1_{\{t<\zeta\}}\big] - {\mathbb E}_x\big[e^{N^{\psi}_t}\,f(B_t)\,1_{\{t<\zeta\}}\big]\Big|\le
 {\mathbb E}_x\Big[ \big(e^{N^{\psi_n}_t}+e^{N^{\psi}_t}\big)\, |f|(B_t)\,1_{\{\zeta>t>\tau_n\}}\Big].
\end{eqnarray*}
For every $q,s\in(1,\infty)$ with $\frac1p+\frac1q+\frac1s=1$, by Hölder's inequality
\begin{eqnarray*}
 {\mathbb E}_x\Big[ e^{N^{\psi}_t}\, |f|(B_t)\,1_{\{\zeta>t>\tau_n\}}\Big]
 &\le&e^{2\|\psi\|_{L^\infty}}\cdot  {\mathbb E}_x\Big[ e^{-M^{\psi}_t}\, |f|(B_t)\,1_{\{\zeta>t>\tau_n\}}\Big]\\
 &\le&e^{2\|\psi\|_{L^\infty}+\frac s2(\Lip\psi)^2t}\cdot  {\mathbb E}_x\Big[ e^{-M^{\psi}_t-\frac s2\langle M^\psi\rangle_t}
 \, |f|(B_t)\,1_{\{t>\tau_n\}}\Big]\\
  &\le&e^{2\|\psi\|_{L^\infty}+\frac s2(\Lip\psi)^2t}\cdot  {\mathbb E}_x\Big[ e^{-sM^{\psi}_t-\frac{s^2}2\langle M^\psi\rangle_t}\,1_{\{t<\zeta\}}\Big]^{1/s}
  \\
  &&\cdot
 {\mathbb E}_x\Big[ |f|^p(B_t)\,1_{\{t<\zeta\}}\Big]^{1/p}
 \cdot {\mathbb P}_x\big[\{\zeta>t>\tau_n\}\big]^{1/q}\\
 &\le&e^{2\|\psi\|_{L^\infty}+\frac s2(\Lip\psi)^2t}\cdot(P_{t/2}|f|^p)^{1/p}(x)\cdot {\mathbb P}_x\big[\{\zeta>t>\tau_n\}\big]^{1/q}
\end{eqnarray*}
for $\mm$-a.e.~$x\in X$. For the last estimate we used the fact that $e^{-sM^{\psi}_t-\frac{s^2}2\langle M^\psi\rangle_t}\,1_{\{t<\zeta\}}$ is a super-martingale.
Thus obviously ${\mathbb E}_x\Big[ e^{N^{\psi}_t}\, |f|(B_t)\,1_{\{\zeta>t>\tau_n\}}\Big]\to0$ for $\mm$-a.e.~$x\in X$ as $n\to\infty$.
Analogously, we can estimate
\begin{eqnarray*}
 {\mathbb E}_x\Big[ e^{N^{\psi_n}_t}\, |f|(B_t)\,1_{\{\zeta>t>\tau_n\}}\Big]
 &\le&e^{2\|\psi_n\|_{L^\infty}+\frac s2(\Lip_n\psi)^2t}\cdot(P_{t/2}|f|^p)^{1/p}(x)\cdot {\mathbb P}_x\big[\{\zeta>t>\tau_n\}\big]^{1/q}\\
 &\le&e^{2\|\psi\|_{L^\infty}+\frac s2(\Lip\psi+ \|\psi\|_{L^\infty})^2t}\cdot(P_{t/2}|f|^p)^{1/p}(x)\cdot {\mathbb P}_x\big[\{\zeta>t>\tau_n\}\big]^{1/q}\\
 &\to& 0
\end{eqnarray*}
for $\mm$-a.e.~$x\in X$ as $n\to\infty$.
This proves \eqref{zwei} in the case $f\in L^p(X,\mm)$ for some $p\in(1,\infty]$. The claim for $f\in L^1(X,\mm)$ follows by a simple truncation argument and monotone convergence.

To prove  \eqref{eins}, it suffices to consider the case $f\in L^2(X,\mm)$.
The assertion for $f\in L^p$, $p\not=2$, follows by density of $L^2\cap L^p$ in $L^p$ and by boundedness of $P^\kappa_t$ (as well as boundedness of $P^{\kappa_n}_t$, uniformly in $n$) on $L^p$, cf. previous Proposition. To deduce \eqref{eins} in the case $p=2$, Duhamel's formula allows us to derive
\begin{eqnarray*}
\int g \big(P^\kappa_tf-P^{\kappa_n}_tf)\,d\mm=-\int_0^t \E\big(P^\kappa_sg\cdot P^{\kappa_n}_{t-s}f, \psi-\psi_n\big)\,ds\\
\end{eqnarray*}
for all $f,g\in L^2$. Thus
\begin{eqnarray*}
\Big|\int g \big(P^\kappa_tf-P^{\kappa_n}_tf)\,d\mm\Big|
&\le& \Lip(\psi-\psi_n)\cdot
\int_0^t \int_{\B_{n+1}(z)}\Gamma\big(P^\kappa_sg\cdot P^{\kappa_n}_{t-s}f\big)^{1/2}\,d\mm\,ds\\
&\le& \big(\Lip(\psi)+\|\psi\|_{L^\infty}\big)\cdot
\int_0^t \Big[\int_{\B_{n+1}(z)}\Gamma\big(P^\kappa_sg\big)+\big|P^\kappa_sg\big|^2
\,d\mm\Big]^{1/2}\\
&&\qquad\qquad\qquad\qquad\qquad
\cdot
\Big[\int_X\Gamma\big( P^{\kappa_n}_{t-s}f\big)+ \big|
 P^{\kappa_n}_{t-s}f\big|^2
\,d\mm\Big]^{1/2}
\,ds.
\end{eqnarray*}
Form boundedness of $\E^\kappa$ w.r.t.~$\E$ implies 
$$\int_{X}\Gamma\big(P^\kappa_sg\big)+\big|P^\kappa_sg\big|^2
\,d\mm\le C\, \Big(\E^\kappa(P^\kappa_sg)+\|P^\kappa_sg\|_{L^2}^2\Big)
$$
for all $s\in [0,t]$. Hence,
$\int_0^t \int_{X}\Gamma\big(P^\kappa_sg\big)+\big|P^\kappa_sg\big|^2
\,d\mm\,ds\le C_t\cdot \|g\|_{L^2}^2$
and thus
$$\int_0^t\int_{\B_{n+1}(z)}\Gamma\big(P^\kappa_sg\big)+\big|P^\kappa_sg\big|^2
\,d\mm\,ds\to 0$$ as $n\to\infty$. Similarly,
$$\int_0^t\int_X\Gamma\big( P^{\kappa_n}_{t-s}f\big)+ \big|
 P^{\kappa_n}_{t-s}f\big|^2
\,d\mm\le C\, \Big(\E^{\kappa_n}(P^{\kappa_n}_{t-s}g)+\|P^{\kappa_n}_{t-s}g\|_{L^2}^2\Big)
\le C_t\cdot \|g\|_{L^2}^2
$$
uniformly in $n$. This proves $\big|\int g \big(P^\kappa_tf-P^{\kappa_n}_tf)\,d\mm\big|\to0$ as $n\to\infty$ which is the claim.
\end{proof}

\begin{cor}  {\bf (i)}  Given $\phi\in L^\infty(X)$ and $\psi\in \Lip_b(X)$, the semigroup $(P^\kappa_t)_{t\ge0}$ for $\kappa=\phi-\underline\Delta\psi$ is given by
\begin{eqnarray}\label{double-pot}
P^{\kappa}_{t/2}f(x)&=&{\mathbb E}_x\big[e^{-\int_0^t\phi(B_s)ds + N^\psi_t}f(B_t)\,1_{\{t<\zeta\}}\big]
\end{eqnarray}
for each $f\in  \bigcup_{p\in[1,\infty]} L^p(X,\mm)$ and 
for $\mm$-a.e.~$x\in X$.

 {\bf (ii)} Even more, letting $\tilde P^{\kappa}_tf$ denote a quasi continuous version of $P^{\kappa}_tf$, then 
\begin{eqnarray*}
\tilde P^{\kappa}_{t/2}f(x)&=&{\mathbb E}_x\big[e^{-\int_0^t\phi(B_s)ds + N^\psi_t}f(B_t)\,1_{\{t<\zeta\}}\big]
\end{eqnarray*}
holds true for q.e.~$x\in X$.
\end{cor}

\begin{proof}  {\bf (i)} Define a semigroup $(Q_t)_{t\ge0}$ by the right hand side of \eqref{double-pot}, i.e. 
$Q_{t/2}f(x):={\mathbb E}_x\big[e^{-\int_0^t\phi(B_s)ds + N^\psi_t}f(B_t)\big]$.
We will prove that it is associated with the quadratic form $\E^\kappa$.
Put $\kappa_0=-\underline\Delta\psi$. 
From the probabilistic representations of $Q_tf$ and $P^{\kappa_0}f$, we easily deduce 
$$P^{\kappa_0}_tf-Q_tf=\int_0^t P^{\kappa_0}_s\,\big( \phi \, Q_{t-s}f\big)\,ds$$
(``Duhamel's formula'') and thus
$$\lim_{t\to0}\frac1t \int (f-Q_tf)f\,d\mm-\E^{\kappa_0}(f)=
\lim_{t\to0}\frac1t \int_0^t\int \big(P^{\kappa_0}_sf\big)\, \phi \, \big(Q_{t-s}f\big)\,d\mm\,ds=\int f^2\phi\,d\mm.$$
(Note that $\E^\kappa$ is obtained from $\E^{\kappa_0}$ by perturbation with a bounded potential. Hence, both $Q_t$ and $P^{\kappa_0}$ are strongly continuous semigroups on $L^2$.)
Therefore,
$$\lim_{t\to0}\frac1t \int (f-Q_tf)f\,d\mm=\E^{\kappa}(f)$$
for all $f\in\Dom(\E^\kappa)$ and thus $Q_tf=P^\kappa_t f$ for all $t$ and all $f$.

{\bf (ii)} follows by standard arguments for quasi-regular Dirichlet forms.
\end{proof}

\begin{remark} Throughout this subsection, the assumptions $\psi\in \Lip(X)$ (or $\psi\in \Lip_b(X)$)  always can be replaced by
$\psi\in W_*^{1,\infty}(X)$ (or $\psi\in W^{1,\infty}(X)$, resp.). In the Fukushima decomposition, then one has to choose a quasi-continuous version $\tilde\psi$ of $\psi$ to guarantee well-definedness of the contribution $\tilde \psi(B_t)-\tilde \psi(B_0)$.
\end{remark}

\subsection{Bochner Inequality \BE$_1(\kappa,\infty)$ and Gradient Estimate}
For $n\in\N$, we define the Hilbert space $V^n(X):=\big(-\Delta+1\big)^{-n/2}\big(L^2(X)\big)$ equipped with the norm 
 $$\big\|f\big\|_{V^n}:=\big\|(-\Delta+1)^{n/2}f\big\|_{L^2}.$$
 Of particular interest are the spaces 
 $V^1(X)=\Dom(\E)=W^{1,2}(X)$, $V^2(X)=\Dom(\Delta)$, and 
 $V^3(X)=\{f\in \Dom(\Delta):\ \Delta f\in Dom(\E)\}$.

\begin{Def} Given $\kappa\in  W^{-1,\infty}(X)$, we say that the \emph{Bochner inequality} or \emph{Barky-\'Emery condition} \BE$_1(\kappa,\infty)$ holds true if
$\Gamma(f)^{1/2}\in V^{1}(X)$ for all $f\in V^2(X)$
and if
\begin{equation}\label{H3-inequ}
-\int_X \Gamma(\Gamma(f)^{1/2},\phi) \,d\mm- \int_{\{ \Gamma(f)>0\}}
\Gamma(f)^{-1/2}\,\Gamma(f,\Delta f)\phi\,d\mm \geq
\big\langle \Gamma(f)^{1/2}\phi,\kappa\big\rangle_{W^{1,1+}, W^{-1,\infty}}.
\end{equation}
for all $f\in V^3(X)$ and all  nonnegative $\phi \in V^1(X)$.
\end{Def}
Note that equally well also the first integral in the above estimate can be restricted to the set $\{ \Gamma(f)>0\}$.

\begin{thm}\label{L1-grad} Given $\kappa\in  W^{-1,\infty}(X)$, the Bochner inequality \BE$_1(\kappa,\infty)$ is equivalent to the following gradient estimate \GE$_1(\kappa,\infty)$: \quad $\forall f\in V^{1}(X)$, $\forall t>0$:
\begin{equation}\label{grad-f*}
 \Gamma(P_{t}f)^{1/2}
 \le P^{\kappa}_t\big( \Gamma(f)^{1/2}\big)\quad\mm\text{-a.e. on }X
\end{equation}
\end{thm}

\begin{proof} 
{\bf a)} \ Assume that  \BE$_1(\kappa,\infty)$ holds true.
Put $\eta_\delta(r)=(r+\delta)^{1/2}-\delta^{1/2}$ as an approximation of $r^{1/2}$. For fixed $\delta>0$, nonnegative $\phi \in V^{1}(X)$. $f\in V^3(X)$, and $t>0$ consider 
$$s\mapsto a_s^{(\delta)}=\int\phi P^{\kappa}_s\big( \eta_\delta(\Gamma(P_{t-s}f))\big)\,d\mm$$
as an absolutely continuous function on $(0,t)$. Then
\begin{eqnarray*}
\lefteqn{\int\phi P^{\kappa}_t\big( \Gamma(f)^{1/2}\big)\,d\mm-\int\phi \, \Gamma(P_{t}f)^{1/2}\,d\mm}
\\
&=&
\lim_{\delta\to0}\int\phi P^{\kappa}_t\big( \eta_{\delta}(\Gamma(f))\big)\,d\mm-\int\phi \, \eta_{\delta}(\Gamma(P_{t}f))\,d\mm\\
&=&
\lim_{\delta\to0}\big[a^{(\delta)}_t-a^{(\delta)}_0\big]\\
&=&
-\lim_{\delta\to0}
\int_0^t\Big[ \E^{\kappa}(\phi_s, \eta_{\delta}(\Gamma(f_s))+2\int_X \phi_s \eta'_{\delta}(\Gamma(f_s))\,\Gamma(f_s,\Delta f_s)\,d\mm
\Big]ds\\
&=&-\int_0^t\Big[ \E^{\kappa}(\phi_s, \Gamma(f_s)^{1/2})+\int_{\{ \Gamma(f_s)>0\}}\phi_s \Gamma(f_s)^{-1/2}\,\Gamma(f_s,\Delta f_s)\,d\mm
\Big]ds
\end{eqnarray*}
where we have put $\phi_s=P^{\kappa}_s\phi$ and $f_s=P_{t-s}f$. 
The crucial point now is that the semigroup $(P_t)_{t\ge0}$ preserves the class 
$V^3(X)$ where $f$ is chosen from, and that the the semigroup $(P^\kappa_t)_{t\ge0}$ preserves the cone of nonnegative elements in 
$V^{1}(X)$ where $\phi$ is chosen from.
Assuming \BE$_1(\kappa,\infty)$ and applying it to $f_s$ and $\phi_s$ in the place of $f$ and $\phi$  implies that in the last integral the expression in $[...]$ is nonpositive. 
This proves the claimed gradient estimate \eqref{grad-f*} for $f\in V^3(X)$. The assertion for general $f$ then follows by approximation.
Indeed, each $f\in V^1(X)$ is approximated in $V^1$-norm by the sequence $f_n:=P_{1/n}f\in V^3$. Moreover, the map $V^1\to L^2, g\mapsto \Gamma(g)^{1/2}$ is continuous,  and so is $P^\kappa_t: L^2\to L^2$. Hence,
 $$\Gamma(P_{t}f)^{1/2} =\lim_n \Gamma(P_{t}f_n)^{1/2}
 \le P^{\kappa}_t\big( \Gamma(f_n)^{1/2}\big)
 = P^{\kappa}_t\big( \Gamma(f)^{1/2}\big).$$

{\bf b)} \ 
Now let us assume that the gradient estimate holds true.
Let us first derive the assertion on domain inclusion which in our formulation is requested for  \BE$_1(\kappa,\infty)$. Using the
 gradient estimate, we  conclude that
 \begin{equation}\label{hilfs}\frac1t\int_X \Big[\Gamma(f)^{1/2}-P^\kappa_t\Gamma(f)^{1/2}\Big]\, \Gamma(f)^{1/2}\,d\mm\le
 \frac1t\int_X \Big[\Gamma(f)^{1/2}-\Gamma(P_tf)^{1/2}\Big]\, \Gamma(f)^{1/2}\,d\mm.\end{equation}
By spectral calculus, it is well known that for $t\to0$ the LHS of \eqref{hilfs} converges monotonically to $\E^\kappa(\Gamma(f)^{1/2})$ and $$\Gamma(f)^{1/2}\in \D(\E^\kappa)\qquad\Longleftrightarrow\qquad\lim_{t\to0}
\frac1t\int_X \Big[\Gamma(f)^{1/2}-P^\kappa_t\Gamma(f)^{1/2}\Big]\, \Gamma(f)^{1/2}\,d\mm<\infty.$$
To deal with the RHS of \eqref{hilfs}, 
first observe  that $\Gamma(P_tf)^{1/2}\to \Gamma(f)^{1/2}$ as $t\to0$ since
$$\big|\Gamma(P_tf)-\Gamma(f)\Big|=\Big|\int_0^t 2\Gamma(P_sf,\Delta P_sf)ds\Big|\le2\int_0^t P_s^\kappa\Gamma(f)^{1/2}\cdot P_s^\kappa\Gamma(\Delta f)^{1/2}ds\le C\, t
$$
for $f\in V^3(X)$. Moreover,  obviously each of the integrals $\int_X\ldots d\mm$ in \eqref{hilfs} can be replaced by 
$\int_{\{\Gamma(f)>0\}}\ldots d\mm$.
But on the set $\{\Gamma(f)>0\}$, the chain rule for the $\Gamma$ operator yields
$$\frac1t \Big[\Gamma(f)^{1/2}-\Gamma(P_tf)^{1/2}\Big]\to \frac{-1}{\Gamma(f)^{1/2}}\,\Gamma(f,\Delta f).$$
Thus for $t\to0$, the RHS of \eqref{hilfs} converges as follows 
$$ \frac1t\int_X \Big[\Gamma(f)^{1/2}-\Gamma(P_tf)^{1/2}\Big]\, \Gamma(f)^{1/2}\,d\mm\to-
\int_{\{\Gamma(f)>0\}}\Gamma(f,\Delta f)\,d\mm=\int_X (\Delta f)^2\,d\mm.
$$
Combining the asymptotic results for both sides of \eqref{hilfs}, we obtain
$$\E^\kappa(\Gamma(f)^{1/2})\le \|\Delta f\|_{L^2}^2<\infty,$$
in particular 
 $\Gamma(f)^{1/2}\in\Dom(\E^\kappa)=V^{1}(X)$
 for $f\in V^3(X) $. 
 Since the class of these $f$'s  is dense in $V^2(X)=\Dom(\Delta)$, it follows that $\Gamma(f)^{1/2}\in\Dom(\E^\kappa)=W^{1,2}(X)$
 for all $f\in \Dom(\Delta)$. This yields the  domain assertion requested for \BE$_1(\kappa,\infty)$.

{\bf c)} \ 
 To derive the requested functional inequality for \BE$_1(\kappa,\infty)$, we  integrate the gradient estimate for $f\in V^3(X)$ w.r.t.~$\phi_\delta\,d\mm$ and subtract  $\int \phi_\delta\,
\Gamma(f)^{1/2}\,d\mm$ on both sides.
Here for arbitrary $\phi\in V^1(X)$ and $\delta>0$, we put
$$\phi_\delta=\phi\cdot\frac{\Gamma(f)^{1/2}}{\Gamma(f)^{1/2}+\delta}.$$  This yields
\begin{eqnarray*}\frac1t\int_X \Big[\Gamma(f)^{1/2}-P^\kappa_t\Gamma(f)^{1/2}\Big]\, \phi_\delta\,d\mm\le
 \frac1t\int_X \Big[\Gamma(f)^{1/2}-\Gamma(P_tf)^{1/2}\Big]\, \phi_\delta\,d\mm.\
\end{eqnarray*}
In the limit $t\to0$, this gives
$$\E^\kappa\big(\Gamma(f)^{1/2}, \phi_\delta\big)\le -
\int_{\{\Gamma(f)>0\}}\frac1{\Gamma(f)^{1/2}}\Gamma(f,\Delta f)\,\phi_\delta\,d\mm.$$
One easily verifies that for $\delta\to0$ this converges to
$$\E^\kappa\big(\Gamma(f)^{1/2}, \phi\big)\le -
\int_{\{\Gamma(f)>0\}}\frac1{\Gamma(f)^{1/2}}\Gamma(f,\Delta f)\,\phi\,d\mm$$
which is the claim.
\end{proof}

\begin{cor} The  Bochner inequality \BE$_1(\kappa,\infty)$ with $\kappa\in  W^{-1,\infty}(X)$ implies 
$$\E(\Gamma(f)^{1/2})\le \|\Delta f\|_{L^2}^2-\big\langle \Gamma(f),\kappa\big\rangle_{W^{1,1+}, W^{-1,\infty}}$$
for all $f\in \Dom(\Delta)$ and thus 
$$\E(\Gamma(f)^{1/2})\le\frac1{1-\delta}\Big( \|\Delta f\|_{L^2}^2+(C+C^2/\delta)\cdot \E(f)\Big)$$
for each $\delta\in (0,1)$ with $C:=\|\kappa\|_{W^{-1,\infty}}$.
\end{cor}

\paragraph{Localization.}
In the sequel, we will also  localize various statements. Ding this will require some care since in general $X$ will not be  locally compact. Given a space $\mathcal G(X)$ of functions  (or of $\mm$-equivalence classes of functions) on $X$ we denote by  $\mathcal G_\sloc(X)$ the set of all functions $g$ (or   $\mm$-equivalence classes of functions, resp.) on $X$ ``which semi-locally lie in $\mathcal G(X)$'' in the sense  that for each bounded open subset $B\subset X$ there exists a $g_B\in \mathcal G(X)$  such that $g=g_B$ on $B$ (or $m$-a.e.\ on $B$, resp.). This way, e.g.\ we define the spaces
 $W^{1,1+}_\sloc(X)$.
 
We denote by $W^{-1,\infty}_\sloc(X)$ the set of all $\kappa$ such that for all bounded open sets $B\subset X$ there exist $\kappa_B\in W^{-1,\infty}(X)$ which are consistent in the sense that 
$\langle \phi,\kappa_B\rangle_{W^{1,1+},W^{-1,\infty}} =\langle \phi,\kappa_{B'}\rangle_{W^{1,1+},W^{-1,\infty}}$  for all such $B,B'$ and for all $\phi\in W^{1,1+}(X)$ with support in $B\cap B'$.
In this case, we say that $\kappa=\kappa_B$ on $B$ and put
$$\langle \phi,\kappa\rangle_{W_{bs}^{1,1+},W_\sloc^{-1,\infty}} :=\langle \phi,\kappa_{B}\rangle_{W^{1,1+},W^{-1,\infty}}$$
provided $\phi$ is supported in $B$.

\begin{lma}\label{H1-Bochner-N} The Bochner inequality \BE$_1(\kappa,\infty)$ is equivalent to
the fact that
$\Gamma(f)^{1/2}\in V^{1}(X)$ for all $f\in V^2(X)$
and 
\begin{equation}\label{H3-inequ-loc}
-\int_X \Gamma(\Gamma(f)^{1/2},\phi) + \Gamma(f)^{-1/2}\Gamma(f,\Delta f)\phi\,d\mm \geq
\big\langle \Gamma(f)^{1/2}\phi,\kappa\big\rangle_{W^{1,1+}, W^{-1,\infty}}
\end{equation}
for all $f\in V^3_\sloc(X)$ and all  nonnegative $\phi \in V_{bs}^{1}(X)$.
\end{lma}

The domain inclusion requested for \BE$_1(\kappa,\infty)$ obviously implies the inclusion for the localized domains: $\Gamma(f)^{1/2}\in V_\sloc^{1}(X)$ for all $f\in V^2_\sloc(X)$. 
\begin{proof}
``$\Rightarrow$'': Given $\phi \in V_{bs}^{1}(X)$ and $f\in V^3_\sloc(X)$, there exists bounded open $B\subset X$ such that $\phi=0$ on $X\setminus B$ and there exists $f_B\in V^3(X)$ with $f=f_B$   on $B$.
Applying \eqref{H3-inequ} to $f_B$ and $\phi$ implies \eqref{H3-inequ-loc} for $f$ and $\phi$.

``$\Leftarrow$'': Given nonnegative $\phi \in V^{1}(X)$, by partition of unity we can find countably many nonnegative $\phi_n \in V_{bs}^{1}(X)$ such that $\phi=\sum_n\phi_n$.  Applying \eqref{H3-inequ-loc} to each $\phi_n$ and the given 
$f\in V^3(X)$, and adding up these estimates yields  \eqref{H3-inequ} for the given $\phi$ and $f$.
\end{proof}

\section{Equivalence of \BE$_2(k,N)$ and \CD$(k,N)$}

Throughout this section, $(X,\d,\mm)$ will be a metric measure space,  $N\in[1,\infty)$ a number,  
and $k\colon X \to\R$
will be a bounded, lower semicontinuous 
 function.
We present the Eulerian and the Lagrangian  characterizations of \emph{``Ricci curvature at $x$ bounded from below by $k(x)$ and dimension bounded from above by $N$''} and prove their equivalence. Put $K_0=\inf_x k(x)$ and $K_1=\sup_x k(x)$.

Without loss of generality, we will assume that $(X,\d,\mm)$ satisfies 
the Riemannian curvature-dimension condition \RCD$(K,\infty)$ for some constant $K\in\R$.
Among others, this will guarantee that the space is infinitesimally Hilbertian, that the volume of balls does not grow faster than $e^{Cr^2}$, and that functions with bounded gradients have Lip\-schitz continuous versions (``Sobolev-to-Lipschitz property''). Moreover, it implies that $\Gamma(u)^{1/2}\in\Dom(\E)$ for each $u\in\Dom(\Delta)$.

In the sequel, as usual $\Prob_2(X)$ will denote the space of probability measures $\mu$ on $X$ with $\int\d^2(.,z)\,d\mu<\infty$ equipped with the $L^2$-Kantorovich-Wasserstein distance $W_2$.
We say that a measure $\boldsymbol{\pi}\in\Prob(\Geo(X))$ \emph{represents the $W_2$-geodesic  $(\mu_r)_{r\in [0,1]}$} if $\mu_r = (e_r)_\sharp\boldsymbol{\pi}$ for $r\in
[0,1]$. Here $\Geo(X)$ denotes the set of  $\d$-geodesics $\gamma:[0,1]\to X$ and $e_r: \Geo(X)\to X, t\mapsto\gamma_r$ denotes the projection or evaluation operator.

Thanks to our a priori assumption \RCD$(K,\infty)$, there exists a heat kernel $\big(p_t(x,y)\big)_{x,y\in X, t\ge0}$ on $X$ such that 
$$P_tf(x):=\int f(y)\,p_t(x,y)\,d\mm(y)$$
defines a strongly continuous, non expanding semigroup in $L^p(X,\mm)$ for each $p\in[1,\infty)$. For $p=2$, this actually can be defined (or re-interpreted)  as  the gradient flow for the energy $\E$ in $L^2(X,\mm)$.
Moreover, 
$$d\DP_t\mu(y):=\Big[\int p_t(x,y)\,d\mu(x)\Big]\,d\mm(y)$$
defines a semigroup on $\Prob_2(X)$. The latter can be equivalently regarded as the gradient flow for the Boltzmann entropy $\Ent$ in the Wasserstein space $\Prob_2(X)$. Here and in the sequel, $\Ent(\mu):=\int u\log u\,d\mm$ if $\mu=u\,\mm$ and $\Ent(\mu):=\infty$ if $\mu$ is not absolutely continuous w.r.t~$\mm$.

\begin{Def}\label{def-cd-var} We say that a metric measure
space $(X,\d,\mm)$ satisfies the \emph{curvature-dimension
condition}  with variable curvature bound $k$ and dimension bound $N$, briefly \CD$(k,N)$, if for every
$\mu_0,\mu_1\in\Prob_2(X)\cap\Dom(\Ent)$ there exists a
measure $\boldsymbol{\pi}\in\Prob(\Geo(X))$ representing some 
$W_2$-geodesic $(\mu_r)_{r\in [0,1]}$ connecting $\mu_0$ and $\mu_1$
 such that 
\begin{equation*}
\frac{d}{dr}\Ent(\mu_r)\Big|_{r=1-}-\frac{d}{dr}\Ent(\mu_r)\Big|_{r=0+}
\ge
   \int_0^1\int_{\Geo(X)} k(\gamma_r)\vert\dot{\gamma}\vert^2\,d\boldsymbol{\pi}(\gamma)\,d r+\frac1N\Big[\Ent(\mu_1)-\Ent(\mu_0)\Big]^2.
\end{equation*}
\end{Def}

The $(k,N)$-convexity of the entropy allows for various straightforward  reformulations, cf.
 \cite{sturm2018b}.

\begin{lma}\label{threeCDs} The following are equivalent:
\begin{itemize}
\item[(i)] the mm-space satisfies \CD$(k,N)$;
\item[(ii)] for every
$\mu_0,\mu_1\in\Prob_2(X)\cap\Dom(\Ent)$ there exists a
measure $\boldsymbol{\pi}\in\Prob(\Geo(X))$ representing some 
$W_2$-geodesic $(\mu_r)_{r\in [0,1]}$ connecting $\mu_0$ and $\mu_1$
 such that 
\begin{equation*}
\frac{d}{dr}\Ent(\mu_r)\Big|_{r=1-}-\frac{d}{dr}\Ent(\mu_r)\Big|_{r=0+}
\ge
   \int_0^1\int_{\Geo(X)} k(\gamma_r)\vert\dot{\gamma}\vert^2\,d\boldsymbol{\pi}(\gamma)\,d r+\frac1N \int_0^1\Big[\frac d{dr}\Ent(\mu_r)\Big]^2dr;
\end{equation*}
\item[(iii)] for every
$\mu_0,\mu_1\in\Prob_2(X)\cap\Dom(\Ent)$ there exists a
measure $\boldsymbol{\pi}\in\Prob(\Geo(X))$ representing some 
$W_2$-geodesic $(\mu_r)_{r\in [0,1]}$ connecting $\mu_0$ and $\mu_1$
 such that for all $r\in (0,1)$
 \begin{eqnarray*} \Ent(\mu_r)&\le&
(1-r)\,\Ent(\mu_0)+r\,\Ent(\mu_1)\\
&&-
   \int_0^1 g(r,s)\,\Big(\int_{\Geo(X)}k(\gamma_s)\vert\dot{\gamma}\vert^2\,d\boldsymbol{\pi}(\gamma)+\frac1N \,\Big[\frac d{ds}\Ent(\mu_s)\Big]^2\Big)ds
\end{eqnarray*}
where $g(.,.)$ denotes the Green function on $[0,1]$.
\end{itemize}
Moreover, in (ii) as well as in (iii), the phrase ``for every
$\mu_0,\mu_1\in\Prob_2(X)\cap\Dom(\Ent)$ there exists a
measure $\boldsymbol{\pi}\in\Prob(\Geo(X))$ representing some 
$W_2$-geodesic $(\mu_r)_{r\in [0,1]}$ connecting $\mu_0$ and $\mu_1$
 such that \ldots'' can equivalently be replaced by ``for every measure $\boldsymbol{\pi}\in\Prob(\Geo(X))$ representing some 
$W_2$-geodesic $(\mu_r)_{r\in [0,1]}$ with endpoints $\mu_0,\mu_1$ of finite entropy \ldots''.
\end{lma}

\begin{proof} Firstly note that the addendum follows from the uniqueness of the measure representing a $W_2$-geodesics connecting a given pair of  measures of finite entropy \cite{giglirajala2016}.

(ii) $\Rightarrow$ (i): Trivial since $\int_0^1\big[\frac d{ds}\Ent(\mu_s)\big]^2ds\ge \big[\Ent(\mu_1)-\Ent(\mu_0)\big]^2$.

(i) $\Rightarrow$ (iii): Given a $W_2$-geodesic $(\mu_r)_{r\in [0,1]}$ and its representing measure $\boldsymbol{\pi}\in\Prob(\Geo(X))$,  apply (i) to $\mu_s,\mu_{s+\delta}$ in the place of $\mu_0,\mu_1$ to deduce for a.e.\ $s\in(0,1)$
\begin{equation*}
\frac{d^2}{ds^2}\Ent(\mu_s)
\ge
   \int_{\Geo(X)} k(\gamma_s)\vert\dot{\gamma}\vert^2\,d\boldsymbol{\pi}(\gamma)+\frac1N\Big[\frac d{ds}\Ent(\mu_s)\Big]^2.
\end{equation*}
(where the LHS has to be understood as the distributional second derivative of a semiconvex function).
Integrating this w.r.t.\ the measure $g(s,r)\,ds$ on $(0,1)$ yields (iii).

(iii) $\Rightarrow$ (ii): Given a $W_2$-geodesic $(\mu_r)_{r\in [0,1]}$ and its representing measure $\boldsymbol{\pi}\in\Prob(\Geo(X))$, we add up the estimate (iii) together with its counterpart with $1-r$ in the place of $r$ to obtain
\begin{eqnarray*} \lefteqn{\Ent(\mu_r)+\Ent(\mu_{1-r})\le
\Ent(\mu_0)+\Ent(\mu_1)}\\
&&-
   \int_0^1 \big[g(r,s)+g(1-r,s)\big]\,\Big(\int_{\Geo(X)}k(\gamma_s)\vert\dot{\gamma}\vert^2\,d\boldsymbol{\pi}(\gamma)+\frac1N \,\Big[\frac d{ds}\Ent(\mu_s)\Big]^2\Big)ds.
\end{eqnarray*}
Dividing by $r$ and then letting $r\to0$ yields (ii).
\end{proof}

\begin{Def} We say that
$(X,\d,\mm)$ satisfies the   \emph{2-Bochner inequality} or \emph{2-Bakry-\'Emery estimate} with variable curvature bound $k$ and dimension bound $N$, briefly  \BE$_2(k,N)$, if
\begin{equation*}
\int_X \frac{1}{2}\Gamma(f)\Delta\phi - \Gamma(f,\Delta f)\phi\,d\mm \geq \int_X\big[ k\Gamma(f)+\frac1N (\Delta f)^2\big]\phi\,d\mm
\end{equation*}
for all $f\in\Dom(\Delta)$ with $\Delta f \in \Dom(\E)$ and all nonnegative $\phi \in \Dom(\Delta)\cap L^\infty(X,\mm)$ with $\Delta\phi\in L^\infty(X,\mm)$.
\end{Def}

Our first main results states that also for variable curvature bound $k$ and finite $N$, the Eulerian and Lagrangian approaches to synthetic lower Ricci bounds are equivalent.
For constant $k$, this has been proven in joint work  \cite{erbar2015} of the author with Erbar and Kuwada. For variable $k$ and $N=\infty$, it has been proven in joint work \cite{braun2019} with Braun and Habermann. In particular, in the latter work a formulation of the transport estimate has been given in terms of the following quantity:
$$W_{2,k}(\mu,\nu,t):=\inf_{(B^1,B^2)}\, {\mathbb E}\Big[e^{-2\int_0^tk(B^1_{2s}, B^2_{2s})ds}\cdot d^2(B^1_{2t},B^2_{2 t})\Big]^{1/2}
$$
where the infimum is taken over all coupled pairs of Brownian motions $(B^1_s)_{0\le s\le 2t}$ and $(B^2_s)_{0\le s\le 2\lambda t}$ with initial distributions $\mu$ and $\nu$, resp.

\begin{thm}\label{cd=be} The following are equivalent
\begin{itemize}
\item[(i)] the curvature-dimension condition \CD$(k,N)$

\item[(ii)] the evolution-variational inequality \EVI$(k,N)$: \
for all $\mu_0,\mu_1\in{\mathcal P}_2(X)$ with finite entropy and for $\boldsymbol{\pi}\in
{\mathcal P}(\Geo(X))$ representing the unique $W_2$-geodesic connecting them:
\begin{eqnarray*}-\frac12\frac{d^+}{dt}\Big|_{t=0}W_2(\DP_t\mu_0,\mu_1)^2&\ge&
\Ent(\mu_0)-\Ent(\mu_1)\\
&+&
  \int_0^1(1-r)\,\Big(\int_{\Geo(X)}
   k(\gamma_r)\vert\dot{\gamma}\vert^2\,d\boldsymbol{\pi}(\gamma)+\frac1N
    \Big[\frac d{dr}\Ent(\mu_r)\Big]^2\Big)
  dr
\end{eqnarray*}

\item[(iii)] the differential transport  estimate \DTE$_2(k,N)$: \
for all $\mu_0,\mu_1\in{\mathcal P}_2(X)$ with finite entropy and for $\boldsymbol{\pi}\in
{\mathcal P}(\Geo(X))$ representing the unique $W_2$-geodesic connecting them:
$$-\frac12\frac{d^+}{dt}\Big|_{t=0}W_2(\DP_t\mu_0,\DP_t\mu_1)^2\ge
  \int_0^1\int_{\Geo(X)} k(\gamma_r)\vert\dot{\gamma}\vert^2\,d\boldsymbol{\pi}(\gamma)\,d r+\frac1N\Big[\Ent(\mu_0)-\Ent(\mu_1)\Big]^2
$$

\item[(iv)] the transport  estimate \TE$_2(k,N)$: \ for all $\mu_0,\mu_1\in{\mathcal P}_2(X)$ and all $0\le s\le t$:
$$W_{2,k}(\mu_0,\mu_1,t)^2\le W_{2,k}(\mu_0,\mu_1,s)^2-\frac1N\int_s^t\Big| \Ent(\DP_r\mu_0)-\Ent(\DP_r\mu_1)\Big|^2dr$$

\item[(v)] the gradient estimate \GE$_2(k,N)$: \ for all $f\in \Dom(\E)$ and all $t>0$:
$$\Gamma(P_tf)+\frac{2t}Ne^{-2K_1t}\big(\Delta P_tf\big)^2\le P_t^{2k}\Gamma(f)$$
\item[(vi)] 
 the Bochner inequality \BE$_2(k,N)$.
\end{itemize}
Here and henceforth, $\frac{d^+}{dt}f(t):=\limsup_{h\to 0}(f(t+h)-f(t))/h$ 
denotes the upper derivative.
\end{thm}

\begin{proof}
{\bf (i) $\Rightarrow$ (ii):} \ Using the equivalent \CD$(k,N)$ formulation from the previous Lemma \ref{threeCDs}(iii) and passing there to the limit $r\to0$, one easily sees that 
 (i) implies 
\begin{eqnarray*}\frac{d^+}{dr}\Big|_{r=0}\,\Ent(\mu_r)&\le& \Ent(\mu_1)-\Ent(\mu_0)\\
&&- \int_0^1\int   (1-r)\,k(\gamma_r)\vert\dot{\gamma}\vert^2_r\,d\boldsymbol{\pi}(\gamma)\,d r+\frac1N
   \int_0^1 (1-r)\, \Big[\frac d{dr}\Ent(\mu_r)\Big]^2
  dr.
\end{eqnarray*}
Thus the claim  (ii) is an immediate consequence of the fact  that
$$\frac{d^+}{dr}\Big|_{r=0}\,\Ent(\mu_r)\ge \frac12\frac{d^+}{dt}\Big|_{t=0}W_2(\DP_t\mu_0,\mu_1)^2,$$
\cite{ambrosiogigli2015}, Thm. 6.3.

\medskip

{\bf (ii) $\Rightarrow$ (i).} \ We follow the standard path of argumentation. Given two probability measures $\mu_0,\mu_1$ of finite entropy, let $(\mu_r)_{r\in[0,1]}$, represented by $\boldsymbol{\pi}$, denote the unique $W_2$ geodesic connecting them and note that the standing \CD$(K,\infty)$-assumption implies that the $\mu_r$'s also have finite entropy. Consider the heat flow starting in $\mu_r$ with observation point $\mu_0$ as well as with observation point $\mu_1$.  Note that $$0 \le\frac1r \frac{d^+}{dt}\Big|_{t=0}W_2(\mu_0,\DP_t\mu_r)^2+\frac1{1-r}\frac{d^+}{dt}\Big|_{t=0}W_2(\DP_t\mu_r,\mu_1)^2$$
since
$W_2(\mu_0,\mu_1)^2=\frac1r W_2(\mu_0,\mu_r)^2+\frac1{1-r}W_2(\mu_r,\mu_{1})^2$
whereas
$W_2(\mu_0,\mu_1)^2\le\frac1r W_2(\mu_0,\DP_t\mu_r)^2+\frac1{1-r}W_2(\DP_t\mu_r,\mu_1)^2$.
Applying the \EVI$(k,N)$ with 
$\mu_r,\mu_0$ as well as with 
$\mu_r,\mu_1$ in the place of $\mu_0,\mu_1$ thus yields
\begin{eqnarray*}
0 &\le&\frac{1-r}2 \, \frac{d^+}{dt}\Big|_{t=0}W_2(\mu_0,\DP_t\mu_r)^2+\frac r2 \, \frac{d^+}{dt}\Big|_{t=0}W_2(\DP_t\mu_r,\mu_1)^2\\
&\le&
(1-r)\Big[\Ent(\mu_r)-\Ent(\mu_0)\\
&&\qquad-
  \int_0^r s\,\Big(\int_{\Geo(X)}
   k(\gamma_s)\vert\dot{\gamma}\vert^2\,d\boldsymbol{\pi}(\gamma)+\frac1N
    \Big[\frac d{ds}\Ent(\mu_s)\Big]^2\Big)
  ds\Big]\\
&+&  r\Big[\Ent(\mu_r)-\Ent(\mu_1)\\
&&\qquad-
  \int_r^1(1-s)\,\Big(\int_{\Geo(X)}
   k(\gamma_s)\vert\dot{\gamma}\vert^2\,d\boldsymbol{\pi}(\gamma)+\frac1N
    \Big[\frac d{ds}\Ent(\mu_s)\Big]^2\Big)
  ds\Big]\\
  &=&
  \Ent(\mu_r)-(1-r)\,\Ent(\mu_0)-r\,\Ent(\mu_1)\\
 && \quad- \int_0^1 g(r,s)\,\Big(\int_{\Geo(X)}k(\gamma_s)\vert\dot{\gamma}\vert^2\,d\boldsymbol{\pi}(\gamma)+\frac1N \,\Big[\frac d{ds}\Ent(\mu_s)\Big]^2\Big)ds.
\end{eqnarray*}
This proves the \CD$(k,N)$-estimate.

\medskip

{\bf (i) $\Rightarrow$ (iii):} \ 
For $t>0$ let $\phi_t,\psi_t$ denote a $W_2$-optimal pair of Kantorovich potentials for the transport from $\DP_t\mu_0=u_t\,\mm$ to $\DP_t\mu_1=v_t\,\mm$. Then following \cite{ambrosiogigli2015}, Thm.~6.3 and Thm.~6.5,  by Kantorovich duality
for a.e.~$t>0$
\begin{eqnarray*}
\frac{d^+}{dt}\,\frac12 W_2(\DP_t\mu_0,\DP_t\mu_1)^2&=&
\lim_{s\to t}\frac1{t-s}\int\big[\phi_t(u_t-u_s)+\psi_t(v_t-v_s)\big]dm\\
&=&-\E(\phi_t,u_t)-\E(\psi_t,v_t).
\end{eqnarray*}
Moreover, 
$$-\E(\phi_t,u_t)\le\frac{d^+}{dr}\Ent(\mu^t_r)\Big|_{r=0}, \quad -\E(\phi_t,u_t)\le-\frac{d^+}{dr}\Ent(\mu^t_r)\Big|_{r=1}.$$
Thus
\begin{eqnarray}\label{W-vs-Ent}
\frac{d^+}{dt}\,\frac12 W_2(\DP_t\mu_0,\DP_t\mu_1)^2\le 
\frac{d^+}{dr}\Ent(\mu^t_r)\Big|_{r=0}-\frac{d^+}{dr}\Ent(\mu^t_r)\Big|_{r=1}
\end{eqnarray}
where $(\mu^t_r)_{r\in[0,1]}$, represented by $\boldsymbol{\pi}^t$,
 denotes the $W_2$-geodesic connecting $\mu_0^t:=\DP_t\mu_0$ and $\mu_1^t:=\DP_t\mu_1$.
Together with (i) this implies
\begin{eqnarray*}
\frac{d^+}{dt}\,\frac12 W_2(\DP_t\mu,\DP_t\nu)^2&\le&
 - \int_0^1\int_{\Geo(X)} k(\gamma_r)\vert\dot{\gamma}\vert^2\,d\boldsymbol{\pi}^t(\gamma)\,d r+\frac1N\Big[\Ent(\DP_t\mu_0)-\Ent(\DP_t\mu_1)\Big]^2
  \end{eqnarray*}
 for a.e.\ $t$ and thus
  \begin{eqnarray*}
\frac1{s} \Big[\frac12W_2(\DP_s\mu,\DP_s\nu)^2-\frac12W_2(\mu_0,\mu_1)^2\Big]&\le&
 -\frac1s\int_0^s \Big(\int_0^1\int_{\Geo(X)} k(\gamma_r)\vert\dot{\gamma}\vert^2\,d\boldsymbol{\pi}^t(\gamma)\,d r\\
 &&\qquad\qquad+\frac1N\Big[\Ent(\DP_t\mu_0)-\Ent(\DP_t\mu_1)\Big]^2\Big)dt.
  \end{eqnarray*}
  Passing to the limit $s\to0$ finally yields the claim (iii) since $\Ent(\DP_t\mu_0)$ as well as $\Ent(\DP_t\mu_0)$ are  continuous in $t$ and since
  $\boldsymbol{\pi}^t$ weakly converges to $\boldsymbol{\pi}$ and $k$ is lower semicontinuous.
\medskip

{\bf (iii)$_{loc}$ $\Rightarrow$ (i).} \
This implication  can be proven with the ``trapezial argument'' from \cite{kopfer2018}. Note that thanks to the local-to-global property of the \CD$(k,N)$-condition, for this implication it suffices that the differential transport inequality holds locally, that is, for each $z\in X$ there exist $\delta>0$ such that
\DTE$(k,N)$
holds true for all $\mu_0,\mu_1$ which are supported in $\B_\delta(z)$.

Given $\mu_0,\mu_1$ of finite entropy and $\epsilon\in(0,\frac12)$ as well as $t>0$,
note that
$$W_2(\mu_0,\mu_1)^2=\frac1\epsilon W_2(\mu_0,\mu_\epsilon)^2+\frac1{1-2\epsilon}W_2(\mu_\epsilon,\mu_{1-\epsilon})^2+\frac1\epsilon W_2(\mu_{1-\epsilon},\mu_1)^2$$
whereas
$$W_2(\mu_0,\mu_1)^2\le\frac1\epsilon W_2(\mu_0,\DP_t\mu_\epsilon)^2+\frac1{1-2\epsilon}W_2(\DP_t\mu_\epsilon,\DP_t\mu_{1-\epsilon})^2+\frac1\epsilon W_2(\DP_t\mu_{1-\epsilon},\mu_1)^2.$$
Thus
$$0 \le\frac1\epsilon \frac{d^+}{dt}W_2(\mu_0,\DP_t\mu_\epsilon)^2+\frac1{1-2\epsilon}\frac{d^+}{dt}W_2(\DP_t\mu_\epsilon,\DP_t\mu_{1-\epsilon})^2+\frac1\epsilon \frac{d^+}{dt}W_2(\DP_t\mu_{1-\epsilon},\mu_1)^2.$$
Estimating the first and third term on the RHS by means of \EVI$(K,\infty)$ (which is true as consequence of our standing a priori assumption) and the second term by means of \DTE$(k,N)$ yields
\begin{eqnarray*}
0&\le&
\frac2\epsilon\Big[\Ent(\mu_\epsilon)-\Ent(\mu_0)-K\,W_2(\mu_0,\mu_\epsilon)^2\Big]\\
&&
-\frac2{1-2\epsilon}\Big[(1-2\epsilon)^2
 \int_\epsilon^{1-\epsilon}\int_{\Geo(X)} k(\gamma_r)\vert\dot{\gamma}\vert^2_r\,d\boldsymbol{\pi}(\gamma)\,d r+\frac1N\Big[\Ent(\mu_\epsilon)-\Ent(\mu_{1-\epsilon})\Big]^2
\Big]\\
&&
+\frac2\epsilon\Big[\Ent(\mu_{1-\epsilon})-\Ent(\mu_1)-K \, W_2(\mu_{1-\epsilon},\mu_1)^2\Big].
\end{eqnarray*}
In the limit $\epsilon\to0$, this gives the \CD$(k,N)$-inequality (i).

\medskip

{\bf (iii) $\Leftrightarrow$ (iv).} \ 
The proof of this equivalence follows the argumentation for proving  Theorem 5.6 and Corollary 5.7 in \cite{braun2019}.
\medskip

{\bf (v) $\Rightarrow$ (iii)$_{loc}$:} \
This follows similar as in the proof of Theorem 5.16 in
 \cite{braun2019} from a localization argument.
 
 \medskip

{\bf (v) $\Leftrightarrow$ (vi):} \ The proof follows  the standard line of argumentation via differentiating the forward-backward evolution. More precisely, for bounded, nonnegative $\phi\in\Dom(\E)$ and fixed $t>0$, put
$a(s):=\int \phi\, P^{2k}_s\Gamma\big(P_{t-s}f\big)\,dm$. This function is absolutely continuous in $s$ with 
\begin{eqnarray*}
a'(s)&=&\int \phi_s\, \big[ (\Delta-2k)\Gamma(f_s)-2\Gamma(f_s,\Delta f_s)\big]\,d\mm
\end{eqnarray*}
for a.e. $s\in[0,t]$  where we have put $\phi_s:=P^{2k}_s\phi$ and $f_s:=P_{t-s}f$.
Assuming (vi)  implies
$$a'(s)\ge
\frac2N\int \phi_s\,(\Delta f_s)^2\,d\mm$$
and thus
\begin{eqnarray*}
\int\phi\big( P^{2k}_t\Gamma(f)-\Gamma(P_tf)\big)\,d\mm
&=&
a(t)-a(0)\\
&\ge& \frac2N\int_0^t\int \phi\,P^{2k}_s(\Delta P_{t-s}f)^2\,d\mm\\
&\ge& \frac{2}Ne^{-2K_1t}\int_0^t\int \phi\,(P_s\Delta P_{t-s}f)^2\,d\mm\\
&=& \frac{2t}Ne^{-2K_1t}\int \phi\,(\Delta P_{t}f)^2\,d\mm.
\end{eqnarray*}
Varying over $\phi$, this yields (v). Conversely, assuming (v) yields
\begin{eqnarray*}
\frac{2}N\int \phi\,(\Delta f)^2\,d\mm&=&
\lim_{t\to0}\frac1t  \Big[\frac{2t}Ne^{-2K_1t}\int \phi\,(\Delta P_{t}f)^2\,d\mm\Big]\\
&\le&
\lim_{t\to0}\frac1t\Big[\int\phi\big( P^{2k}_t\Gamma(f)-\Gamma(P_tf)\big)\,d\mm\Big]\\
&=&
\int\phi\, \big[ (\Delta-2k)\Gamma(f)-2\Gamma(f,\Delta f)\big]\,d\mm
\end{eqnarray*}
for all bounded nonnegative $\phi\in\Dom(\E)$ and all sufficiently regular $f$.

 \medskip
 
{\bf  (i) $\Rightarrow$ (vi).} \ 
We will first derive an estimate of the form (4.2) in \cite{erbar2015} for $W_2(\DP_t\mu,\DP_s\nu)$. 
Given measures  $\mu,\nu\in\Prob_2(X)$ of finite entropy and numbers $\lambda, t>0$ we can estimate similar as in \eqref{W-vs-Ent}
\begin{eqnarray*}
\frac{d^+}{dt}\,\frac12 W_2(\DP_t\mu,\DP_{\lambda t}\nu)^2\le 
\frac{d^+}{dr}\Ent(\mu^t_r)\Big|_{r=0}-\lambda\frac{d^+}{dr}\Ent(\mu^t_r)\Big|_{r=1}.
\end{eqnarray*}
From Lemma \ref{threeCDs} we easily deduce
\begin{eqnarray*}
\lefteqn{
\frac{d^+}{dr}\Ent(\mu^t_r)\Big|_{r=0}-\lambda\frac{d^+}{dr}\Ent(\mu^t_r)\Big|_{r=1}
\le (\lambda-1)\cdot \Big(\Ent(\mu^t_0)-\Ent(\mu^t_1)\Big)}\\
&&-\int_0^1 \big[1-r+\lambda r\big]\cdot 
\Big(\int_{\Geo(X)}k(\gamma_r)\vert\dot{\gamma}\vert^2\,d\boldsymbol{\pi}_t^\lambda(\gamma)+\frac1N \,\Big[\frac d{dr}\Ent(\mu_r^t)\Big]^2\Big)dr
\end{eqnarray*}
where $\boldsymbol{\pi}_t^\lambda$ denotes the measure on  $\Prob(\Geo(X))$ representing the geodesic $(\mu_r^t)_{r\in[0,1]}$ from $\DP_t\mu$ to $\DP_{\lambda t}\nu$. Adding up these inequalities and using Young's inequality we obtain
\begin{eqnarray*}
\frac{d^+}{dt}\,\frac12 W_2(\DP_t\mu,\DP_{\lambda t}\nu)^2&\le&
(\lambda-1)\cdot \Big(\Ent(\mu^t_0)-\Ent(\mu^t_1)\Big)\\
&&-\int_0^1 [1-r+\lambda r]\,
\Big(\int_{\Geo(X)}k(\gamma_r)\vert\dot{\gamma}\vert^2\,d\boldsymbol{\pi}_t^\lambda(\gamma)+\frac1N \,\Big[\frac d{dr}\Ent(\mu_r^t)\Big]^2\Big)dr\\
&\le&-\int_0^1 [1-r+\lambda r]\,
\int_{\Geo(X)}k(\gamma_r)\vert\dot{\gamma}\vert^2\,d\boldsymbol{\pi}_t^\lambda(\gamma)dr\\
&&+\frac N4 (\lambda-1)^2\cdot \int_0^1 \frac1{1-r+\lambda r}dr\\
&=&-\int_0^1 [1-r+\lambda r]\,
\int_{\Geo(X)}k(\gamma_r)\vert\dot{\gamma}\vert^2\,d\boldsymbol{\pi}_t^\lambda(\gamma)dr
+\frac N4 (\lambda-1)\, \log\lambda.
\end{eqnarray*}
Introducing the function 
$$k_\lambda(x,y):=\lim_{R\to0} \, \inf\Big\{
\int_0^1  [1-r+\lambda r]\,k(\gamma_r)dr: \ \gamma\in \Geo(X), \gamma_0\in \B_R(x), \gamma_1\in \B_R(y)
\Big\}$$
and denoting by $q^\lambda_t$ the $W_2$-optimal coupling of $\DP_t\mu$ and $\DP_{\lambda t}\nu$,
the latter estimate can be rephrased as
\begin{eqnarray}\label{DT-lambda1}
\frac{d^+}{dt}\,\frac12 W_2(\DP_t\mu,\DP_{\lambda t}\nu)^2&\le&
-\int_{X\times X} k_\lambda(x,y)\, d^2(x,y)\, dq^\lambda_t(x,y)+\frac N4 (\lambda-1)\, \log\lambda.
\end{eqnarray}
Slightly extending the scope of  \cite{braun2019}, we define
$$W_{2,k,\lambda}(\mu,\nu,t):=\inf_{(B^1,B^2)}\, {\mathbb E}\Big[e^{-2\int_0^tk_\lambda(B^1_{2s}, B^2_{2\lambda s})ds}\cdot d^2(B^1_{2t},B^2_{2\lambda t})\Big]^{1/2}
$$
where the infimum is taken over all coupled pairs of Brownian motions $(B^1_s)_{0\le s\le 2t}$ and $(B^2_s)_{0\le s\le 2\lambda t}$ with initial distributions $\mu$ and $\nu$, resp.
Following the proof of Theorem 4.6 in \cite{braun2019}, from \eqref{DT-lambda1} we conclude
\begin{eqnarray*}
\frac{d^+}{dt}\,\frac12 W_{2,k,\lambda}(\DP_t\mu,\DP_{\lambda t}\nu)^2&\le&
\frac N4 (\lambda-1)\, \log\lambda
\end{eqnarray*}
and thus
\begin{eqnarray}\label{DT-lambda3}
 W_{2,k,\lambda}(\DP_t\mu,\DP_{\lambda t}\nu)^2&\le& W_2(\mu,\nu)^2+
\frac N2 (\lambda-1)\, \log\lambda\cdot t.
\end{eqnarray}

To proceed, we now will make use of a subtle localization argument. Recall from \cite{ambrosio2008} or from \cite{braun2019}, Lemma 2.1, that we may assume without restriction that $k$ is continuous (even Lipschitz continuous). Given $z\in X$ and $\epsilon>0$, choose $\delta>0$ and $K_z$ such that $K_z\le k\le K_z+\epsilon$ in $\B_{2\delta}(z)$.
Then following the proof of Theorem 4.2 in  \cite{sturm2018b}, we conclude that for each $p<2$, there exists $T>0$ such that for all $t,\lambda>0$ with $t(1+\lambda)\le T$ and for all $\mu,\nu$ with support in $\B_\delta(z)$
\begin{eqnarray}\label{DT-lambda4}
W_p(\DP_t\mu,\DP_{\lambda t}\nu)^2\le e^{-(K_z-\epsilon)(\lambda+1)t}\cdot 
 W_{2,k,\lambda}(\DP_t\mu,\DP_{\lambda t}\nu)^2.
\end{eqnarray}
Combining this with the previous estimate \eqref{DT-lambda3} yields
\begin{eqnarray}\label{DT-lambda5}
W_p(\DP_t\mu,\DP_{\lambda t}\nu)^2\le 
 e^{-(K_z-\epsilon)(\lambda+1)t}\cdot \Big[W_2(\mu,\nu)^2+
\frac N2 (\lambda-1)\, \log\lambda\cdot t
 \Big].
 \end{eqnarray}
This is very similar to the estimates (4.1) and (4.2) in  \cite{erbar2015} which are used there as key ingredients for deriving gradient estimates -- the main difference being now that $p<2$ on the LHS of \eqref {DT-lambda5}.
Given a bounded Lipschitz function $f$  on $X$ and putting 
$G_Rf(x)=\sup_{y\in \B_r(x)}\frac{|f(y)-f(x)|}{\d(x,y)}$, following the proof of Theorem 4.3 in \cite{erbar2015}, instead of their estimate (4.7) we now obtain with $\mu=\delta_x, \nu=\delta_y$ and $q>2$ being the dual exponent for $p$
\begin{eqnarray}
\int\big|f(x')-f(y')\big|\, dq^\lambda_t(x',y')&\le&
\int \big(P_{\lambda t}|G_Rf|^q\big)^{1/q}\cdot W_p(P_t\delta_x,P_{\lambda t}\delta_y)\nonumber\\
&&+2\frac{\| f\|_\infty}{R^2}\cdot W_p^2(P_t\delta_x,P_{\lambda t}\delta_y).
\end{eqnarray}
Choosing a sequence $(y_n)_{n\in\N}$ such that $y_n\to x$ and 
$|\nabla P_tf(x)|=\limsup_n \frac{P_tf(x)-P_tf(y_n)}{\d(x,y_n)}$ as in \cite{erbar2015},  and putting $\lambda_n=1+\alpha \, \d(x,y_n)$ leads to
\begin{eqnarray*}
\alpha \frac d{dt}P_tf(x)+|\nabla P_tf|(x)&=&
\lim_{n\to\infty}\frac1{\d(x,y_n)}\big(P_{\lambda_nt}f(x)-P_tf(y)\big)\\
&\le&\big(P_{ t}|G_Rf|^q\big)^{1/q}(x)\cdot  e^{-(K_z-\epsilon)t}
\cdot \sqrt{1+\alpha^2\frac N{2t}}.
\end{eqnarray*}
Optimizing w.r.t.\ $\alpha$ and passing to the limit $R\to0$ then yields
\begin{eqnarray}\label{grad-est-st}
\frac{2t}N \big(\Delta P_tf\big)^2(x)+\big|\nabla P_tf\big|^2(x)
&\le&\big(P_{ t}|\nabla f|^q\big)^{2/q}(x)\cdot  e^{-2(K_z-\epsilon)t}.
\end{eqnarray}

Integrating this estimate w.r.t.\ $\phi(x)\,d\mm(x)$ with a bounded nonnegative $\phi \in \Lip(X)$ supported in $\B_\delta(z)$ and then differentiating it at $t=0$ yields the following perturbed, local form of the Bochner inequality
\begin{eqnarray*}-\int\frac12 \Gamma(\phi, \Gamma(f))+\phi\,\Gamma(f,\Delta f)\,d\mm
&\ge& (K_z-\epsilon) \int \phi \,\Gamma(f)\,d\mm +\frac1N \int \phi (\Delta f)^2\,d\mm\\
&&-(q-2)\int \phi\,\Gamma(\Gamma(f)^{1/2})\,d\mm\\
&\ge&  \int  (k-2\epsilon)\,\phi\,\Gamma(f)\,d\mm +\frac1N \int \phi (\Delta f)^2\,d\mm\\
&&-(q-2)\int \phi\,\Gamma(\Gamma(f)^{1/2})\,d\mm
\end{eqnarray*}
 provided $f\in\Dom(\Delta)\cap \Lip(X)$ with $\Delta f \in \Dom(\E)$.
Covering the whole space by balls $\B_{\delta/2}(z)$ of the above type, we can find a partition of unity consisting of functions $\phi$ of the above type which allows us to deduce the perturbed Bochner inequality on all of $X$, cf. the analogous argumentation formulated as Theorem 3.10 in \cite{braun2019}. Since $\epsilon>0$ and $q>2$ were arbitrary we finally obtain the Bochner inequality in the following form:
\begin{eqnarray*}-\int\frac12 \Gamma(\phi, \Gamma(f))+\phi\,\Gamma(f,\Delta f)\,d\mm
&\ge&  \int  k\,\phi\,\Gamma(f)\,d\mm +\frac1N \int \phi (\Delta f)^2\,d\mm
\end{eqnarray*}
for all $f\in\Dom(\Delta)\cap \Lip(X)$ with $\Delta f \in \Dom(\E)$ and all bounded nonnegative $\phi \in \Lip(X)$. Following the argumentation in the proof  of  Lemma \ref{H1-Bochner-N}, one verifies the equivalence to the Bochner inequality \BE$_2(k,N)$ in its standard form. This proves the claim.
\end{proof}

\section{Time-Change and Localization}

This section is devoted to prove the transformation formula for the curvature-dimension condition under time-change. In contrast to our previous work with  Han \cite{han2019}, we now also will consider weight functions $e^\psi$ where $\psi$ is no longer in  $\Dom_\loc(\Delta)$ but merely in $\Lip_b(X)$.
This will result in $W^{-1,\infty}$-valued Ricci bounds involving the distributional Laplacian $\underline\Delta\psi$.

Moreover, we deal with weight functions  $\frac1\phi=e^{\psi}$ where the local Lipschitz function $\phi$ may  degenerate in the sense that $\phi=0$ is admitted.
Choosing $\phi$ to be an appropriate cut-off function, this allows us to ``localize'' the \RCD-condition:
we can 
restrict a given \RCD-space $(X,\d,\mm)$ to any subset $X':=\{\phi>0\}\subset X$.

\subsection{Curvature-Dimension Condition under Time-Change}

Assume that a metric measure space $(X,\d,\mm)$ is given which satisfies \RCD$(k,N)$ for some  lower bounded Borel function $k$ on $X$ and some finite number $N\in[1,\infty)$. Given $\psi\in \Lip_\loc(X)\cap \Dom_\loc(\Delta)$, we define  a new metric  and a new measure 
on $X$  by $\d':=e^\psi \odot\d$ and $\mm':=e^{2\psi}\mm$, resp. Recall that
$$\Big(e^\psi  \odot\d\Big)(x,y):=\inf\Big\{\int_0^1 e^{\psi(\gamma_s)}\,|\dot\gamma_s|\,ds: \, \gamma\in\mathcal{AC}(X), \, \gamma_0=x, \gamma_1=y\Big\}.$$

\begin{remark} 
Since both metric and measure are transformed in a coordinated manner, the Cheeger energy on the new mm-space  $(X,\d',\mm')$ coincides with the Cheeger energy on the old space:
$$\E'(f)=\int |D'f|^2d\mm'=\int |Df|^2d\mm=\E(f).$$
The point is that this energy now is regarded as a quadratic form on $L^2(X,\mm')$. The new Laplacian thus is given by $\Delta'f=e^{-2\psi}\Delta f$.

Brownian motion $(\PP'_x,B'_t)_{x\in X,t\ge0}$ on the new mm-space $(X,\d',\mm')$ is obtained by ``time change'' from the Brownian motion $(\PP_x,B_t)_{x\in X,t\ge0}$ on $(X,\d,\mm)$:
$$\PP'_x=\PP_x,\quad B'_t=B_{\tau(t)},\quad \zeta'=\sigma(\zeta)$$
and vice versa $B_t=B'_{\sigma(t)}, \zeta=\tau(\zeta')$
with 
$$\sigma(t):=\int_0^t e^{2\psi(B_s)}ds,\quad \tau(t):=\int_0^t e^{-2\psi(B'_s)}ds$$
such that $\tau(\sigma(t)=\sigma(\tau(t)=t$.

Note that in the case of bounded $\psi$, the new Brownian motion  $(\PP_x,B'_t)$ has infinite lifetime $\zeta'$ if and only if 
 $(\PP_x,B_t)$ has infinite lifetime $\zeta$.
\end{remark}

\begin{thm}[\cite{han2019}]\label{time-change} 
 {\bf i)}  For any  number $N'\in(N,\infty]$,  the ``time-changed'' metric measure space $(X,\d',\mm')$ satisfies \BE$_2(k',N')$  with
\begin{equation}\label{psi-ric}k':= e^{-2\psi}\Big[k-\Delta\psi-\frac{(N-2)(N'-2)}{N'-N}\Gamma(\psi)\Big].
\end{equation}

  {\bf ii)} Assume that $k$ is lower semicontinuous, Then  $(X,\d',\mm')$ satisfies \RCD$(k',N')$  for any lower bounded, lower semicontinuous function $k'$ on $X$ and any  number $N'\in(N,\infty]$ such that
\begin{equation*}k'\le e^{-2\psi}\Big[k-\Delta\psi-\frac{(N-2)(N'-2)}{N'-N}\Gamma(\psi)\Big]\qquad\mm'\text{-a.e.~on }X'.
\end{equation*}
\end{thm}

\begin{remark} {\bf i)} Let us re-formulate the previous Theorem in terms of $\phi:=e^{-\psi}$. That is, assume that $\phi\in \Lip_\loc(X)\cap \Dom_\loc(\Delta)$ is given with $\phi>0$ on $X$ and define  a metric  and a measure 
on $X$  by $\d':=\frac1\phi \odot\d$ and $\mm':=\frac1{\phi^2}\mm$, resp.
Observe that for $\psi:=-\log\phi$,
$$ \phi\in \Lip_\loc(X)\cap \Dom_\loc(\Delta), \ \phi>0 \quad\Longleftrightarrow\quad \psi \in \Lip_\loc(X)\cap \Dom_\loc(\Delta)$$
with $\Delta (\phi^2)=e^{-2\psi}(4\Gamma(\psi)-2\Delta \psi)$. 
  Thus the metric measure space $(X,\d',\mm')$ satisfies \RCD$(k',N')$  for any lower bounded, lower semicontinuous functions $k'$ on $X$ and any  number $N'\in(N,\infty]$ such that
\begin{equation}\label{phi-ric}k'\le k\phi^2+\frac12\Delta\phi^2-\Big[2+\frac{(N-2)(N'-2)}{N'-N}\Big]\,\Gamma(\phi)\qquad\mm'\text{-a.e.~on }X'.
\end{equation}

{\bf ii)} Another remarkable way of re-formulating the previous result is in terms of 
$$\rho:=\phi^{-(N^*-2)}=e^{(N^*-2)\psi}$$
with 
$N^*:=2+\frac{(N-2)(N'-2)}{N'-N}$ provided $N^*>2$. Then estimate \eqref{phi-ric} can be re-written as 
\begin{equation}\label{rho-ric}k'\le \rho^{-\frac2{N^*-2}}\,\Big[k-\frac1{N^*-2}\rho^{-1}\,\Delta\rho\Big]\qquad\mm'\text{-a.e.~on }X'.
\end{equation}
Recall that in the case $N^*=2$, estimate \eqref{rho-ric} states
\begin{equation*}k'\le e^{-2\psi}\Big[k-\Delta\psi\Big]\qquad\mm'\text{-a.e.~on }X'.
\end{equation*}
\end{remark}

\subsection{Localization}

We are now going to relax the positivity assumption on $\phi$, admitting $\phi$ also to vanish on subsets of $X$.

\begin{thm} {\bf (i)}
Given  $\phi\in\Lip_\loc(X)$ such that the set $\{\phi>0\}$ is connected. Define a metric measure space $(X',\d',\mm')$ by
$$X':=\{\phi>0\}, \qquad \d'=\frac1\phi \odot\d,\qquad \mm':=\frac1{\phi^2}\mm\big|_{X'}.$$
Then $\d'$ is a complete separable metric on $X'$ and $\mm'$ is a locally finite Borel measure on $(X',\d')$. 
 The metric measure space $(X',\d',\mm')$ is infinitesimally Hilbertian.
 
The sets $\Lip_\loc(X',\d)$ and $\Lip_\loc(X',\d')$ coincide. For $f\in W^{1,2}_\loc(X',\d,\mm)=W^{1,2}_\loc(X',\d',\mm')$, the 
 minimal weak upper gradients $|Df|$ and $|D'f|$ w.r.t.~the mm-spaces 
  $(X,\d,\mm)$ and  $(X',\d',\mm')$, resp., coincide.

{\bf (ii)}
Assume  in addition that $\phi\in\Lip_\loc(X)\cap \Dom_\loc(\Delta)$. 
 Then the metric measure space $(X',\d',\mm')$ satisfies \RCD$(k',N')$ for any number $N'\in(N,\infty]$ and any lower semicontinuous function $k'$ on $X'$ with
$$k'\le k\phi^2+\frac12\Delta\phi^2-N^*\Gamma(\phi)\qquad\mm'\text{-a.e.~on }X'$$
where $N^*:=2+\frac{(N-2)(N'-2)}{N'-N}$.
\end{thm}

\begin{proof} {\bf (i)} The crucial point is the completeness of the metric $\d'$. Since the metrics $\d'$ and $\d$ are obviously locally equivalent on $X'$, this will follow from the fact that
$$\lim_{y\to\partial X'}\d'(x,y)=\infty\qquad (\forall x\in X').$$
To see the latter, let points $x\in X'$ and $z\in \partial X'$ be given and let $(\gamma_t)_{t\in [0,1]}$ be any absolutely continuous curve in $(X,\d)$ with $\gamma_0=x$ and $\gamma_1=z$. Without restriction, we may assume that $\gamma$ has constant speed. 
Let $L=\Lip\phi$. Then
$$\int_0^t \frac1{\phi(\gamma_s)}|\dot\gamma_s|\,ds\ge \frac1L \int_0^t \frac1{1-s}\,ds\to\infty$$
as $t\to 1$.

{\bf (ii)} It is easy to check that the  \RCD$(k',N')$ condition has the local-to-global property, see \cite{sturm2015} for the proof in the case $N'=\infty$. Therefore, it suffices to prove that $X'$ is covered by open sets   $B$ such that the Boltzmann entropy is $(k',N')$-convex along $W'_2$-geodesics with endpoints supported in $\overline B$. We are going to verify this for $B:=\B'_r(z):=\{y\in X': \d'(y,z)<r\}$ with 
$$\B_{r}\big(\B'_{2r}(z)
\big):=\big\{x\in X: \d\big(x, \B'_{2r}(z) \big)<r\big\}
\subset X'.$$

Given such a ball $B=\B'_r(z)$, we choose $\phi_B\in\Lip_\loc(X)\cap \Dom_\loc(\Delta)$ with
$\phi_B=\phi$ in $\B'_{2r}(z)$,  $\phi_B=1$ in $X\setminus \B_{r}\big(\B'_{2r}(z)\big)$, and  $\phi_B>0$ in $X$.
See the subsequent Lemma \ref{cutoff} for the construction of such $\phi_B$'s.
According to the previous Theorem \ref{time-change}, the mm-space $(X,\frac1{\phi_B} \odot\d,\frac1{\phi_B^2}\mm)$ satisfies \RCD$(k',N')$ with $N'$ and $k'$ as claimed. Thus the Boltzmann entropy is $(k',N')$-convex along  2-Kantorovich-Wasserstein geodesics $(\mu_t)_{t\in[0,1]}$ w.r.t.~the metric $\frac1{\phi_B} \odot\d$.
If  the endpoint measures $\mu_0$ and $\mu_1$ are supported in $\overline \B'_r(z)$, however, these are exactly the  2-Kantorovich-Wasserstein geodesics w.r.t.~the metric $\frac1{\phi} \odot\d$. This proves the claim.
\end{proof}

For the reader's convenience, we quote an important result concerning cut-off functions from \cite{ambrosio2016}, Lemma 6.7.

\begin{lma}\label{cutoff} Given a locally compact \RCD$(K,\infty)$-space $(X,\d,\mm)$ and open subsets $D_0$, $D_1\subset X$ with $\overline{D_0}\subset D_1$, there exist $\phi\in \Lip_b(X)\cap\Dom(\Delta)$ with $\Delta\phi\in L^\infty(X)$ and
$\phi=1$ in $D_0$, $\phi=0$ in $X\setminus D_1$, and $\phi\ge0$ in $X$.
\end{lma}

\begin{cor}\label{localRCD}
Assume that a metric measure space $(X,\d,\mm)$ is given which satisfies \RCD$(K,N)$ for some finite numbers $K,N\in\R$. 

Then for any open subsets $D_0,D_1\subset X$ with $\overline{D_0}\subset D_1$, there exists a  metric measure space $(X',\d',\mm')$ satisfying \RCD$(K',N')$ for some finite numbers $K',N'\in\R$ such that
$$\overline{D_0}\subset X'\subset D_1,\qquad \d'=\d \text{ locally on }D_0, \qquad \mm'=\mm \text{ on }D_0.$$
\end{cor}

\begin{proof} Previous Theorem, part (ii),  plus existence of cut-off functions with bounded Laplacian according to previous Lemma.
\end{proof}

\subsection{Singular Time Change}
In the previous paragraph we dealt with an extension of Theorem \ref{time-change} where $\psi=-\log\phi$ is allowed to degenerate in the sense that it attains the value $\infty$ on closed subsets of arbitrary seize.
Now we will deal with the extension towards $\psi$ which are no longer in  $\Dom_\loc(\Delta)$ but merely in $\Lip_b(X)$.

Assume that a metric measure space $(X,\d,\mm)$ is given which satisfies \BE$_2(k,N)$ for some  lower bounded function $k$ on $X$ and some finite number $N\in[1,\infty)$. 

\begin{thm}\label{sing-time-change}  Given $\psi\in \Lip_b(X)$, the ``time-changed'' metric measure space $(X,\d',\mm')$ with  $\d':=e^\psi \odot\d$ and $\mm':=e^{2\psi}\,\mm$ satisfies $\BE_1(\kappa,\infty)$  for
\begin{equation}\label{psi-ric}\kappa:= \big[k-(N-2)\Gamma(\psi)\big]\,\mm-\underline\Delta\psi.
\end{equation}
\end{thm}

\begin{proof} {\bf i)} \ Without restriction, assume that $k$ is bounded. Choose $K\in\R_+$ with $k\ge -K$ on $X$. Given $\psi\in \Lip_b(X)\cap \Dom(\E)$, we will approximate it by $\psi_n:=P_{1/n}\psi$. Thanks to the $\BE_2(-K,N)$ assumption, the heat semigroup preserves the class of Lipschitz functions and of course it always maps $L^2$ into $\Dom(\Delta)$. Thus $\psi_n \in \Dom(\Delta)$ with  $\sup_n \Lip\,\psi_n<\infty$  and with $\psi_n\to\psi_\infty:=\psi$ in $\Dom(\E)$ and in $L^\infty$.
To see the latter, observe that by Ito's formula, 
$$1+\frac{K}{2N}\Big(\EE_x\,\d(B_t,x)\Big)^2\le \EE_x\cosh\Big(\sqrt\frac KN\d(B_t,x)\Big)\le e^{Kt/e}$$
since $\Delta \cosh\big(\sqrt{K/N}\, \d(\,.\,,x)\big)\le K\cdot  \cosh\big(\sqrt{K/N}\,\d(\,.\,,x)\big)$  by Laplace comparison.
Thus
$$|P_t\psi(x)-\psi(x)|\le \Lip\,\psi \cdot\EE_x \d(B_{2t},x)\le  \Lip\,\psi \cdot \sqrt{\frac{2N}K\Big(e^{Kt}-1\Big)}.$$

{\bf ii)} \ For $n\in\N\cup\{\infty\}$, consider the mm-space $(X,\d_n,\mm_n)$ with 
$\d_n:=e^{\psi_n}\odot \d$ and
$\mm_n:=e^{2\psi_n}\, \mm$. Let  $(P^{n}_t)_{t\ge0}$ and $(\PP_x,B^n_t)_{x\in X,t\ge0}$ denote the heat semigroup and the Brownian motion, resp., associated with it. 
Note that $B_t^{n}=B_{\tau_n(t)}$ with $\tau_n(t)$ being the inverse to
$$\sigma_n(t):=\int_0^t e^{2\psi_n}(B_s)ds.$$
Also note that due to the \BE$_1(-K,\infty)$-property, the lifetime of the original Brownian motion is infinite and thus also the lifetime of any of the time-changed Brownian motions.
Moreover, 
 as $n\to\infty$, obviously ${\tau_n(t)}\to \tau({t})$ (even uniformly in $\omega$), thus $B^n_{t}\to B^\infty_{t}$ a.s.~and 
\begin{eqnarray}
P^{n}_tf(x)=\EE_x\Big[ f\big(B_{2t}^{n}\big)\Big] \ 
\to \ \EE_x\Big[ f\big(B_{2t}^{\infty}\big)\Big]=P^{\infty}_tf(x)\label{conv1}
\end{eqnarray}
for every bounded continuous function $f$ on $X$ and every $x\in X$.

\medskip
{\bf iii)} \ 
According to  Theorem \ref{time-change},  for finite $n$, the 
the mm-space
 $(X,\d_n,\mm_n)$ satisfies  the \BE$_1(k_n,\infty)$-condition with 
\begin{eqnarray}\label{kn-formel}
k_n&=&e^{2\psi_n}\Big[k-(N-2)|\nabla\psi_n|^2-\Delta\psi_n\Big].
\end{eqnarray}
In particular, the associated heat semigroup $(P^{n}_t)_{t\ge0}$ satisfies
\begin{equation}
\label{grad_n}
\big| \nabla_n P_t^{n}f\big|
\le 
P^{k_n}_t\Big( 
\big| \nabla_n f\big| \Big)
\end{equation}
with the Feynman-Kac semigroup $P^{k_n}_t$ given in terms of the Brownian motion on  $(X,\d_n,\mm_n)$  by
$$P^{k_n}_tg(x)=\EE_x\Big[
e^{-\int_0^{t}k_n(B_{2s}^{n})ds}\, g\big(B_{2t}^{n}\big)\Big].$$
As already observed before, this can be reformulated as 
$$P^{k_n}_{t/2}g(x)=\EE_x\Big[
e^{-\frac12\int_0^{t}k_n(B_{\tau_n(s)})ds}\, g\big(B_{\tau_n(t)}\big)\Big]$$
  in terms of the Brownian motion   on $(X,\d,\mm)$.
Moreover,
\begin{eqnarray*}\frac12\int_0^{t}k_n(B_{\tau_n(s)})ds&=&
\frac12\int_0^{\tau_n(t)}\big[k-(N-2)|\nabla\psi_n|^2-\Delta\psi_n\big](B_s)ds\\
&=&A_n\big(\tau_n(t)\big)-N_n\big(\tau_n(t)\big)
\end{eqnarray*}
with $A_n(t):=\frac12\int_0^t\big[k-(N-2)|\nabla\psi_n|^2\big](B_s)ds$ and
$$N_n(t):=\frac12\int_0^t\Delta\psi_n(B_s)ds=\psi_n(B_t)-\psi_n(B_0)-M_n(t),$$
the additive functional of vanishing quadratic variation in the Fukushima-Ito decomposition of $\psi_n(B_t)$
whereas $M_n$ denotes the martingale additive functional.

\medskip
{\bf iv)} \ Since $\psi_n\to\psi$ in $\Dom(\E)$ as $n\to\infty$, obviously 
$$A_n(t)\to A(t):=\frac12\int_0^t\big[k-(N-2)|\nabla\psi|^2\big](B_s)ds$$
 $\PP_x$-a.s.~for $\mm$-a.e.~$x$.
Moreover, $A_n$ is Lipschitz continuous in $t$, uniformly in $n$, and $\tau_n(t)\to\tau_\infty(t)$ (uniformly in $\omega$).
Thus $\PP_x$-a.s.~for $\mm$-a.e.~$x$
$$A_n\big(\tau_n(t)\big)\to A_\infty\big(\tau_\infty(t)\big)$$
as $n\to\infty$.

To deal with the convergence of $N_n\big(\tau_n(t)\big)$, let $N_\infty$ and $M_\infty$ denote the additive functional of vanishing quadratic variation and the martingale additive functional, resp., 
 in the Ito decomposition $$\psi(B_t)=\psi(B_0)+M_\infty(t)+N_\infty(t),$$
 see \eqref{N-ito}.
As $n\to\infty$, of course,
$\psi_n(B_0)\to\psi(B_0)$ (uniformly in $\omega$) and
$$\psi_n(B_{\tau_n(t)})=\psi_n(B^n_t)\to \psi(B_t)$$
$\PP_x$-a.s.~for every $x$.

Furthermore,
\begin{eqnarray*}
\lefteqn{
\EE_\mm\Big| M_n(\tau_n(t))-M_\infty(\tau_\infty(t))\Big|^2}\\
&\le&
2\EE_\mm\Big| M_n(\tau_n(t))-M_\infty(\tau_n(t))\Big|^2
+
2\EE_\mm\Big| M_\infty(\tau_n(t))-M_\infty(\tau_\infty(t))\Big|^2\\
&=&2\EE_\mm\int_0^{\tau_\infty(t)}\big|\nabla(\psi_n-\psi)\big|^2(B_s)\,ds+2\EE_\mm\int_{\tau_n(t)\wedge\tau_\infty(t)}^{\tau_n(t)\vee\tau_\infty(t)}\big|\nabla\psi\big|^2(B_s)\,ds
\\
&=&C\,t\cdot \E(\psi_n-\psi)+C\,t\cdot \E(\psi)\cdot\|\psi_n-\psi\|_\infty\\
&\to& 0
\end{eqnarray*}
as $n\to\infty$ and  thus
$M_n(\tau_n(t))\to M_\infty(\tau_\infty(t))$ $\PP_x$-a.s.~for $\mm$-a.e.~$x$.

\medskip
{\bf v)} \ 
Define the taming semigroup $(P^{\kappa}_t)_{t\ge0}$ by
\begin{equation}\label{taming}P^{\kappa}_tg(x)=\EE_x\Big[e^{-A_\infty(\tau_\infty(2t)+N_\infty(\tau_\infty(2t)}\, g(B_{\tau_\infty(2t)})\}
\Big]
\end{equation}
with $A_\infty$ and $N_\infty$ as introduced above.
Then
 for  every bounded,  quasi continuous function $g$ on $X$
\begin{eqnarray*}
P^{k_n}_tg(x) \to
 P^{\kappa}_tg(x)
\end{eqnarray*} 
as $n\to\infty$ for $\mm$-a.e.~$x\in X$.
(Note that quasi continuity of $g$ implies that $t\mapsto g(B_t)$ is continuous $\PP_x$a.s.~for $\mm$-a.e.~$x\in X$.)
Moreover, recall from Proposition \ref{psi-form} (iii) and estimate \eqref{kn-formel} that
$$\big| P^{k_n}_tg(x) \big|\le e^{C_0+C_1t}\,\|g\|_\infty$$
 uniformly in $n$  for $\mm$-a.e.~$x\in X$
with constants $C_0, C_1$ depending only on $\psi$, on $\sup_n\mathrm{osc}(\psi_n)$, and on $\sup_n\Lip\,\psi_n$.
For any test plan $\Pi$ on $X$, therefore
\begin{eqnarray}
\int\int_0^1 P^{k_n}_tg(\gamma_s)\, |\dot\gamma_s|\,ds \,d\Pi(\gamma)\to
\int\int_0^1 P^{\kappa}_tg(\gamma_s)\, |\dot\gamma_s|\,ds \,d\Pi(\gamma)\label{conv2}.
\end{eqnarray}

\medskip
{\bf vi)} \ Now assume 
that $f\in\Dom(\Delta)\cap \Lip(X)$.
Since  the mm-space 
$(X,\d,\mm)$
satisfies an $\RCD$-condition, it  implies $|\nabla f|\in \Dom(\E)\cap L^\infty(X)$. By quasi-regularity of the Dirichlet form $(\E,\Dom(\E))$, therefore  $|\nabla f|$ admits a quasi continuous version. Thus applying \eqref{grad_n}, \eqref{conv1}, and \eqref{conv2} to a  quasi continuous version $g$ of $|\nabla f|$ yields
\begin{eqnarray*}
\int
\big| P_t^\infty f(\gamma_1)-P_t^\infty f(\gamma_0)\big|\, d\Pi(\gamma)&\leftarrow&\int
\big| P_t^{n}f(\gamma_1)-P_t^{n}f(\gamma_0)\big|\, d\Pi(\gamma)\\
&\le&\int\int_0^1 P^{k_n}_t|\nabla_n f|(\gamma_s)\, |\dot\gamma_s|\,ds\, d\Pi(\gamma) \\
&\le&e^{2\|\psi-\psi_n\|_\infty}\cdot\int\int_0^1 P^{k_n}_t|\nabla_\infty f|(\gamma_s)\, |\dot\gamma_s|\,ds\, d\Pi(\gamma) \\
&\to&\int
\int_0^1 P^{\kappa}_t|\nabla_\infty f|(\gamma_s)\, |\dot\gamma_s|\,ds\, d\Pi(\gamma)
\end{eqnarray*}
for any test plan $\Pi$ on $X$.
Therefore, $P^{k}_t|\nabla_\infty f|$ is a weak upper gradient for $P_t^\infty f$. Hence, in particular,
$$|\nabla_\infty P_t^\infty f|\le P^{\kappa}_t|\nabla_\infty f|.$$
By density of $\Dom(\Delta)\cap \Lip(X)$ in $\Dom(\E)$, this $L^1$-gradient estimate extends to all $f\in \Dom(\E)$. According to Theorem \ref{L1-grad}, the latter indeed is equivalent to the $L^1$-Bochner inequality \BE$_1(\kappa,\infty)$ with
\begin{equation}\label{kappa,n}
\kappa:=[k -(N-2)  |\nabla \psi|^2]\,\mm-\underline\Delta\psi.
\end{equation}
This proves the claim in the case  $\psi\in\Lip_b(X)\cap \D(E)$.

\medskip
{\bf vii)} \ In the general case of $\psi\in\Lip_b(X)$, let us choose a sequence of $\psi_j\in\Lip_b(X)\cap \D(E)$, $j\in\N$, with $\|\psi_j\|_\infty\le \|\psi\|_\infty$,
$\Lip\,\psi_j\le\Lip\,\psi$  and $\psi_j\equiv\psi$ on $\B_j(z)$ for some fixed $z\in X$. For $j\in\N$, let $(P^{\psi_j}_t)_{t\ge0}$ denote the heat semigroup on the mm-space $(X,e^{\psi_j}\odot, e^{2\psi_j}\,\mm)$ and let $(P^{\kappa_j}_t)_{t\ge0}$ denote the associated taming semigroup defined as in \eqref{taming} with $\psi$ now replaced by $\psi_j$.
Then  obviously
$$\big| P^{\psi}_tf- P^{\psi_j}_tf\big|(x)\le \|f\|_\infty\cdot \PP_x\Big( B_s\not\in \B_j(z)\ \text{ for some } s\le t\, e^{2\|\psi\|_\infty}\Big)\ \to \ 0
$$ 
as $j\to\infty$
as well as 
$$\big| P^{\kappa}_tg- P^{\kappa_j}_tg\big|(x)\le \|g\|_\infty\cdot e^{C_0+C_1t}\cdot \PP_x\Big( B_s\not\in \B_j(z)\ \text{ for some }  s\le t\, e^{2\|\psi\|_\infty}\Big)\ \to \ 0.
$$
Thus 
for any $f\in\Dom(\Delta\cap \Lip(X)$ and  any test plan $\Pi$ on $X$, as $j\to\infty$,
\begin{eqnarray*}
\int
\big| P_t^\psi f(\gamma_1)-P_t^\psi f(\gamma_0)\big|\, d\Pi(\gamma)&\leftarrow&\int
\big| P_t^{\psi_j}f(\gamma_1)-P_t^{\psi_j}f(\gamma_0)\big|\, d\Pi(\gamma)\\
&\le&\int\int_0^1 P^{\kappa_j}_t|\nabla_j f|(\gamma_s)\, |\dot\gamma_s|\,ds\, d\Pi(\gamma) \\
&\to&\int
\int_0^1 P^{\kappa}_t|\nabla_\infty f|(\gamma_s)\, |\dot\gamma_s|\,ds\, d\Pi(\gamma).
\end{eqnarray*}
 Arguing as in the previous part vi), this proves that the mm-space $(X,e^{\psi}\odot, e^{2\psi}\,\mm)$ satisfies 
 \BE$_1(\kappa,\infty)$ with
$\kappa=[k -(N-2)  |\nabla \psi|^2]\,\mm-\underline\Delta\psi$.
\end{proof}

\begin{cor} For $\psi\in \Lip_b(X)$, the heat flow $P'_t)_{t\ge0}$ on the metric measure space $(X,\d',\mm')$ with   $\d':=e^\psi \odot\d$ and $\mm':=e^{2\psi}\,\mm$ satisfies 
\begin{equation}\big|\nabla'P'_{t/2}f\big|(x)\le \EE_x\Big[e^{-A'_t}\cdot \big|\nabla' f\big|(B_t')\Big]
\end{equation}
with
$$A'_t:=\frac12
\int_0^t\big[k-(N-2)\Gamma(\psi)\big](B_s')ds+{M'}^\psi_t+\psi(B_0')-\psi(B_t')$$
where $({M'}^\psi_t)_{t\ge0}$ denotes the martingale additive functional in the Fukushima decompositions of $(\psi(B_t'))_{t\ge0}$. 

Equivalently, this can be re-formulated as
\begin{equation}\big|\nabla P'_{t/2}f\big|(x)\le \EE_x\Big[e^{-A_t}\cdot \big|\nabla  f\big|(B_t')\Big],
\end{equation}
now with $\nabla$ in the place of $\nabla'$ and with $A'$ replaced by
$$A_t:=\frac12
\int_0^t\big[k-(N-2)\Gamma(\psi)\big](B_s')ds+{M'}^\psi_t.$$
\end{cor}

\begin{ex}\label{Cantor2}  Let $(X,\d,\mm)=(\R^2,\d_{\sf Euc},\mm_{\sf Leb})$ be the standard 2-dimensional  mm-space. Define functions $\psi_j$, $j\in\N$, and $\psi:\R^2\to\R$ by
 $$\psi_j(x_1,x_2)=\Phi_j(x_1)\cdot \eta(x_2), \quad \psi(x_1,x_2)=\Phi(x_1)\cdot \eta(x_2)$$
 with $\Phi, \Phi_j$ as defined in \eqref{CantorPhi} for some $\varphi\in {\mathcal C}^2(\R)$ and with $\eta\in {\mathcal C}^2(R)$ given  by $\eta(t):=(t^2-1)^3t$ for $t\in [-1,1]$ and $\eta(t):=0$ else.
For each $j\in\N$, the time-changed mm-space $(\R^2,\d_j,\mm_j)$ with $d_j:=e^{\psi_j}\odot \d_{\sf Euc},
\mm_j:=e^{2\psi_j}\,\mm_{\sf Leb}$ corresponds to the Riemannian manifold $(\R,{\mathsf g}_j)$ with metric tensor given by ${\mathsf g}_j:=e^{2\psi_j}\,{\mathsf g}_{\sf Euc}$.  Its Ricci tensor is bounded from below by
$$k_j=-e^{-2\psi_j}\Delta\psi_j$$
cf. previous Theorem. 
cf. Propostion \ref{time-change}. (Note that the measure-valued Ricci bound is given by $k_j\,\mm_j=-\underline\Delta\psi_j$.) In the limit $j\to\infty$, we will end up with a mm-space $(\R^2,\d_\infty,\mm_\infty)$ with distributional Ricci bound given by 
$$\kappa=-\underline\Delta\psi.$$
\end{ex}

\section{Gradient Flows and Convexification}

This section is devoted to extensions of the results from \cite{sturm2018a} on existence of gradient flows and from \cite{lierl2018} on convexification of semi-convex subsets towards functions with variable lower bound for the convexity (\cite{sturm2018a}) or domains with  variable lower bound  for the curvature of the boundary
(\cite{lierl2018}), resp. Of particular importance is the fact that these lower bounds may change sign. Even  the case of constant positive lower bounds leads to new insights not covered by previous results.

\subsection{Gradient Flows for Locally Semiconvex Functions}

The goal of this subsection  is to extend the existence result and the contraction estimate for gradient flows for semiconvex functions from \cite{sturm2018a} to the setting of locally semiconvex functions. The contraction estimate for the flow will be in terms of the variable lower bound for the local semiconvexity of the potential.

\medskip
 
Let $(X,\d,\mm)$ be a locally compact metric measure space satisfying an \RCD$(K,\infty)$-condition (cf. Definition \ref{def-cd-var}) for some $K\in\R$.
Assume moreover, that we are given a
continuous potential $V:X\to\R$ which is weakly $\ell$-convex for some continuous, lower bounded function $\ell:X\to\R$ in the following sense:
for all $x, y\in X$ there exists a geodesic $(\gamma_r)_{r\in [0,1]}$ connecting them such that for all $r\in[0,1]$
\begin{equation}\label{conv}
V(\gamma_r)\le (1-r)V(\gamma_0)+r V(\gamma_1)-\int_0^1 \ell(\gamma_s)g(r,s)ds\cdot \d^2(\gamma_0,\gamma_1)
\end{equation}
where $g(r,s):=\min\{(1-s)r,(1-r)s\}$ denotes the Green function  on $[0,1]$.

We say that a curve $(x_t)_{t\in[0,\infty)}$ in $X$ is an \EVI$_\ell$-gradient flow for $V$ (where \EVI \ stands for ``evolution-variational inequality'') if the curve is locally absolutely continuous in $t\in(0,\infty)$ and if for every $t>0$, 
every $y\in X$, and every geodesic $(\gamma_r)_{r\in [0,1]}$ connecting $x_t$ and $y$,
\begin{eqnarray}\label{evi-flow}
-\frac12\frac{d^+}{dt}\d^2(x_t,y)&\ge&
V(x_t)-V(y)+
  \int_0^1(1-r)\,
   \ell(\gamma_r)\,dr\cdot \d^2(x_t,y).
\end{eqnarray}

\begin{thm}\label{ex-grd-flow} For every $x_0\in X$, there exists a unique \EVI$_\ell$-gradient flow for $V$ starting in $x_0$  (and existing for all time). Flows starting in $x_0$ and $y_0$  satisfy 
\begin{equation}\d(x_t,y_t)\le e^{-\int_0^t \overline\ell(x_s,y_s)ds}\d(x_0,y_0)
\end{equation}
for all $t\ge0$.
Here 
$\overline\ell(x,y):=\sup_\gamma\int_0^1\ell(\gamma_r)dr$
with $\sup_\gamma$ taken over all geodesics $(\gamma_t)_{t\in [0,1]}$  in $(X,\d)$  with $\gamma_0=x,\gamma_1=y$. 
\end{thm}

\begin{proof} Existence and uniqueness of an \EVI$_\Lambda$-gradient flow $(\Phi_t(x))_{t\ge0}$ for any $x\in X$ with $\Lambda:=\inf\ell$ follows from our previous work \cite{sturm2018a}. 
Following the previous proof, one also can deduce the refined  \EVI$_\ell$-property. Indeed, this will follow as before by a scaling argument from the  \EVI$_{K-n\ell}$-property for the heat flow on the weighted metric measure space $(X,\d,e^{-nV}\mm)$ which (obviously) satisfies the \RCD$(K+n\ell,\infty)$-condition. If we now compare two flows, then we can apply \eqref{evi-flow} twice: first to the flow $(x_t)_{t\ge0}$ and with $y_t$ in the place of $y$; then to the flow  $(y_t)_{t\ge0}$ and with $x_t$ in the place of $y$. Adding up both estimates yields (after some tedious arguments to deal with weakly differentiable functions with double dependence on the varying parameter)
   $$-\frac12\frac{d}{dt} \d^2\big(x_t,y_t\big)\ \ge \  
  \overline\ell\big(x_t,y_t\big)\cdot \d^2\big(x_t,y_t\big).$$

\medskip

Alternatively, one can argue as follows:
Given $\epsilon>0$,  let $X$ be covered by balls $B_i=\B_{r_i}(z_i)$ such that $\ell_i\le\ell\le \ell_i+\epsilon$ on $\B_{2r_i}(z_i)$.
Thanks to the Localization Theorem \ref{localRCD}, for each $i$ there exists an \RCD-spaces $(X_i,\d_i,\mm_i)$ with 
 $\B_{r_i}(z_i)\subset X_i$  whose local data  on  $\B_{r_i}(z_i)$ coincide with those of the original one. Thus  as long as the flow does not leave $\B_{r_i}(z_i)$, we  can consider the original flow also as an \EVI-gradient flow for $V$ on  the mm-spaces $(X_i,\d_i,\mm_i)$.
 
 Given any $(X,\d)$-geodesic $(\gamma_t)_{t\in [0,1]}$ and $r,s\in[0,1]$ with $\gamma_r,\gamma_s\in \B_{r_i}(z_i)$ we thus conclude that
 $$\frac{d}{dt}\Big|_{t=0} \d\big(\Phi_t(\gamma_r),\Phi_t(\gamma_s)\big)\ \le \ -\ell_i\cdot \d\big(\gamma_r,\gamma_s\big)
 \ \le\ -\d\big(\gamma_0,\gamma_1\big)\cdot\int_r^s\big(\ell(\gamma_q)-\epsilon\big)\,dq .$$ 
 Adding up these estimates for consecutive pairs of points on the geodesic $(\gamma_t)_{t\in [0,1]}$ finally gives
  $$\frac{d}{dt}\Big|_{t=0} \log\d\big(\Phi_t(\gamma_0),\Phi_t(\gamma_1)\big)\ \le \  -
  \int_0^1\ell\big(\gamma_q\big)\,dq+\epsilon.$$
  Choosing $\gamma$ optimal, we therefore obtain for arbitrary $x_0,x_1\in X$ and for all $t\ge0$
    $$\frac{d}{dt} \log\d\big(\Phi_t(x_0),\Phi_t(x_1)\big)\ \le \  -
  \overline\ell\big(\Phi_t(x_0),\Phi_t(x_1)\big)+\epsilon.$$
  Thus
 $$\d\big(\Phi_t(x_0),\Phi_t(x_1)\big)\ \le \ e^{ -\int_0^t
  \overline\ell\big(\Phi_s(x_0),\Phi_s(x_1)\big)ds+\epsilon\,t}\cdot \d\big(\Phi_t(x_0),\Phi_t(x_1)\big).$$
Since $\epsilon>0$ was arbitrary, this yields the claim. 
\end{proof}

As pointed out in \cite{sturm2018a} in the case of constant $\ell$, the existence of \EVI-flows for $V$ can be regarded as a strong  formulation of semiconvexity of $V$. 

\begin{cor} Every weakly $\ell$-convex function is indeed strongly $\ell$-convex in the sense that
the inequality \eqref{conv} holds for \emph{every} geodesic $(\gamma_t)_{t\in [0,1]}$  in $X$.
\end{cor}

A closer look on the proof of the previous Theorem \ref{ex-grd-flow} shows that appropriate re-formulations of the results also hold true for flows which are defined only locally.

\begin{thm} Assume that continuous functions $V$ and $\ell: Y\to\R$ are defined on an open subset $Y\subset X$ and that $V$ is 
$\ell$-convex on $Y$ in the sense that the inequality \eqref{conv} holds for every geodesic $(\gamma_t)_{t\in [0,1]}$  contained in  $Y$.

{\bf (i)} Then for each $x_0\in Y$ there exists a unique local \EVI$_\ell$-gradient flow $(x_t)_{t\in[0,\tau)}$ for $V$  with maximal life time $\tau=\tau(x_0)\in (0,\infty]$.
If $\tau<\infty$ then $x_\tau=\lim_{t\to\tau}x_t$ exists and $x_\tau\in\partial Y$.

{\bf (ii)} For any pair of initial points $x_0,y_0\in Y$ and their \EVI$_\ell$-flows $(x_t)_{t\ge0}, (y_t)_{t\ge0}$, the estimate
\begin{equation*}\d(x_t,y_t)\le e^{-\int_0^t \overline\ell(x_s,y_s)ds}\d(x_0,y_0)
\end{equation*}
holds  
for all $t\le T^*$ where $T^*=T^*(x_0,y_0)$ denotes the first time where a geodesic connecting $x_t$ and $y_t$ will leave $Y$.
\end{thm}

\begin{proof} {\bf (i)} Existence and  uniqueness of a local \EVI$_\ell$-gradient flow are straightforward. Applying the estimate to points $x_0$ and $y_0:=x_\delta$ proves that the flow $(x_t)_t$ has finite speed. Assuming $\tau<\infty$, the family $(x_t)_{t<\tau}$ will be bounded and therefore admits a unique accumulation point  for $t\to\tau$, say $x_\tau\in \overline Y$. Assuming that $\tau$ is the maximal life time for the flow implies that $x_\tau\in\partial Y$.

{\bf (ii)} follows exactly as in the case of the globally defined gradient flow.
\end{proof}

\subsubsection*{The Hessian along Geodesics}
On a Riemannian manifold $(M,g)$, the {Hessian} $D^2f$ of a smooth function $f: M\to\R$ can be regarded as a bilinear form $D^2f: TM\times TM\to\R$ or equivalently as a quadratic form on the tangent space $TM$.
With the latter interpretation, for $\xi=(x,v)\in TM$ with $x\in M$ and $v\in T_xM$,
\begin{equation}
D^2f(\xi)=\frac1{|\dot\gamma_0|^2}\cdot \frac{d^2}{dt^2}f(\gamma_t)\big|_{t=0}
\end{equation}
for any $\gamma\in\Geo_x(M)$ with $\dot\gamma_0=v$.
Here $\Geo_x(M)$ denotes the set of all  geodesics $\gamma: (a,b)\to M$ with $0\in (a,b)$ and $\gamma_0=x$. (As usual, `geodesics' are constant speed and minimizing.)

\begin{Def} Given a geodesic space $(X,\d)$,  an open set $Y\subset X$ and 
functions $V:Y\to\R$ and $\lambda:Y\to\R$. We say that the \emph{
Hessian of $V$ along geodesics} (or the geodesic Hessian of $V$) is bounded from below by $\lambda$, briefly $$D^2_{\Geo}V\ge \lambda\quad\text{on }Y,$$ if
for every unit speed geodesic $\gamma:(a,b)\to Y$ the function $\lambda\circ\gamma: (a,b)\to\R$ is locally integrable and  $u:=V\circ\gamma: (a,b)\to\R$ is locally absolutely continuous with
$$u''\ge 
\lambda\circ\gamma\quad \text{on $(0,1)$
 in distributional sense}.$$
\end{Def}
 
 \begin{ex} A 
 function $V: X\to\R$ is strongly $K$-convex if and only if 
$$D^2_{\Geo}V\ge K\quad\text{on }X.$$
 \end{ex}
 \begin{Prop}\label{1:intr-conv-metr} Given a geodesic space $(X,\d)$,  an open set $Y\subset X$, a continuous function $V:Y\to\R$, and a number $\kappa\in\R$. Then  
 $$D^2_{\Geo}V\ge -\kappa\, V\quad\text{on }Y$$
 if and only if
 for all geodesics $\gamma:[0,1]\to Y$ of length $<R_\kappa:=\pi/\sqrt\kappa$
 and 
 for all $s\in (0,1)$,
\begin{equation}\label{1:implicit-convex-metric}
V(\gamma_s)\le \sigma^{(1-s)}_\kappa\big(|\dot\gamma|\big) \cdot V(\gamma_0) +\sigma^{(s)}_\kappa\big(|\dot\gamma|\big)\cdot V(\gamma_1).
\end{equation}
Indeed, it suffices to verify the latter for $s=1/2$, and at each point $x\in\X$ for all sufficiently short geodesics with $\gamma_{1/2}=x$.
 \end{Prop}
 
 \begin{proof} By definition $D^2_{\Geo}V\ge -\kappa\, V$ on $Y$
 if and only if for each geodesic  $\gamma:(a,b)\to Y$ the function $u:=V\circ\gamma: (a,b)\to\R$ is absolutely continuous and satisfies $u''\ge -\bar\kappa\,u$ on $(a,b)$ in distributional sense with $\bar\kappa:=|\dot\gamma|^2\,\kappa$. Straightforward calculations and comparison results for Sturm-Lioville equations yield that this is equivalent to 
\begin{equation}\label{1:implicit-convex-metric2}
V(\gamma_s)\le \sigma^{(\frac{t-s}{t-r})}_\kappa\big((t-r)\,|\dot\gamma|\big) \cdot V(\gamma_r) +\sigma^{(\frac{s-r}{t-r})}_\kappa\big((t-r)\,|\dot\gamma|\big)\cdot V(\gamma_t)
\end{equation}
 for all $a<r<s<t<b$ with $(t-r)\,|\dot\gamma|<R_\kappa:=\pi/\sqrt\kappa$. The latter obviously follows from \eqref{1:implicit-convex-metric} by linear  rescaling of  the interval 
 $[r,t]$ onto $[0,1]$. Conversely,  \eqref{1:implicit-convex-metric} follows from  \eqref{1:implicit-convex-metric2} with $a=1, b=1$ by passing to the limit $r\searrow0, t\nearrow1$ and using continuity of $V$.
  \end{proof}

\subsection{Convexification}

In this subsection, we will prove the fundamental Convexification Theorem which (via time-change) allows to transform the metric  of a  mm-space $(X,\d,\mm)$  in such a way that a given semiconvex subset $Y\subset X$ will become  geodesically convex  w.r.t.~the new  metric $\d'$. We will prove this in two versions: first, for closed sets $Y$, then for open sets $Z$.

Throughout this section, we fix a locally compact \RCD$(K,\infty)$-space $(X,\d,\mm)$.
Given a function $V:\X\to (-\infty,\infty]$, we denote its \emph{descending slope} by $$|\nabla^- V|(x):=\limsup_{y\to x}\frac{[V(y)-V(x)]^-}{\d(x,y)}$$ provided  $x$ is not isolated, and by
 $|\nabla^- V|(x):=0$ otherwise. Moreover, we put $|\nabla^+ V|:=|\nabla^- (-V)|$.

\begin{Def} We say that a subset $Y\subset X$ is \emph{locally geodesically convex} if
there exists an open covering $\bigcup_{i\in I}U_i\supset X$  such that  every geodesic $(\gamma_s)_{s\in[0,1]}$ in $X$ completely lies in $Y$ provided $\gamma_0,\gamma_1\in Y\cap U_i$ for some $i\in I$.
\end{Def}

Every geodesically convex set is locally geodesically convex but not vice versa.

\begin{ex} Let $X$ denote the cylinder $\R\times S^1$ and $Y_0=\{(t,\varphi): t=0, |\varphi|\le\pi/2\}$. Then $Y:=\B_1(Y_0)$ is locally geodesically convex but not (geodesically) convex.
\end{ex}

\begin{thm}\label{1stConvThm} Let $V,\ell:X\to\R$ be continuous functions and assume that for each 
$\epsilon>0$ there exists a  neighborhood $D_\epsilon$ of the closed set $Y:=\{V\le0\}$ 
 such that 
\begin{itemize}
\item $1-\epsilon\le |\nabla^- V|\le 1+\epsilon$ in $D_\epsilon\setminus Y$ 
\item $V$ is  $(\ell-\epsilon)$-convex in $D_\epsilon\setminus Y$.
\end{itemize}
Then for every $\epsilon>0$, the set $Y$ is locally  geodesically convex in $(X,\d')$ for $\d'=e^{(\epsilon-\ell)\cdot V}\odot \d$.
\end{thm}

\begin{remark} (i) The above Theorem provides a far reaching extension of our previous result in \cite{lierl2018} which covers the case of constant negative $\ell$. Now we also admit variable $\ell$ and $\ell$'s of arbitrary signs.

(ii) Note that in the case of positive $\ell$, the set $Y$ will already be convex in the old metric space $(X,\d)$ and it will be ``less convex'' in the new space $(X,\d')$.

(iii)   In the above Theorem, without restriction, we may put $V\equiv0$ in $Y$. Moreover, for both functions $V$ and $\ell$ it suffices that they exist as continuous functions on $D\setminus Y^0$ for some neighborhood $D$ of $Y$.
\end{remark}

\begin{proof} 
Let $\epsilon'>0$ be given and put $\d'=e^{(\epsilon'-\ell)V}\odot \d$.

(i)
In order to prove the local convexity of $Y$ in $(X,\d')$, let $z\in\partial Y$  be given and choose  $\epsilon>0$ sufficiently small (to be determined later). In any case, assume that $(\frac{1+\epsilon}{1-\epsilon})^2<2$. Choose $\delta>0$  such that
\begin{itemize}
\item $1-\epsilon\le |\nabla^- V|\le 1+\epsilon$ in $\B_\delta(z)\setminus Y$
\item $V$ is geodesically $(\ell(z)-\epsilon)$-convex in $\B_\delta(z)\setminus Y$
\item $|\ell(x)-\ell(z)|\le\epsilon$ for all $x\in \B_\delta(z)\setminus Y$.
\end{itemize}
Our proof of the local convexity of $Y$ will be based on a curve shortening argument under the gradient flow for $V$:
Assume that $(\gamma_a)_{a\in[0,1]}$ was a $\d'$-geodesic in $\B_{\delta/3}(z)$ with endpoints $\gamma_0,\gamma_1\in Y$ and $\gamma_a\not\in Y$ for some $a\in (0,1)$. Then we will construct a new curve $(\gamma^0_a)_{a\in[0,1]}$ with the same endpoints but which is shorter (w.r.t.~$\d'$) than the previous one --  which obviously contradicts the assumption.
For each $a\in[0,1]$, we consider the gradient flow curve $\big(\Phi_t(\gamma_a)\big)_{t\ge0}$ for $V$ starting in $\Phi_0(\gamma_a)=\gamma_a$ and we stop it as soon as the flow enters the set $Y$. Then we put
$\gamma^0_a:=\Phi_{T_0}(\gamma_a)$ where $T_0=\inf\{t\ge0: \Phi_t(\gamma_a)\in Y\}$.

\medskip

(ii) To get started, let us first summarize some key facts for the gradient flow for $V$, that is, for the solution to  $\dot x_t=-\nabla V(x_t)$ in the sense of \EVI-flows.
For $x\in \B_\delta(z)\setminus Y$, let $\big(\Phi_t(x)\big)_{t\in[0,\tau)}$ denote the  \EVI-gradient flow for $V$ starting in $x$ with maximal life time $\tau=\tau(x)$ in $\B_\delta(z)\setminus Y$.
Note that $V(x)\le(1+\epsilon)\delta$ for $x\in \B_\delta(z)\setminus Y$ and for a.e.~$t$,
$$\frac d{dt}\Phi(t,x)=-|\nabla^- V|^2(\Phi(t,x))\ge -(1+\epsilon)^2\quad\text{and}\quad \le -(1-\epsilon)^2.$$ 
We easily conclude that $\tau(x)<\infty$ for $x\in \B_{\delta/3}(z)$ and, since  $(\frac{1+\epsilon}{1-\epsilon})^2<2$, 
$$\Phi_{\tau(x)}(x)\in\partial Y\cap \B_\delta(z).$$
In particular, $\tau(x)=T_0(x)=\lim_{r\downarrow 0}T_r(x)$ with
$T_r(x):=\inf\{t\ge0: V(\Phi_t(x))\le r\}$ for $r\ge0$.

\medskip

(iii) To proceed, it is more convenient to parametrize the flow not by time (as we did before) but by ``height'', measured by the value of $V$.
That is, for $r\ge0$ we put $\hat\Phi_r(x)=\Phi_{T_r(x)}(x)$
with $T_r(x)$ as above. Moreover, for $x\in Y$ we put $T_r(x):=0$ and $\hat\Phi_r(x):=x$ for all $r\ge0$.

 The $(\ell(z)-\epsilon)$-convexity of $V$ implies 
 \begin{eqnarray}\d(\hat\Phi_r(x),\hat\Phi_r(y))\le e^{(\epsilon-\ell(x))(T_r(x)+T_r(y))/2}\cdot \d(x,y),
 \end{eqnarray}
 see 
 \cite{lierl2018}, 
 Lemma 2.13 (or, more precisely, estimate (10) in the proof of it).
\medskip

(iv) 
 Given any rectifiable curve $(\gamma_a)_{a\in[0,1]}$ in $\B_{\delta/3}(x)$ let
 $(\gamma_a^0)_{a\in[0,1]}$ be the curve in $\B_\delta(x)\setminus Y^0$ defined by $\gamma_a^0=\hat\Phi_0(\gamma_a)$. Then
\begin{eqnarray*} |\dot\gamma_a^0|&\le& e^{(\epsilon-\ell(z))T_0(\gamma_a)}\cdot |\dot\gamma_a|\\
&\le& e^{(2\epsilon-\ell(\gamma_a))(1\pm\epsilon)^2V(\gamma_a)}\cdot |\dot\gamma_a|\\
&\le& e^{(\epsilon'/2-\ell(\gamma_a))V(\gamma_a)}\cdot |\dot\gamma_a|.
\end{eqnarray*}
Indeed, for every $\epsilon'>0$ one can choose $\epsilon>0$ such that 
$(\epsilon'/2-\ell(x)) \ge (2\epsilon-\ell(x))(1\pm\epsilon)^2$
for all $x\in \B_\delta(x)\setminus Y$. Here and above, the sign in the expression $(1\pm\epsilon)^2$ has to be chosen according to the sign of $(2\epsilon-\ell(x))$.

  Measuring the speed of the curves now in the metric $d'=e^{(\epsilon'-\ell)V}\odot d$, the previous estimate yields
 \begin{eqnarray*} |\dot\gamma_a^0|'&\le& |\dot\gamma_a|'
\end{eqnarray*}
and, 
moreover, 
$ |\dot\gamma_a^0|'< |\dot\gamma_a|'$ whenever $\gamma_a\not\in Y$ and $|\dot\gamma_a|'\not=0$.
This proves the claim.
\end{proof}

In the previous Theorem, we used the gradient flow w.r.t.~a function $V$ (which shares basic properties with the distance function $\d(\,.\,,\partial Y)$)
as a path-shortening flow on the exterior of $Y$ in order to prove that the \emph{closed set} $Y$ is locally geodesically convex w.r.t.~the new metric $\d'$.

To make a given \emph{open set} $Z\subset X$ locally geodesically convex w.r.t.~a new metric $\d'$, we will proceed in a complementary way: 
we will use the gradient flow w.r.t.~a function $V$ which shares basic properties with the negative distance function $-\d(\,.\,,\partial Z)$
as a path-shortening flow in the interior of $Z$.
This is the content of the 
 Second Convexification Theorem.

\begin{thm}\label{2ndConvThm} Let $V,\ell:X\to\R$ be continuous functions and assume that for each 
$\epsilon>0$ there exists $\delta>0$
 such that 
\begin{itemize}
\item $1-\epsilon\le |\nabla^- V|\le 1+\epsilon$ on the set $\{-\delta<V<0\}$ 
\item $V$ is  $(\ell-\epsilon)$-convex on the set $\{-\delta<V<0\}$.
\end{itemize}
Then for every $\epsilon>0$, the open set $Z:=\{V<0\}$ is locally  geodesically convex in $(X,\d')$ for $\d'=e^{(\epsilon-\ell)\cdot V}\odot \d$.
\end{thm}

\begin{proof} Given $\epsilon>0$, we choose $\delta>0$ as above. Then for each $r\in(0,\delta]$, we can apply the First Convexification Theorem \ref{1stConvThm}  with $V$ as above, with  the closed set $Z_r:=\{V\le-r\}$ in the place of $Y$, and with  $D_\epsilon:=Z\setminus Z_r$. This yields that the set $Z_r$ is locally geodesically convex w.r.t.~the metric $\d':=e^{(\epsilon-\ell)V}\odot \d$. Having a closer look on the proof of Theorem  \ref{1stConvThm}, we see that we can choose an open covering 
$\bigcup_{i\in I}U_i\supset X$, independently of $r$, 
such that  every $\d'$-geodesic $(\gamma_s)_{s\in[0,1]}$ in $X$ completely lies in $Z_r$ -- and thus in particular in $Z$ -- provided $\gamma_0,\gamma_1\in Z_r\cap U_i$ for some $i\in I$. This proves the claim: every $\d'$-geodesic $(\gamma_s)_{s\in[0,1]}$ completely lies in $Z$  provided $\gamma_0,\gamma_1\in  Z\cap U_i$ for  some $i\in I$.
\end{proof}

 \subsection{Bounds for the Curvature of the Boundary}

The canonical choice for the function $V$ in both of the previous Convexification Theorems is  the signed distance function 
 $V=\d(\,.\,,Y)-\d(\,.\,,X\setminus Y)$ (or suitable truncated and/or smoothened modifications of it). In the Riemannian setting, a lower bound for the Hessian of this function has a fundamental geometric meaning: it is a lower bound of the fundamental form of the boundary. In the abstract setting, this observation will  provide a synthetic definition for the variable lower bound curvature of the boundary.

\begin{Def} \label{def-curv}Let a closed set $Y\subset X$  and a continuous function $\ell:X\to\R$ be given, and put $V(x):=\d(x,Y)-\d(x,X\setminus Y)$. 
We say that $Y$ is \emph{locally $\ell$-convex} or that \emph{$\ell$ is a
 lower bound for the curvature  of $\partial Y$} if for every $\epsilon>0$ there exist an open covering $\bigcup_{i\in I}U_i\supset \partial Y$ and  continuous functions $V_i:U_i\to\R$ for $i\in I$ such that 
 \begin{itemize}
 \item 
  $(1-\epsilon)\, V\le V_i\le (1+\epsilon)\,V$ on $U_i$ 
\item $1-\epsilon\le |\nabla^- V_i|\le 1+\epsilon$ on $U_i$ 
\item $V_i$ is  $(\ell-\epsilon)$-convex on $U_i$.
\end{itemize}
We say that  $Y$ is \emph{locally $\ell$-convex from outside} (or that  $Y^0$ is \emph{locally $\ell$-convex from inside)} if the three latter properties are 
merely requested on $U_i\setminus Y$ (or on $U_i\cap Y^0$, resp.) instead of being requested on $U_i$.
\end{Def}

%
%
%
%
%
%
%

\begin{remark}
The previous approach does not only allow us to define \emph{lower bounds} for the curvature of the boundary (interpreted as lower bounds for the ``second fundamental form of the boundary'') but also to define the \emph{second fundamental form} $\text{I\!I}_{\partial Y}$ itself as well as the \emph{mean curvature} $\rho_{\partial Y}$: the former as the Hessian (restricted to vectors orthogonal to $\nabla V$) and the latter as the Laplacian of the signed distance function $V=\d(\,.\,,Y)-\d(\,.\,,X\setminus Y)$. That is,
$$\text{I\!I}_{\partial Y}(f,f):=H_V(f,f)=\Gamma(f,\Gamma(V,f))-\frac12\Gamma(v,\Gamma(f^2))$$
 provided  $V\in\Dom_\loc(\Delta)$ and $\Gamma(f,V)=0$, and
$$\rho_{\partial Y}:=\Delta V.$$
\end{remark}

This concept of curvature bounds for the boundary has been introduced in 
\cite{lierl2018}, restricted there to the case of constant, nonpositive $\ell$. As already observed there, the two most important classes of examples are Riemannian manifolds and Alexandrov spaces. In these cases, if not explicitly specified otherwise, $\mm$ always will denote the corresponding $n$-dimensional Riemannian volume measure ${\mathcal L}^n$ or the $n$-dimensional Hausdorff measure ${\mathcal H}^n$. The proof of the next result is literally as in \cite{lierl2018}.

\begin{Prop} Let $X$ be a Riemannian manifold, $Y$ a closed subset of $X$ with smooth boundary. Then $\ell$ is a lower bound (or interior lower bound or exterior lower bound, resp.) for the curvature  of  $\partial Y$
if and only if the real-valued second fundamental form of $\partial Y$ satisfies $$\text{I\!I}_{\partial Y}\ge 
\ell.$$
\end{Prop}

\begin{lma}\label{5.12} Let $(X,\d)$ be an Alexandrov space with generalized sectional curvature $\ge K$.
Put $\rho_K=\frac\pi{2\sqrt{K}}$ if $K>0$ and $\rho_K=\infty$ else. Moreover, given $z\in\X$ put
$\rho(z):=\sup\{r\ge0: |\nabla^+\d(.,z)|=1 \text{ in }\B_r(z)\}$. Then for each $r\in (0,\rho_K\wedge\rho(z))$, the curvature of the boundary of 
$Y:=X\setminus \B_r(z)$
 is  bounded from below by
\begin{equation}\label{ell-unten}\ell=-\cot_K(r):=\left\{
\begin{array}{ll}
-\frac1{r}, &\mbox{if }K = 0\\
-\sqrt K\,\cot(\sqrt{K} r), &\mbox{if }K>0\\
-\sqrt{-K}\,\coth(\sqrt{-K} r), &\mbox{if }K<0.\\
\end{array}\right.
\end{equation}
\end{lma}

\begin{proof}
Given  $z\in X$ and $r\in (0,\rho_K)$, put 
\begin{equation}\label{Vrz}V(x):=V_{r,z}(x):=\left\{
\begin{array}{ll}
\frac1{2 r}\Big(r^2-\d^2(x,z)\Big), &\mbox{if }K = 0\\
\frac1{\sqrt{K}\,\sin\big(\sqrt{K} r\big)}\Big(
\cos\big(\sqrt{K} \d(x,z)\big)-\cos\big(\sqrt{K} r\big)
, &\mbox{if }K>0\\
\frac1{\sqrt{-K}\,\sinh\big(\sqrt{-K} r\big)}\Big(
\cosh\big(\sqrt{-K} r-\cosh\big(\sqrt{-K} \d(x,z)\big)\big)
, &\mbox{if }K<0.
\end{array}\right.
\end{equation}
Then obviously $\{V\le0\}=Y$ and $|\nabla^- V|=1$ on $\partial Y$ (and close to 1 in a neighborhood of $\partial Y$). Moreover, by comparison results for Hessians of distance functions in Alexandrov spaces
$$
\begin{array}{ll}D^2_\Geo\,\d^2(x,z)\le 2
, &\mbox{if }K = 0\\
D^2_\Geo \cos\big(\sqrt{K} \d(x,z)\big)
\ge
-K\,\cos\big(\sqrt{K} \d(x,z)\big)
, &\mbox{if }K>0\\
D^2_\Geo \cosh\big(\sqrt{-K} \d(x,z)\big)
\le
(-K)\,\cosh\big(\sqrt{-K} \d(x,z)\big)
, &\mbox{if }K<0.
\end{array}
$$
Thus
$D^2_\Geo V\ge -\cot_K(r)$ on  $\partial Y$  (and $\ge -\cot_K(r)-\epsilon$ in a neighborhood of $\partial Y$). 
This proves the claim.
\end{proof}

\begin{lma} Let $(X,\d)$ be an CAT space with generalized sectional curvature $\le L$.
Put $\rho_L=\frac\pi{2\sqrt{L}}$ if $L>0$ and $\rho_L=\infty$ else. Then for each $r\in (0,\rho_L)$, the curvature of the boundary of $Y:= \B_r(z)$ 
  is bounded from below by
\begin{equation}\label{ell-oben}\ell=\cot_L(r):=\left\{
\begin{array}{ll}
\frac1{r}, &\mbox{if }L = 0\\
\sqrt L\,\cot(\sqrt{L} r), &\mbox{if }L>0\\
\sqrt{-L}\,\coth(\sqrt{-L} r), &\mbox{if }L<0.\\
\end{array}\right.
\end{equation}
\end{lma}

\begin{proof}
Given  $z\in X$ and $r\in (0,\rho_L)$, put 
\begin{equation*}V(x):=\left\{
\begin{array}{ll}
-\frac1{2 r}\Big(r^2-\d^2(x,z)\Big), &\mbox{if }L = 0\\
-\frac1{\sqrt{L}\,\sin\big(\sqrt{L} r\big)}\Big(
\cos\big(\sqrt{L} \d(x,z)\big)-\cos\big(\sqrt{L} r\big)
, &\mbox{if }L>0\\
-\frac1{\sqrt{-L}\,\sinh\big(\sqrt{-L} r\big)}\Big(
\cosh\big(\sqrt{-L} r-\cosh\big(\sqrt{-L} \d(x,z)\big)\big)
, &\mbox{if }L<0.
\end{array}\right.
\end{equation*}
Then obviously $\{V\le0\}=Y$ and $|\nabla^- V|=1$ on $\partial Y$ (and close to 1 in a neighborhood of $\partial Y$). Moreover, by comparison results for Hessians of distance functions in CAT spaces
$$
\begin{array}{ll}D_\Geo^2\,\d^2(x,z)\ge 2
, &\mbox{if }L = 0\\
D_\Geo^2 \cos\big(\sqrt{L} \d(x,z)\big)
\le
-L\,\cos\big(\sqrt{L} \d(x,z)\big)
, &\mbox{if }L>0\\
D_\Geo^2 \cosh\big(\sqrt{-L} \d(x,z)\big)
\ge
(-L)\,\cosh\big(\sqrt{-L} \d(x,z)\big)
, &\mbox{if }L<0.
\end{array}
$$
Thus
$D_\Geo^2 V\ge \cot_L(r)$ on  $\partial Y$(and $\ge \cot_L(r)-\epsilon$ in a neighborhood of $\partial Y$) which proves the claim.
\end{proof}

\begin{Prop}
Let $(X,\d)$ be an Alexandrov space with generalized sectional curvature $\ge K$.
Put $\rho_K=\frac\pi{2\sqrt{K}}$ if $K>0$ and $\rho_K=\infty$ else.
Assume that a closed set $Y\subset \X$ satisfies the ``exterior ball condition'' with radius $r<\rho_K$. That is,
$$X\setminus Y= \bigcup_{z\in Y_r} \B_r(z)$$
with $Y_r := \{z\in X: \d(z,Y)=r,\,  \rho(z)>r\}$.
Then $Y$ is locally
$\ell$-convex  with $\ell=-\cot_K r$.
%
%
\end{Prop}
%
%
%
%

\begin{proof} For $x\in X$, put $$V(x):=\sup_{z\in Y_r}V_{r,z}(x)$$ 
with $V_{r,z}$ as introduced in Lemma \ref{5.12}.
Then obviously $Y=\{V\le 0\}$ and 
$$|\nabla^-V|(x)=\frac{\sin_K \d(x,Y_r)}{\sin_K r}
\cdot\big|\nabla^+\d(.,Y_r)\big|(x)
$$
for all $x\in\X$. 
Hence,
$1-\epsilon\le 
|\nabla^- V|(x)\le1$
for all $x\in U\setminus Y$ for a suitable neighborhood $U$ of $\partial Y$.
Moreover,
for each $z \in Y_r$, by comparison results for Hessians of distance functions in Alexandrov spaces,
$$D_\Geo^2 V_{r,z}(x)\ge -\frac{\cos_K d(x,z)}{\sin_K r},$$ 
and therefore,
$$D_\Geo^2 V(x)\ge 
 -\frac{\cos_K r}{\sin_K r}-\epsilon$$
for all $x$ in a suitable neighborhood $U$ of $\partial Y$.
This  proves the claim.
\end{proof}

Analogously, we conclude

\begin{Prop}
Let $(X,\d)$ be a CAT space with generalized sectional curvature $\le L$.
Put $\rho_L=\frac\pi{2\sqrt{L}}$ if $L>0$ and $\rho_L=\infty$ else.
Assume that a closed set $Y\subset \X$ satisfies the ``reverse exterior ball condition'' with radius $r<\rho_L$. That is,
$$Y= \bigcap_{z\in Z} \overline\B_r(z)$$
for some compact set $Z\subset\X$.
Then $Y$ is locally $\ell$-convex with
$\ell:=\cot_L r$.
%
%
\end{Prop}

%
\begin{proof} Similar as in the proof of the previous Proposition, put $$V(x):=\sup_{z\in Z}\big(-V_{r,z}\big)(x)
$$
with 
$V_{r,z}$ defined as before, but now with $L$ in the place of $K$.
Then it is easily seen that $V\le 0$ on $Y$ and $V>0$ on $X\setminus Y$. Moreover, by comparison results for the Hessian of distance functions under upper curvature bounds, 
$$D_\Geo^2\, V(x)\ge\frac{\cos_L \big(\sup_{z\in Z}\d(x,z)\big)}{\sin_L r}\ge \cot_L (r+\epsilon)$$ for all $x\in \B_\epsilon(Y)$.
%
Furthermore, 
$$\frac{\sin_L (r-\epsilon)}{\sin_L r}\le |\nabla^-V|(x)\le \frac{\sin_L (r+\epsilon)}{\sin_L r}$$
 for all $x\in \B_\epsilon(\partial Y)$.
\end{proof}

\medskip

The Convexification Theorems \ref{1stConvThm} and \ref{2ndConvThm} from the previous subsection immediately yield

\begin{thm}\label{ConvThm2}
{\bf i)}
Assume that $\ell\in{\mathcal C}(X)$ is an exterior  lower bound for the curvature  of $\partial Y$. 
 Then for every $\epsilon>0$, the set $Y$ is locally  geodesically convex in $(X,\d')$ for $\d'=e^{(\epsilon-\ell)\cdot V}\odot \d$
 where $V=\d(\,.\,,Y)$.
 
 {\bf ii)}
Assume that $\ell\in{\mathcal C}(X)$ is an interior  lower bound for the curvature  of $\partial Y$. 
 Then for every $\epsilon>0$, the set $Y^0$ is locally  geodesically convex in $(X,\d')$ for $\d'=e^{(\epsilon-\ell)\cdot V}\odot \d$
 where $V=-\d(\,.\,,X\setminus Y)$.
\end{thm}

\begin{remark} The Convexification Theorems  \ref{1stConvThm} and \ref{2ndConvThm} provide a method 
to make a given set convex
by local changes of the  geometry. By construction of this transformed geometry, the given set will be ``as little convex as possible''. Indeed,  in regions where the set already was convex, the set will be transformed into a less convex set.
\end{remark}

\begin{figure}[h]
    \centering
   \includegraphics[scale=0.75]{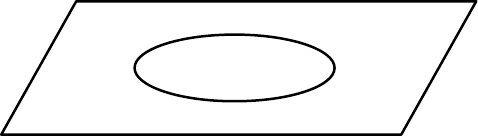}
    \caption{Euclidean plane: convex disc, nonconvex complement}
    \label{fig:convexify_1}
\end{figure}
\begin{figure}[h]
    \centering
   \includegraphics[scale=0.75]{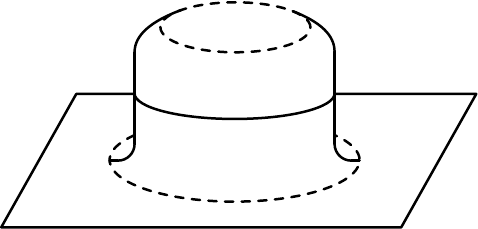}
    \caption{Conformally changed plane: totally geodesic circle}
    \label{fig:convexify_3}
\end{figure}
\begin{figure}[H]
    \centering
   \includegraphics[scale=0.75]{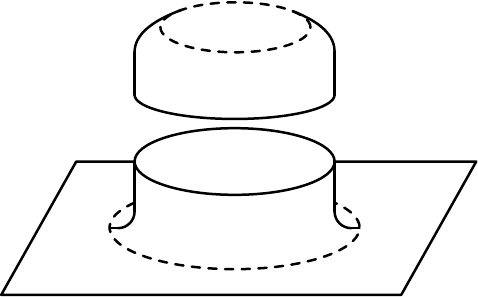}
    \caption{Conformally changed plane: decomposition into two convex subsets}
    \label{fig:convexify_5}
\end{figure}

\begin{ex} Let $X=\R^n$ for $n\ge 2$, equipped with the Euclidean distance and the Lebesgue measure. If we apply the previous results to the complement of a ball,  say $Y=\R^n\setminus \B_r(z)$, then we see that $Y$ (as well as $Y^0$) will be locally geodesically convex in $(\R^n,e^{(1+\epsilon)\psi}\odot \d)$ for any $\epsilon>0$
where
$$\psi=-\ell\cdot V=\frac12\Big(1-\frac{|x-z|^2}{r^2}\Big).$$
On the other hand, applying the previous results to a  ball, say $Z= \B_r(z)$, then we see that $Z$ (as well as $\overline Z$) will be locally geodesically convex in $(\R^n,e^{(1-\epsilon)\psi}\odot \d)$ for any $\epsilon>0$ with the same $\psi$ as before.
The same ``convexification effect'' will be achieved by choosing
$$\psi(x)=-\log\frac{|x-z|}{r}$$
in a neighourhood of $\partial \B_r(z)$. 
In the case $n=2$, with this choice of $\psi$, in a neighorhood of $\partial \B_r(z)$
the space $(\R^n,e^{\psi}\odot \d)$ will be a flat torus. In particular, $\partial \B_r(z)$ will be a totally geodesic subset. This provides a simple explanation why both, $B_r(z)$ and its complement, are convex.
\end{ex}



\section{Ricci Bounds for Neumann Laplacians}
\label{sec6}

\subsection{Neumann Laplacian and Time Change}

Let  a metric measure space $(X,\d,\mm)$ be given; assume that it is geodesic, locally compact, and infinitesimally Hilbertian. Observe that due to the local compactness, 
$W^{1,2}(X)=\big\{u\in W^{1,2}_\loc(X): \int_{X} [\Gamma(u)+u^2]\, d\mm<\infty\big\}$.

By \emph{restriction} to a  closed set $Y\subset X$,
we define the mm-space $(Y,\d_Y,\mm_Y)$. Here $\d_Y$ denotes the length metric on $Y$ induced by $\d$ and  $\mm_Y$ denotes the measure $\mm$ restricted to $Y$. 
 The Cheeger energy associated with  the  \emph{restricted mm-space} $(Y,\d_Y,\mm_Y)$ will  be denoted by $\E^Y$ and its  domain by $\F^Y=\Dom(\E^Y)=W^{1,2}(Y)$.
To avoid pathologies, throughout the sequel, we assume that  $Y=\overline{Y^0}$, $\mm(Y)>0$,   $\mm(\partial Y)=0$, and that $\d_Y<\infty$ on $Y\times Y$.

The  minimal weak upper gradients (and thus also the $\Gamma$-operators) w.r.t. $(X,\d,\mm)$ and w.r.t. $(Y,\d_Y,\mm_Y)$ will coincide a.e.\ on $Y^0$, i.e. 
$\F^Y_\loc(Y^0)=\F_\loc(Y^0)$ and $\Gamma^Y(u)=\Gamma(u)$ a.e.\ on $Y^0$ for all $u\in \F^Y_\loc(Y^0)$. Moreover, 
\begin{equation}\label{w,w0}
W^{1,2}(X)\big|_Y\subset 
W^{1,2}(Y)\subset W^{1,2}(Y^0)
\end{equation}
where $W^{1,2}(Y^0):=
\big\{u\in \F_\loc(Y^0): \int_{Y^0} [\Gamma(u)+u^2]\, d\mm<\infty\big\}$
and 
\begin{equation}\label{E0}\E^Y(u)=\int_{Y^0} \Gamma(u)\, d\mm
\end{equation}
for all $u\in W^{1,2}(Y)$. 
In particular, the restricted mm-space $(Y,\d_Y,\mm_Y)$ is also  infinitesimally Hilbertian.

The  heat  semigroup associated with the restricted mm-space $(Y,\d_Y,\mm_Y)$ will be called \emph{Neumann heat semigroup} and denoted by $(P^Y_t)_{t\ge0}$.
The associated Brownian motion 
will be called \emph{reflected Brownian motion} and denoted by $(\mathbb P^Y_x, B^Y_t)$.

\begin{remark} {\bf i)} In literature on Dirichlet forms and Markov processes (in particular, in \cite{chen2012}), Chapter 7, ``reflected Brownian motion'' on the closure of an open set $Y^0\subset X$ is by definition (and by construction)  the reversible Markov process associated with the Dirichlet form $\E^Y$ given by \eqref{E0}
with domain $W^{1,2}(Y^0)\subset L^2(Y,\mm_Y)$.

 {\bf ii)} In general, the sets $W^{1,2}(Y)$ and  $W^{1,2}(Y^0)$ do not coincide, see subsequent Example. In  \cite{lierl2018}, Section 4.2, equality of $W^{1,2}(Y)$ and  $W^{1,2}(Y^0)$ was erroneously stated as a general fact. Instead, it should be added there as an extra assumption.  
 Equality holds if $Y^0$ is regularly locally  semiconvex, see Proposition \ref{coincide} below, and
 of course also  if $Y^0$ has the extension property $W^{1,2}(Y^0)=W^{1,2}(X)\big|_{Y_0}$.
\end{remark}

\begin{ex}[Based on private communication with  T.~Rajala] Given $X=[-1,1]^2$ with Euclidean distance $\d$ and 2-dimensional Lebesgue measure $\mm$, put
$$Y=X\setminus \bigcup_{n\in\N} \B_{3^{-n}}(x_n,0)$$
where $\{x_n\}_{n\in\N}$ denotes a countable dense subset of $[-1,1]$.
Then $W^{1,2}(Y)\not=W^{1,2}(Y^0)$.
For instance, the function $u(x,y)={\mathrm{sign}}(y)$ belongs to $W^{1,2}(Y^0)$ but not to
 $W^{1,2}(Y)$.
 
 Indeed,  functions in $W^{1,2}(Y^0)$ can have arbitrary jumps at the $x$-axis since $Y^0$ has two disconnected components, one being a subset of the open upper half plane, the other one being a subset of the open lower half plane. On the other hand, functions in $W^{1,2}(Y)$ will be continuous along almost every vertical line which does not hit one of the small balls $\overline \B_{2^{-n}}(x_n,0)$, $n\in\N$, (which is the case for more than half of the vertical lines).
\end{ex}

\begin{Def}
 An open subset $Z\subset X$ is called \emph{regularly locally  semiconvex} if there exists an open 
neighborhood  $D$ of $\partial Z$ 
and functions $V,\ell\in\D_\loc^\cont(\Delta)$ such that $V$ is $\ell$-convex and $V=-\d(\,.\,,\partial Z)$  in $D\cap Z$.
\end{Def}
Here and in the sequel, we put  $\D_\loc^\cont(\Delta):=\{f\in \D_\loc(\Delta)$ with $f, \Gamma(f), \Delta f\in {\mathcal C}(X)\}$ and
$\D_\infty^\cont(\Delta):=\{f\in \D_\loc(\Delta)$ with $f, \Gamma(f), \Delta f\in {\mathcal C}(X)\cap L^\infty(X)\}$.
 
 Note that $\D_\infty^\cont(\Delta)\subset \Lip_b(X)$ provided the Sobolev-to-Lipschitz property holds.

\begin{Prop}\label{coincide} Assume that $(X,\d,\mm)$ satisfies \RCD$(K,N)$ for some  $K,N\in\R$ and that $Y^0$ is regularly locally semiconvex.
Then $$W^{1,2}(Y)=W^{1,2}(Y^0)$$
and $|D_Yu|=|Du|$ $\mm$-a.e.~on $Y$ for every $u\in W^{1,2}(Y^0)$.
\end{Prop}

\begin{proof}{\bf  i)} To simplify the subsequent presentation, we assume that 
the defining  functions $\ell, V$ for the regular semiconvexity of $Y^0$ can be chosen to be in $\D_\infty^\cont(\Delta)$ and not just in $\D_\loc^\cont(\Delta)$. Under this simplifying assumption, for any $\epsilon>0$
also  $\psi:=(\epsilon-\ell)\,V\in \D_\infty^\cont(\Delta)$   and thus
the time-changed mm-space $(X,\d',\mm')$ with $\d'=e^\psi\odot \d$ and $\mm'=e^{2\psi}\,\mm$ will satisfy \RCD$(K',\infty)$ with some  $K'\in\R$.
The general case can be treated by a localization and covering argument.

{\bf ii)} Recall that a function $u\in L^2(Y,\mm_Y)$ is in $W^{1,2}(Y)$ with weak upper gradient $g\in L^2(Y,\mm_Y)$ if and only if for each test plan $\Pi$ in $(Y,\d_Y,\mm_Y)$
$$\int_{\mathcal C}\big|u(\gamma_1)-u(\gamma_0)\big|\,d\Pi(\gamma)\le \int_{\mathcal C}\int_0^1 g(\gamma_t)\, |\dot\gamma_t|\,dt\,d\Pi(\gamma).$$
where ${\mathcal C}:={\mathcal C}([0,1], Y)$.
Now observe that any  test plan $\Pi$ in $(Y,\d_Y,\mm_Y)$ can also be regarded as a test plan in $(X,\d,\mm)$. (Indeed,  for each  curve $\gamma\in{\mathcal C}([0,1], Y)\subset {\mathcal C}([0,1], X)$, the speed w.r.t.~$(X,\d)$ will be bounded by the speed w.r.t.~$(Y,\d_Y)$.) And $\Pi$ is a test plan in $(X,\d,\mm)$ if and only if it is a test plan in $(X,\d',\mm')$.
Moreover, $g$ is a weak upper gradient w.r.t.~$(X,\d,\mm)$ implies that $g'=e^{-\psi}g$ is a weak upper gradient w.r.t.~$(X,\d',\mm')$ and vice versa since 
$$|\dot\gamma_t|'=e^{\psi(\gamma_t)}|\dot\gamma_t|.$$

{\bf  iii)} Now let us fix a  test plan $\Pi$ in $(Y,\d_Y,\mm_Y)$. For $n\in\N$, define  $\Pi'_n$ to be the ``piecewise geodesic test plan'' in $(X,\d',\mm')$
such that $\big((e_t)_*\Pi_n\big)_{t\in[0,1]}$
is the $W'_2$-geodesic which interpolates between the measures $(e_{i/n})_*\Pi$ for $i=0,1,\ldots,n$. Thanks to the \RCD-property of $(X,\d',\mm')$, such a piecewise geodesic interpolation indeed is a test plan.

 For each $t\in[0,1]$, we know that $\gamma_t\in Y^0$ for $\Pi$-a.e.~$\gamma$ since $(e_t)_*\Pi\le C\,\mm$ and $\mm(\partial Y)=0$. Geodesic convexity of $Y^0$ w.r.t.~$\d'$ thus implies
$$\gamma_t\in Y^0\qquad(\forall t\in[0,1])$$ for $\Pi_n$-a.e.~$\gamma$. In particular, thus for each $n\in\N$ and each $\epsilon>0$, there exists a compact set $Y_\epsilon\subset Y^0$ such that 
$\Pi_n({\mathcal C}_\epsilon)\ge 1-\epsilon$.
where ${\mathcal C}_\epsilon:={\mathcal C}([0,1], Y_\epsilon)$. Put
$$\Pi_n^\epsilon(\,.\,):=\frac1{\Pi_n({\mathcal C}_\epsilon)}\, \Pi_n(\,.\, \cap {\mathcal C}_\epsilon).$$

{\bf iv)}  Given the compact set $Y_\epsilon\subset Y^0$, there exists $u_\epsilon\in W^{1,2}(X,\d',\mm')$ such that
$$u=u_\epsilon, \ |D'u|=|D'u_\epsilon| \ \mm\text{-a.e.~on a neighborhood of }Y_\epsilon.$$
Since $\Pi_n^\epsilon$ is  a test plan in $(X,\d',\mm')$, we obtain for each $n\in\N$ and each $\epsilon>0$
\begin{eqnarray*}
\int_{\mathcal C}\int_0^1|Du|(\gamma_t)\,|\dot\gamma_t|\,dt\,d\Pi_n(\gamma)&=&
\int_{\mathcal C}\int_0^1|D'u|(\gamma_t)\,|\dot\gamma_t|'\,dt\,d\Pi_n(\gamma)\\
&\ge& Z_\epsilon
\int_{{\mathcal C}_\epsilon}\int_0^1|D'u_\epsilon|(\gamma_t)\,|\dot\gamma_t|'\,dt\,d\Pi_n^\epsilon(\gamma)\\
&\ge& Z_\epsilon
\int_{{\mathcal C}_\epsilon}\big|u_\epsilon(\gamma_1)-u_\epsilon(\gamma_0)\big|\,d\Pi_n^\epsilon(\gamma)\\
&=& 
\int_{{\mathcal C}_\epsilon}\big|u(\gamma_1)-u(\gamma_0)\big|\,d\Pi_n(\gamma).
\end{eqnarray*}
In the limit $\epsilon\to0$ this yields
\begin{eqnarray*}
\int_{\mathcal C}\int_0^1|Du|(\gamma_t)\,|\dot\gamma_t|\,dt\,d\Pi_n(\gamma)\ge
\int_{{\mathcal C}}\big|u(\gamma_1)-u(\gamma_0)\big|\,d\Pi_n(\gamma)=
\int_{{\mathcal C}}\big|u(\gamma_1)-u(\gamma_0)\big|\,d\Pi(\gamma).
\end{eqnarray*}
Since by assumption $|Du|\in L^2(Y,\mm_Y)$, according to the subsequent Lemma we may pass to the limit $n\to\infty$ and finally obtain
\begin{eqnarray*}
\int_{\mathcal C}\int_0^1|Du|(\gamma_t)\,|\dot\gamma_t|\,dt\,d\Pi(\gamma)\ge
\int_{{\mathcal C}}\big|u(\gamma_1)-u(\gamma_0)\big|\,d\Pi(\gamma).
\end{eqnarray*}
This proves the claim: $u\in W^{1,2}(Y)$ with minimal weak upper gradient $\mm$-a.e.~dominated by $|Du|$.
\end{proof}

\begin{lma}[Private communication by N. Gigli]
Assume that $(X,\d,\mm)$  satisfies ${\sf RCD}(K,N)$ for some $K,N\in\R$. Then for every test plan $\Pi$ in $X$ and every $g\in L^2(X,\mm)$
\begin{equation}
\label{piecew}
\lim_{n\to\infty}
\iint_0^1 g(\gamma_t)\, |\dot\gamma_t|\,dt\,d\Pi_n(\gamma)=\iint_0^1 g(\gamma_t)\, |\dot\gamma_t|\,dt\,d\Pi(\gamma)
\end{equation}
where $\Pi_n$ denotes the ``piecewise geodesic test plan'' such that $\big((e_t)_*\Pi_n\big)_{t\in[0,1]}$
is the $W_2$-geodesic which interpolates between the measures $(e_{i/n})_*\Pi$ for $i=0,1,\ldots,n$.
\end{lma}

\begin{proof} First of all, observe that it obviously suffices to prove the claim for test plans supported on bounded subsets of $X$. Secondly observe, that it suffices to consider bounded continuous functions $g$. Indeed, given any $g\in L^2(X,\mm)$ and $\varepsilon>0$, there exists $g_\varepsilon\in{\mathcal C}_b(X)$ with $\|g-g_\varepsilon\|_{L^2}\le\varepsilon$. Since $\Pi_n$ is a test plan, this implies 
\begin{eqnarray*}
\Big|\iint_0^1 [g(\gamma_t)-g_\eps(\gamma_t)]\, |\dot\gamma_t| dt d\Pi_n(\gamma)\Big|^2&\le&
{\iint_0^1|g(\gamma_t)-g_\eps(\gamma_t) |\, dt d\Pi_n(\gamma)}  \cdot {\iint_0^1 |\dot\gamma_t|^2 dt d\Pi_n(\gamma)}\\
&\le&
\eps \cdot \sup_n C_n\cdot \sup_nA_n
\end{eqnarray*}
 for each $n\in\N\cup\{\infty\}$ with $\Pi_\infty:=\Pi$ 
where $C_n$ is the compression of the test plan $\Pi_n$ and 
$$A_n:=\iint_0^1 |\dot\gamma_t|^2 dt d\Pi_n(\gamma)\le \iint_0^1 |\dot\gamma_t|^2 dt d\Pi(\gamma)<\infty.$$
Due to the \RCD$(K,\infty)$-assumption, the compression of $\Pi_n$ is bounded by the compression of $\Pi$ times a constant depending on $K$ and the diameter of the supporting set of $\Pi$. Thus $$\iint_0^1 [g(\gamma_t)-g_\eps(\gamma_t)]\, |\dot\gamma_t| dt d\Pi_n(\gamma)\to0$$ uniformly in $n\in\N\cup\{\infty\}$ as $\eps\to0$.

It remains to prove \eqref{piecew} for bounded continuous $g$. This will be an immediate consequence of the weak convergence 
\begin{equation}\label{weco}d{\boldsymbol \pi}_n(\gamma):= |\dot\gamma_t| dt d\Pi_n(\gamma)
\ \to \ |\dot\gamma_t| dt d\Pi(\gamma)=:d{\boldsymbol \pi}(\gamma)
\end{equation}
as measures on the space ${\boldsymbol X}:=[0,1]\times {\mathcal C}([0,1]\to X)$. To prove the latter, 
we first observe  that the total mass of the measures ${\boldsymbol \pi}_n$ is uniformly bounded on ${\boldsymbol X}$ since
\begin{equation}\label{action}
\Big(\int d{\boldsymbol \pi}_n\Big)^2\le \iint |\dot\gamma_t|^2 dt d\Pi_n(\gamma)\le \iint |\dot\gamma_t|^2 dt d\Pi(\gamma)<\infty.
\end{equation}
 Properness of $X$ (due to the \RCD$(K,N)$-assumption) and uniform boundedness of the supporting sets of $\Pi_n$ then guarantees the existence of a subsequential limit ${\boldsymbol \pi}_\infty$. Lower semicontinuity of the map $\gamma\mapsto |\dot\gamma_t|$ implies that ${\boldsymbol \pi}_\infty\le {\boldsymbol \pi}$. Now assume that ${\boldsymbol \pi}_\infty\not= {\boldsymbol \pi}$. Then in particular ${\boldsymbol \pi}_\infty\not= {\boldsymbol \pi}$ on the set $\{(t,\gamma): |\dot\gamma_t|\not=0\}$. Once again using the lower semicontinuity of $\gamma\mapsto |\dot\gamma_t|$ this will imply
 $$\liminf_{k\to\infty}\iint |\dot\gamma_t|^2 dt d\Pi_{n_k}(\gamma)=\int |\dot\gamma_\cdot|\,d{\boldsymbol \pi}_\infty>
 \int |\dot\gamma_\cdot|\,d{\boldsymbol \pi}=\iint |\dot\gamma_t|^2 dt d\Pi(\gamma)$$
 which is a contraction to \eqref{action} from above. Thus ${\boldsymbol \pi}_\infty= {\boldsymbol \pi}$ and hence \eqref{weco} follows and so does in turn \eqref{piecew}.
\end{proof}

\bigskip

Given a bounded continuous $\psi\in W^{1,2}(X)$, let us consider the mm-space $(Y,\d'_Y,\mm'_Y)$ with $\mm'_Y=e^{2\psi}\, (\mm|_Y)=(e^{2\psi}\mm)|_Y
$ and $\d'_Y=(\d_Y)'=(\d')_Y$. On the level of the mm-spaces, it is clear that \emph{restriction and time-change commute}. Hence, the time change of the reflected Brownian motion is equivalent to the reflected motion of the time changed process; and
 the time change of the Neumann heat flow is the Neumann version of the time changed heat flow.
 
 \bigskip
 
 Now let us have a closer look on the transformation of the curvature-dimension condition under time change and restriction.

 \begin{Prop}
 Assume that $(X,\d,\mm)$ satisfies the \RCD$(k,N)$-condition for some finite number $N\ge2$ and some lower bounded,  continuous function $k$. Moreover, assume that 
 $Y$ is locally geodesically convex in $(X,\d')$ where $\d'=e^{\psi}\odot \d$ for some  $\psi\in\Lip_b(X)\cap\Dom^\cont_{loc}(\Delta)$.
Then  for any (extended) number $N'>N$, the mm-space $(Y,\d_Y',\mm_Y')$ satisfies the \RCD$(k',N')$-condition and the \BE$_2(k',N')$-condition with
  $$k':=
e^{-2\psi}[k-\Delta\psi -\frac{(N-2)(N'-2)}{N'-N}  |\nabla \psi|^2].$$
 \end{Prop}
 
 \begin{proof} {\bf i)} To get started, we first employ the equivalence of the Lagrangian and Eulerian formulation of curvature-dimension conditions as formulated in Theorem \ref{cd=be} to conclude that $(X,\d,\mm)$ satisfies the \BE$_2(k,N)$-condition.
 
 {\bf ii)} Next we apply our result on time change, Proposition \ref{time-change} or \cite{han2019}, Theorem 1.1, to conclude that $(X,\d',\mm')$ satisfies the \BE$_2(k',N')$-condition with the given 
 $N'$ and $k'$.
 
 {\bf iii)} Once again referring to Theorem \ref{cd=be}  for the equivalence of the Lagrangian and Eulerian formulation, we conlucde that $(X,\d',\mm')$ satisfies the \RCD$(k',N')$-condition.
 
  {\bf iv)} In the Lagrangian formulation, it is obvious that a curvature-dimension condition is preserved under restriction to locally geodesically convex subsets. Since by assumption $Y$ is locally geodesically convex in $(X,\d')$, it 
  follows that $(Y,\d_Y',\mm_Y')$ satisfies the \RCD$(k',N')$-condition.
  
   {\bf v)} In a final step, we once again employ Theorem \ref{cd=be} to conclude  \BE$_2(k',N')$, the Eulerian version of the curvature-dimension condition.
 \end{proof}

\subsection{Time Re-Change}

We are now going to make a ``time re-change'': we transform the mm-space $(Y,d_Y',m_Y')$ into the mm-space $(X,d_Y,m_Y)$ by time change, now with $-\psi$  in the place of $\psi$  and with $(k',N')$ in the place of $(k,N)$.

The main challenge will arise from the two conflicting requirements:
\begin{itemize}
\item
$|\nabla\psi|\not=0$ on
$\partial Y$ in order to make use of  the Convexification Theorem 
\item  $\psi\in\Dom(\Delta^Y)$ (which essentially requires $|\nabla\psi|=0$ on
$\partial Y$) in order to control the Ricci curvature under the ``time re-change''.
\end{itemize}

To overcome this conflict, we have developed the concept of $W^{-1,\infty}$-valued Ricci bounds  which will allow us to work with the distribution $\underline\Delta^Y\psi$. More precisely, the crucial ingredient in our estimate will be the distribution
$\underline\Delta^Y\psi\big|_{\partial Y}:=\underline\Delta^Y\psi-\Delta\psi\,\mm|_Y\in W^{-1,\infty}$ defined as
\begin{equation}\langle \underline\Delta^Y\psi\big|_{\partial Y},f\rangle
:=-\int_Y[\Gamma(\psi,f)+\Delta\psi\,f]\,d\mm\qquad(\forall f\in W^{1,1+}(Y)).
\end{equation}
Note that this distribution indeed is supported on the boundary of $Y$ in the sense that $\psi=\psi'$ on a neighborhood of $\partial Y$  implies
\begin{equation}\label{locality of k}\underline\Delta^Y\psi\big|_{\partial Y}=\underline\Delta^Y\psi'\big|_{\partial Y}.
\end{equation}

\begin{thm}\label{k-for-rechange} Assume that $(X,\d,\mm)$ satisfies the \RCD$(k,N)$-condition for some finite number $N\ge2$ and some lower bounded, continuous function $k$. Moreover, assume that 
 $Y$ is locally geodesically convex in $(X,\d')$ where $\d'=e^{\psi}\odot \d$ for some  $\psi\in\Dom_\infty^\cont(\Delta)$ with $\psi=0$ on $\partial Y$. Then the mm-space $(Y,\d_Y,\mm_Y)$ satisfies the \BE$_1(\kappa,\infty)$-condition with 
\begin{equation}\kappa=k \,\mm_Y
+\underline\Delta^Y\psi\big|_{\partial Y}.
\end{equation}
\end{thm}

\begin{proof} {\bf i)} \ 
Let  $\psi\in\Dom^\cont_\infty(\Delta)$, put 
 $\d'=e^{\psi}\odot \d$ and $N'=2(N-1)$. Then according to Theorem \ref{time-change},
the mm-space
 $(X,\d',\mm')$ satisfies the \BE$_2(k',N')$-condition with 
 $$k'=
e^{-2\psi}[k-\Delta\psi -2(N-2)  |\nabla \psi|^2].$$ 
 Since by assumption $Y$ is locally geodesically convex, according to the previous Proposition, the mm-space
 $(Y,\d_Y',\mm_Y')$ also satisfies the \BE$_2(k',N')$-condition with the same $k'$.
 
 On the space $Y$, let us now perform a time change with the weight function $-\psi$ (``time re-change'') to get back
 $$\d_Y=e^{-\psi}\odot \d'_Y, \ \mm_Y=e^{-2\psi}\odot \mm'_Y.$$
 According to Theorem \ref{sing-time-change}, the mm-space $(X,\d_Y,\mm_Y)$ will satisfy $\BE_1(\kappa,\infty)$ with
 \begin{eqnarray}\label{kn-formel}
\kappa&=&\Big[k'-(N'-2)|\nabla'\psi|^2\Big]\mm_Y'+\underline\Delta'_Y\psi\nonumber\\
&=&\Big[k-\Delta\psi-4(N-2)|\nabla\psi|^2\Big]\mm_Y+\underline\Delta_Y\psi\nonumber\\
&=&\Big[k-4(N-2)|\nabla\psi|^2\Big]\mm_Y+\underline\Delta_Y\psi\big|_{\partial Y}.
\label{k_n}
\end{eqnarray}

\medskip
{\bf ii)} \ 
In a final approximation step, we now want to get rid of the term $-4(N-2)  |\nabla \psi|^2\,\mm_Y$ in the  previous distributional Ricci bound $\kappa$.

Given $\psi$ as  above, we define a sequence of functions $\psi_n$ with the same properties as $\psi$, with $\psi_n=\psi$ on $\B_{1/n}(\partial Y)$, with $|\nabla\psi_n|$ being  bounded, uniformly in $n$, and with
$$|\nabla\psi_n|\to0\quad\mm\text{-a.e. on }Y$$ as $n\to\infty$. 
This can easily be achieved by means of the truncation functions from Lemma 4.4.

Then according to \eqref{kappa,n} in the previous part of this proof, for each $n\in \N$ the mm-space $(Y,\d_Y,\mm_Y)$ satisfies the \BE$_1(\kappa_n,\infty)$-condition with 
\begin{eqnarray*}\kappa_n&=&
[k -4(N-2)  |\nabla \psi_n|^2]\,\mm_Y+\underline\Delta^Y\psi_n\big|_{\partial Y}\\
&=&
[k -4(N-2)  |\nabla \psi_n|^2]\,\mm_Y+\underline\Delta^Y\psi\big|_{\partial Y}
\end{eqnarray*}
where the last equality is due to \eqref{locality of k}.
Since the mm-space under consideration does not depend on $n$, this obviously implies the 
 \BE$_1(\kappa,\infty)$-condition with 
$\kappa=
k \,\mm_Y+\underline\Delta^Y\psi\big|_{\partial Y}$.
\end{proof}

\paragraph{Summary.}
Let us illustrate the strategy of the argumentation for the proof of the previous Theorem \ref{k-for-rechange} in a diagram.

\bigskip

$$
\begin{array}{ccccc}
{\color{blue}(X,\d,\mm)}&\
&\CD(k,N) & \quad \Rightarrow \quad& \BE_2(k,N)\\

\color{red}\downarrow&&&&\\
 \mbox{\color{red} time change, convexification}&&&& \Downarrow \quad\\
\color{red}\downarrow&&&&\\

{\color{blue}(X,\d',\mm')}&&
\CD(k',N') &\quad \Leftarrow \quad& \BE_2(k',N')\\

\color{red}\downarrow&&&&\\
\mbox{\color{red} restriction to convex subset}&& \Downarrow&&  \\
\color{red}\downarrow&&&&\\

{\color{blue}(Y,\d'_Y,\mm'_Y)}&&\CD(k',N') & \quad \Rightarrow \quad& \BE_2(k',N')\\

\color{red}\downarrow&&&&\\
\mbox{\color{red} time re-change}&&&& \Downarrow\\
\color{red}\downarrow&&&&\\
{\color{blue}(Y,\d_Y,\mm_Y)}&&& & \BE_1(\kappa,\infty) \\
\end{array}
$$

\bigskip

where
$m'=e^{2\psi} m,\quad
\d'(x,y)=(e^\psi\odot \d)(x,y):=
\inf_{\gamma_0=x,\gamma_1=y}\int_0^1e^{\psi(\gamma_s)}|\dot\gamma_s|ds$. 

\bigskip

\begin{cor}
Under the assumptions of the previous Theorem, the (``Neumann'') heat semigroup on $(Y,\d_Y,\mm_Y)$ satisfies a gradient estimate of the  type:
\begin{equation}
\big|\nabla P_t^Yf\big|(x)\le \EE_x\Big[e^{-\frac12\int_0^{2t} k(B^Y_s)ds+N^{\partial Y,\psi}_{2t}}\cdot\big|\nabla f(B^Y_{2t})\big|\,\Big]
\end{equation}
with
\begin{eqnarray}\label{partial N}N^{\partial Y,\psi}_{t}&:=&N^{Y,\psi}_{t}-\frac12\int_0^t\Delta\psi(B^Y_s)ds\nonumber\\
&=&\psi(B^Y_t)-\psi(B^Y_0)-M^{Y,\psi}_{t}-\frac12\int_0^t\Delta\psi(B^Y_s)ds
\end{eqnarray}
where $M^{Y,\psi}$ and  $N^{Y,\psi}$ denote the local martingale and local additive functional of vanishing quadratic variation w.r.t.~$(\PP^Y_x,B^Y_t)$ in the Fukushima decomposition of $\psi(B^Y_t)$.
\end{cor}

\begin{ex} Consider a time-change of the standard 2-dimensional  metric measure space $(\R^2,\d_{\sf Euc},\mm_{\sf Leb})$ induced by  a function  $\psi:\R^2\to\R$ where $\psi(x_1,x_2)=\varphi(x_1)\cdot \eta(x_2)$
 for some  $\varphi, \eta\in {\mathcal C}^2(\R)$  with $\eta(0)=0$ and $\eta'(0)=1$.
Recall that the time-changed mm-space $(\R^2,\d',\mm')$ with $d':=e^{\psi}\odot \d_{\sf Euc},
\mm':=e^{2\psi}\,\mm_{\sf Leb}$ satisfies $\BE_1(k,\infty)$ with 
$$k=-e^{-2\psi}\Delta\psi=-e^{-2\varphi\,\eta}\big(\varphi''\,\eta+\varphi\, \eta''\big).$$
 Now consider the restriction to the upper halfplane $Y=\R\times\R_+$ which is convex w.r.t.~$\d_{\sf Euc}$ but higly non-convex  w.r.t.~$\d'$. According to Theorem \ref{k-for-rechange}  (applied to  $\d'$ and  $\d_{\sf Euc}$ in the place of  of $\d$ and $\d'$, resp., and with $\psi$ replaced by $-\psi$), the boundary effect amounts in an additional contribution in the Ricci bound given by $$-\underline\Delta_\infty^Y\psi\big|_{\partial Y}=-\underline\Delta^Y\psi\big|_{\partial Y}
 =-(\varphi\,{\mathcal L}^1)\otimes\delta_0.$$
 Indeed, the distribution in turn can be identified with the signed measure since for sufficiently smooth $f:\R^2\to\R$,
 \begin{eqnarray*}
 \langle -\underline\Delta^Y\psi\big|_{\partial Y},f\rangle&=&
 \int_Y\big[\nabla\psi+f\,\Delta\psi\Big]d\mm\\
 &=&\int_\R \Big[\varphi\int_{\R_+}\big[\eta'\,f'+f\,\eta''\big]dx_1\Big]\,dx_2
 =-\int_\R \Big[\varphi(x_1)\eta'(0)f(x_1,0)\Big]\,dx_2.
 \end{eqnarray*}
\end{ex}

\subsection{Boundary Measure and Boundary Local Time}

Let $V:X\to\R$ denote the \emph{signed distance function} from $\partial Y$ (being positive outside $Y$ and negative in the interior of $Y$), i.e.,
$$V:=\d(.,Y)-\d(.,X\setminus Y).$$
Then $V^+:=\d(.,Y)$ and $V^-:=\d(.,X\setminus Y)$. 

We say that $Y$ has the \emph{$W^{1,1+}$-extension property} if $W^{1,1+}(Y)=W^{1,1+}(X)\big|_Y$, that is, if every function $u\in W^{1,1+}(Y)$ can be extended to a function $u'\in  W^{1,1+}(X)$ such that $u'|_Y=u$.
We say that $Y$ has \emph{regular boundary} if it has the $W^{1,1+}$-extension property and  if $V\in \D_\infty^\cont(\Delta)$.

\begin{lma}\label{bdy-measure} Assume that $Y$ has regular boundary. 

{\bf i)} Then the distribution $-\underline\Delta^Y V\big|_{\partial Y}$ is given by a nonnegative measure $\sigma$ supported on $\partial Y$,
denoted henceforth by $\sigma_{\partial Y}$ and 
called \emph{surface measure of $\partial Y$}.

More precisely, there exists a  nonnegative Borel measure $\sigma$ on $X$ which is supported on $\partial Y$ and which does not charge sets of vanishing capacity   such that for all bounded, quasi-continuous $f\in\Dom(\E)$
\begin{eqnarray*}\int_{\partial Y}  f\,d\sigma&=&-\big\langle \underline\Delta^Y V\big|_{\partial Y}, f\big\rangle:=
\int_Y\big[\Gamma(V,f)+\Delta V\,f\big]\,d\mm\\
&=&-
\int_{X\setminus Y}\big[\Gamma(V,f)+\Delta V\,f\big]\,d\mm.
\end{eqnarray*}

{\bf ii)}  The local additive functional of vanishing quadratic variation $N^{\partial Y,-V}$ as defined in \eqref{partial N} (with $-V$ in the place of $\psi$) coincides with the PCAF (= ``positive continuous additive functional'') associated to $\sigma_{\partial Y}$ via Revuz correspondence  (w.r.t.~the Brownian motion $(\PP^Y_x,B^Y_t)$ on $Y$) which henceforth  will be denoted by $L^{\partial Y}=(L^{\partial Y}_t)_{t\ge0}$
and called \emph{local time of $\partial Y$}.
In other words,
$$L^{\partial Y}_t=V(B^Y_0)-V(B^Y_t)+\frac12\int_0^t\Delta V(B^Y_s)ds+\text{local martingale}.$$
\end{lma}

\begin{proof} {\bf i)}
The equality $\int_Y\big[\Gamma(V,f)+\Delta V\,f\big]\,d\mm
=-
\int_{X\setminus Y}\big[\Gamma(V,f)+\Delta V\,f\big]\,d\mm$ obviously holds for all $V\in\D(\Delta)$ and $f\in \D(\E)$. 
On the open set $X\setminus Y$, locality of $\Delta$ implies
\begin{eqnarray*}
-\int_{X\setminus Y}\Gamma(V,f)\,d\mm
&=&
-\int_{\{V>0\}}\Gamma( V^+,f)\,d\mm\\
&=&\lim_{t\to0}\frac1t\int_{\{V>0\}}\big(P_tV^+-V^+\big)f\,d\mm\\ 
&\ge&\lim_{t\to0}\frac1t\int_{\{V>0\}}\big(P_tV-V\big)f\,d\mm\\ 
&=&\int_{\{V>0\}}\Delta V\, f\,d\mm
\end{eqnarray*}
for nonnegative $f\in\Dom(\E)$. In other words,
$\int_Y\big[\Gamma(V,f)+\Delta V\,f\big]\,d\mm\ge0$.
This extends to all nonnegative $f\in W^{1,1+}(X)$ if $\Gamma(V), \Delta V\in L^\infty(X)$. Moreover, due to the extension property which we assumed, it extends to all nonnegative $f\in W^{1,1+}(Y)$.
Thus 
$$-\big\langle \underline\Delta^Y V\big|_{\partial Y}, f\big\rangle\ge0$$
for all nonnegative $f\in W^{1,1+}(Y)$. According to the Riesz-Markov-Kakutani Representation theorem, the distribution $-\underline\Delta^Y V\big|_{\partial Y}$ therefore is given by a Borel measure on $X$, say $\sigma$. Obviously, this measure is supported by $\partial Y$.

Moreover, on each set $X'\subset X$ of finite volume, this measure $\sigma$ has finite energy:
$$\Big|\int f\,d\sigma\Big|=\Big|\big\langle \underline\Delta^Y V\big|_{\partial Y}, f\big\rangle\Big|\le
C\cdot\|f\|_{W^{1,1+}}\le C\cdot \mm(X')^{1/2}\cdot \E(f)^{1/2}$$
for all $f\in\Dom(\E)$ which are supported in $X'$. Thus $\sigma$ does not charge sets of vanishing capacity, \cite{fukushima2011} , Lemma 2.2.3.

{\bf ii)} The fact that $-\underline\Delta^Y V\big|_{\partial Y}$ is a nonnegative measure (of finite energy) implies that $\underline\Delta^Y V$ is a signed measure (of finite energy). Hence, $\underline\Delta^Y V$ and $N^{Y,V}$ are related to each other via Revuz correspondence. And of course the signed measure $\Delta V\, \mm_Y$ corresponds to the additive functional
$(\int_0^t \Delta V(B_s^Y)ds)_{t\ge0}$.
\end{proof}

\begin{lma}\label{IbP} 
Assume that the
``integration-by-parts formula'' holds true for $Y$ with some measure $\sigma$ on $\partial Y$ (charging no sets of vanishing capacity): $\forall f\in\Dom(\Delta), g\in\Dom(\E)$
\begin{equation}
\int_Y \Gamma(f,g)\,d\mm+\int_Y \Delta f\,g\,d\mm=\int_{\partial Y}\tilde\Gamma(f,V)\tilde g\,d\sigma\end{equation}
with $\tilde g$ and $\tilde\Gamma(f,V)$ denoting the quasi continuous versions of $g$ and $\Gamma(f,V)$, resp.
Then $\sigma=\sigma_{\partial Y}.$ 
\end{lma}
Note that for $ f\in\Dom(\Delta^Y)$, the above formula -- with vanishing RHS -- is trivial.

\begin{proof}
Applying the Integration-by-Parts formula to $f=V$ yields
\begin{equation*}
\int_Y \Gamma(V,g)\,d\mm+\int_Y \Delta V\,g\,d\mm=\int_{\partial Y}\tilde g\,d\sigma\end{equation*}
which proves that the distribution $-\underline\Delta^Y V\big|_{\partial Y}$
is represented by the measure $\sigma$
and thus $\sigma=\sigma_{\partial Y}$
\end{proof}

\begin{ex} Let $X$ be a $n$-dimensional Riemannian manifold, $\d$ be the Riemannian distance, $\mm$ be the $n$-dimensional Riemannian volume measure, and $Y$ be a bounded  subset with ${\mathcal C}^1$-smooth boundary. Then $\sigma_{\partial Y}$  is the $(n-1)$-dimensional  surface measure of $\partial Y$.
\end{ex}

\begin{lma} Assume that $\psi=\ell\, V$ with $V$ as in Lemma \ref{bdy-measure} above and $\ell\in \Dom^\cont_\infty(\Delta)$.

{\bf i)} Then
$$-\underline\Delta^Y \psi\big|_{\partial Y}=\ell\,\sigma_{\partial Y}.$$

{\bf ii)} Moreover, with   $(N^{\partial Y,\psi}_t)_{t\ge0}$ and  $(N^{\partial V,\psi}_t)_{t\ge0}$ defined as in \eqref{partial N},
$$N^{\partial Y,\psi}_t=\int_0^t\ell(B_s^Y)\,dN^{\partial Y,V}_s=-\int_0^t\ell(B_s^Y)\,dL^{\partial Y}_s
.$$
\end{lma}

\begin{proof} {\bf i)} For each quasi continuous $f\in W^{1,1+}(X)\cap W^{1,2}(X)$
\begin{eqnarray*}
-\big\langle \underline\Delta^Y\psi\big|_{\partial Y},f\big\rangle
&=&\int_Y\big[\Gamma(\ell V, f)+\Delta(\ell V)\, f\big]d\mm\\
&=&\int_Y\big[\Gamma(V, \ell f)+\Delta V\, \ell f\big]d\mm+
\int_Y\big[\Gamma(\ell,Vf)+\Delta\ell \, Vf\big]\,d\mm\\
&=&-\big\langle \underline\Delta^Y V\big|_{\partial Y},\ell f\big\rangle
-
\int_X\big[\Gamma(\ell,V^-f)+\Delta\ell \, V^-f\big]\,d\mm\\
&=&\int_{\partial Y}\ell f\,d\sigma_{\partial Y}+0
\end{eqnarray*}
where for the last step we also have chosen $\ell$ to be quasi continuous.

{\bf ii)} Fukushima decomposition w.r.t.~Brownian motion on $Y$ and Leibniz rule for stochastic integrals applied to   $\psi=\ell\, V$ yield
\begin{eqnarray*}
d\psi(B^Y_t)&=&dN^{\partial Y,\psi}_t+\frac12\Delta\psi(B^Y_t)dt+\text{loc.~mart.}\\
&=&dN^{\partial Y,\psi}_t+\frac12(\ell\Delta V)(B^Y_t)dt+\frac12(V\Delta \ell)(B^Y_t)dt
+\Gamma(\ell,V)(B^Y_t)dt
+\text{loc.~mart.}
\end{eqnarray*}
as well as
\begin{eqnarray*}
d\psi(B^Y_t)&=&\ell(B^Y_t) dV(B^Y_t)+V(B^Y_t) d\ell(B^Y_t)+\Gamma(\ell,V)(B^Y_t)dt+\text{loc.~mart.}
\end{eqnarray*}
Taking into account that 
$$\frac12(V\Delta \ell)(B^Y_t)dt=V d\ell(B^Y_t)+\text{loc.~mart.}$$
since $V=0$ on $\partial Y$, we end up with
$$dN^{\partial Y,\psi}_t=\ell(B^Y_t) dV(B^Y_t)-\frac12(\ell\Delta V)(B^Y_t)dt=-\ell(B^Y_t) dL_t^{\partial Y}.$$
This is the claim.
\end{proof}

\bigskip

Recall from Definition \ref{def-curv}
 that a function
$\ell: X\to\R$ is a   lower bound for the curvature  of $\partial Y$ if for each $\epsilon>0$ there exists an exterior neighborhood $D$ of $\partial Y$ such that 
 $V$ is  $(\ell-\epsilon)$-convex in $D$.

\begin{thm}
Assume that $(X,\d,\mm)$ satisfies the \RCD$(k,N)$-condition for some finite number $N\ge2$ and some lower bounded, continuous function $k$. Moreover, assume that 
$Y$ has a regular boundary and that
  $\ell\in\D_\infty^\cont(\Delta)$  is a   lower bound for the curvature  of $\partial Y$. Then the mm-space $(Y,\d_Y,\mm_Y)$ satisfies the \BE$_1(\kappa,\infty)$-condition with 
$$\kappa=k \,\mm_Y+\ell\,\sigma_{\partial Y}.$$
\end{thm}

\begin{proof} According to Theorem \ref{ConvThm2}, for each $\epsilon>0$, we can apply the previous Theorem \ref{k-for-rechange} with $\psi:=(\epsilon-\ell)\, V$.
In this case, obviously $\psi\in\D_\infty^\cont(\Delta)$.
Thus $(Y,\d_Y,\mm_Y)$ satisfies the \BE$_1(\kappa,\infty)$-condition with 
$$\kappa=k \,\mm_Y
+\underline\Delta^Y\big((\epsilon-\ell)\, V\big)\big|_{\partial Y}.$$
Since this holds for every $\epsilon>0$, it follows that 
$(Y,\d_Y,\mm_Y)$ satisfies the \BE$_1(\kappa,\infty)$-condition with 
$$\kappa=k \,\mm_Y
-\underline\Delta^Y(\ell\, V)\big|_{\partial Y}.$$
According to the previous Lemma, the distribution $-\underline\Delta^Y(\ell\, V)\big|_{\partial Y}$ is given by the weighted measure $\ell\,\sigma_{\partial Y}$. This proves the claim. 
\end{proof}

\begin{cor} Under the assumptions of the previous Theorem, the heat semigroup on $(Y,\d_Y,\mm_Y)$ satisfies the following gradient estimate:
\begin{eqnarray}\label{final-grad}
\big| \nabla P_{t/2}^Yf\big|(x)&\le
&\mathbb E^Y_{x}\Big[ e^{-\frac12\int_0^{t}  k(B^Y_{s})ds-\frac12\int_0^{t} \ell(B^Y_s)dL^{\partial Y}_s}
\cdot \big|\nabla f(B^Y_{t})\big|\Big].
\end{eqnarray}
Recall that $(L^{\partial Y}_t)_{t\ge0}$, the local time of $\partial Y$, is defined via Revuz correspondence  w.r.t.~the (``reflected'') Brownian motion $(\PP_x^Y, B_t^Y)_{t,x}$ on $Y$ as the PCAF associated with the surface measue $\sigma_{\partial Y}$.
\end{cor}

\begin{remark} {\bf i)} The first estimate of the above type \eqref{final-grad} has been derived in the setting of smooth Riemannian manifolds by 
E.~P.~Hsu \cite{hsu2002} in terms of Brownian motions and their local times. For more recent results of this type in the setting of (weighted) Riemannian manifolds, see \cite{wang2014}, e.g.~Thm.~3.3.1.

{\bf ii)}
In the setting of smooth Riemannian manifolds, inspired by the integration by parts formula,
B.~Han \cite{han2018} was the first to propose the definition of a measure-valued Ricci tensor 
which involves
 the second fundamental form integrated with respect to the boundary measure.
 
 {\bf iii)} For the sake of clarity of presentation we have restricted ourselves in this subsection to the choice $V=\pm(\,.\,,\partial Y)$. However, instead of that, one may choose any sufficiently regular function $V$ which coincides with the signed distance function in a neighborhood of the boundary.
 More generally, the Convexification Theorem allows us to choose any function $V$ with 
 $|\nabla V|(x)\to1$ 
 for $x\to\partial Y$.
 
 {\bf iv)} Note that the previous Theorem and Corollary require that the underlying space $(X,\d,\mm)$ satisfies the \RCD$(k,N)$-condition for some finite $N$ and that our proof strongly depends on finiteness of $N$. However, the value of $N$ does not enter the final estimates.
\end{remark}

Let us finally illustrate our results in the two prime examples, the ball and the complement of the ball. To simplify the presentation, we will formulate the results in the setting of $\RCD(0,N)$ spaces for $N\in\N$ with the CAT(1)-property (or  $\RCD(-1,N)$ spaces with the CAT(0)-property). The extension to $\RCD(K,N)$ spaces with sectional curvature bounded from above by $K'$ is straightforward. 

\begin{ex} \label{ball}

{\bf i)} Consider $Y=\overline\B_r(z)$ for some $z\in X$ and $r\in (0,\pi/4)$
where  $(X,\d,\mm)$ is an $N$-dimensional Alexandrov space with nonnegative Ricci curvature and sectional curvature bounded from above by $1$ (in particular, $\mm={\cal H}^N$). Then
\begin{eqnarray*}
\big| \nabla P_{t/2}^Yf\big|(x)&\le
&\mathbb E^Y_{x}\Big[ e^{-\frac{\cot r}{2} \, L^{\partial Y}_t}
\cdot \big|\nabla f(B^Y_{t})\big|\Big].
\end{eqnarray*}
In particular, $\Lip( P_{t/2}^Yf)\,/\, \Lip(f)\le \sup_x\mathbb E^Y_{x}\big[ e^{-\frac{\cot r}{2} \, L^{\partial Y}_t}
\big]$ and 
\begin{eqnarray}\label{ball-grad}
\frac{\big| \nabla P_{t/2}^Yf\big|^2(x)}{P_{t/2}^Y\big| \nabla f\big|^2(x)}\le  \mathbb E^Y_{x}\big[ e^{-\cot r\cdot L^{\partial Y}_t}
\big]\le e^{-t\frac{N-1}{2}\cot^2r+1}.
\end{eqnarray}

{\bf ii)} Consider  $Y=X\setminus \B_r(z)$ for some $z\in X$ and $r\in (0,\infty)$ where
$(X,\d,\mm)$ is a $N$-dimensional Alexandrov space  with $N\ge 3$, with  Ricci curvature bounded from below by $-1$, and with nonpositive sectional curvature (in particular, $\mm={\cal H}^N$).
Then
\begin{eqnarray*}
\big| \nabla P_{t/2}^Yf\big|(x)&\le
&\mathbb E^Y_{x}\Big[ e^{t/2+\frac1{2r}L^{\partial Y}_t}
\cdot \big|\nabla f(B^Y_{t})\big|\Big].
\end{eqnarray*}
In particular, $\Lip( P_{t/2}^Yf)\,/\, \Lip(f)\le \sup_x\mathbb E^Y_{x}\big[ e^{t/2+\frac1{2r}L^{\partial Y}_t}
\big]$ and 
\begin{eqnarray}\label{c-ball-grad}
\frac{\big| \nabla P_{t/2}^Yf\big|^2(x)}{P_{t/2}^Y\big| \nabla f\big|^2(x)}\le  \mathbb E^Y_{x}\big[ e^{t+\frac1{r}L^{\partial Y}_t}
\big]\le e^{Ct+C'\sqrt{t}}.
\end{eqnarray}
\end{ex}
Let us emphasize that in the latter setting, no estimate of the form 
\begin{eqnarray*}
\frac{\big| \nabla P_{t/2}^Yf\big|^2(x)}{P_{t/2}^Y\big| \nabla f\big|^2(x)}\le e^{Ct}.
\end{eqnarray*}
 can exist.

\begin{proof} {\bf i)} It remains to prove the second inequality in \eqref{ball-grad}. Put
$$V(x)=\frac1{\sin r}\big(\cos r-\cos\d(x,z)\big).$$
Then $V=0$ and $|\nabla V|=1$ on $\partial Y$. Thus
$$V(B^Y_t)=V(B^Y_0)+M^{Y,V}_t+\frac12\int_0^t\Delta V(B^Y_s)ds-L^{\partial Y}_t$$
where $M^{Y,V}$ is  a martingale with quadratic variation $\langle M^{Y,V}\rangle_t=\int_0^t\Gamma(V)(B^Y_s)ds\le t$. Note that
$|V(B^Y_t)-V(B^Y_0)|\le \frac1{\sin r}(1-\cos r)\le r$ and, by Laplace comparison, $\frac12\int_0^t\Delta V(B^Y_s)ds\ge \frac{Nt}2 \,\cot r$.
Therefore
$$e^{-\cot r\, L^{\partial Y}_t}\le e^{\cot r\, M^{Y,V}_t-\frac{Nt}{2r}\cot^2 r+1}$$
and hence
$$\EE_xe^{-\cot r\, L^{\partial Y}_t}\le e^{\frac{t}{2}\cot^2 r-\frac{Nt}{2}\cot^2 r+1}\cdot \EE_xe^{-\cot r\, M^{Y,V}_t-\frac{\cot^2r}{2}\langle M^{Y,V}\rangle_t}\le
 e^{-\frac{(N-1)t}{2}\cot^2 r+1}.$$
 
{\bf ii)} 
Put
$$V(x)=\frac{r^{N-1}}{N-2}\big(\d^{2-N}(x,z)-r^{2-N}\big).$$
Then $V=0$ and $|\nabla V|=1$ on $\partial Y$. Moreover,
$|V(B^Y_t)-V(B^Y_0)|\le \frac r{N-2}$ and $\Delta V\le0$ by Laplace comparison. Thus
$$L^{\partial Y}_t\le V(B^Y_0)-V(B^Y_t)+M^{Y,V}_t$$
with $\langle M^{Y,V}\rangle_t\le t$ and, therefore,
\begin{eqnarray*}\EE_xe^{\frac1rL^{\partial Y}_t}&\le&e^{\frac1{r^2}t}\cdot
\Big[\EE_xe^{\frac2r M^{Y,V}_t-\frac2{r^2}\langle M^{Y,V}\rangle_t}\Big]^{1/2}
\cdot \Big[\EE_xe^{\frac2r(V(B^Y_0)-V(B^Y_t))}\Big]^{1/2}\\
 &\le& e^{\frac1{r^2}t}\cdot e^{\frac er\EE_x\big[V(B^Y_0)-V(B^Y_t)\big]}.
\end{eqnarray*}
where the last inequality follows from the fact that $|\frac2r(V(B^Y_0)-V(B^Y_t))|\le 1$.
To estimate $\EE_x\big[V(B^Y_0)-V(B^Y_t)\big]$, 
we apply the Laplace comparison to the function 
$$V^2(y)=\frac{r^{2N-2}}{(N-2)^2}\cdot \Big(\d^{2-N}(y,z)-r^{2-N}\Big)^2$$
which yields
$$\Delta V^2(y)\le 2+\frac{N}{N-2}\, \d(y,z)\, \coth \d(y,z)\le 5.$$
Therefore, taking into account that  $V(x)\le0$,
\begin{eqnarray*}
\EE_x\Big[V(B^Y_0)-V(B^Y_t)\Big]&\le&
 V(x)+\EE_x\Big[V^2(B^Y_t)\Big]^{1/2}\\
 &=&
 V(x)+\EE_x\Big[V^2(x)+\frac12\int_0^t\Delta V^2(B^Y_s)ds\Big]^{1/2}\\
 &\le&
 \EE_x\Big[\frac12\int_0^t\big(\Delta V^2\big)_+(B^Y_s)ds\Big]^{1/2}\ \le \ \sqrt{\frac52 \, t}.
\end{eqnarray*}
%
Thus
$\EE_xe^{\frac1rL^{\partial Y}_t}\le e^{C t +C'\sqrt{t}}.$
\end{proof}

\begin{cor} In the setting of the previous Example \ref{ball}  i), the effect of the boundary curvature results in a lower bound for the spectral gap:
$$\lambda_1\ge \frac{N-1}{2}\cot^2r.$$

\end{cor}
Let us emphasize that without taking into account the curvature of the boundary, no positive lower bound for $\lambda_1$ will be available.
\begin{proof} In the gradient estimate for the heat flow on the ball $Y=\overline\B_r(z)$, the boundary curvature causes an exponential decay:
\begin{eqnarray*}
{\big| \nabla P_{t}^Yf\big|^2}(x)\le e^{-t(N-1)\cot^2r+1}\, P_{t}^Y\big| \nabla f\big|^2(x)
\end{eqnarray*}
for each $f$ and $x\in Y$, and $P_{t}^Y\big| \nabla f\big|^2(x)\to\frac1{\mm(Y)}\, \int_Y\big| \nabla f\big|^2\,\mm$ as $t\to\infty$. 
On the other hand, by spectral calculus
$$\big| \nabla P_{t}^Yf_1\big|^2(x)=e^{-2\lambda_1}\, \big| \nabla f_1\big|^2(x).$$
for the eigenfunction $f_1$ corresponding to the first non-zero eigenvalue $\lambda_1$.
\end{proof}

\end{document}